\documentclass[11pt]{amsart}

\pdfoutput=1

\usepackage{etex,amsfonts, amsmath, amsthm, amssymb, epsfig, graphics,
  psfrag, latexsym, mathtools, mathrsfs,enumitem, subfig,
  longtable, booktabs, yfonts, centernot,ifthen,eso-pic,longtable,stmaryrd}
\usepackage[clock,weather]{ifsym} 
\usepackage[normalem]{ulem}
\usepackage{xr-hyper}
\usepackage[bbgreekl]{mathbbol}
\DeclareSymbolFontAlphabet{\mathbb}{AMSb}
\DeclareSymbolFontAlphabet{\mathbbl}{bbold}
\usepackage{xcolor}
\usepackage[all,cmtip]{xy}
\usepackage{tikz} \usetikzlibrary{matrix,arrows,shapes,calc,braids,arrows.meta,positioning}
\pgfdeclarelayer{background}
\pgfdeclarelayer{foreground}
\pgfsetlayers{background,main,foreground}
\tikzset{zlevel/.style={%
    execute at begin scope={\pgfonlayer{#1}},
    execute at end scope={\endpgfonlayer}
  }}
\tikzset{khdiff/.style={%
    black!30,opacity=1,->
  }}
\tikzset{knot/.style={%
    black,thick,preaction={draw,white,line width=3pt}    
  }}
\tikzset{resol/.style={%
    black,thick
  }}
\tikzset{linelabel/.style={%
    outer sep=0,inner sep=1pt,fill=white,fill opacity=1,draw opacity=1
  }}
\tikzset{preservemap/.style={%
    ->,black,opacity=1
  }}
\tikzset{dropmap/.style={%
    ->,black!30,opacity=1
  }}

\definecolor{darkblue}{rgb}{0,0,0.4} 
 \usepackage[colorlinks=true, citecolor=darkblue, filecolor=darkblue, linkcolor=darkblue,urlcolor=darkblue]{hyperref} 
\usepackage[all]{hypcap}
\usepackage{xr-hyper}

\graphicspath{{draws/}{}}

\numberwithin{equation}{section}

\newtheorem{lem}{Lemma}[section]               

\newtheorem{theorem}[lem]{Theorem}
\newtheorem{lemma}[lem]{Lemma}

\newtheorem{corollary}[lem]{Corollary}               

\newtheorem{proposition}[lem]{Proposition}

\theoremstyle{definition}     

\newtheorem{question}{Question}
 
\newtheorem{definition}[lem]{Definition}

\theoremstyle{remark}

\newtheorem{remark}[lem]{Remark}

\numberwithin{figure}{section}
\numberwithin{table}{section}

\newcommand{\R}{\mathbb{R}}

\newcommand{\QQ}{\mathbb{Q}}
\newcommand{\FF}{\mathbb{F}}

\newcommand{\smbullet}{{\vcenter{\hbox{\tikz{\fill[black] (0,0) circle (0.05);}}}}}

\newcommand{\wh}{\widehat}
\newcommand{\wt}{\widetilde}

\newcommand{\del}{\partial}

\renewcommand{\emptyset}{\varnothing}
\newcommand{\interior}{\mathring}

\newcommand{\from}{\colon}

\newcommand{\too}{\longrightarrow}

\newcommand{\lra}{\longrightarrow}

\renewcommand{\th}{^{\text{th}}}

\newcommand{\SpinC}{\text{Spin}^{\text{C}}}

\newcommand{\HF}{\mathit{HF}}

\DeclareMathOperator{\nbd}{nbd}

\DeclareMathOperator{\Id}{Id}

\DeclareMathOperator{\Hom}{Hom}

\newcommand{\Kh}{\mathit{Kh}}

\newcommand{\gr}{\mathrm{gr}}

\newcommand{\BNcx}{\mathcal{C}}
\newcommand{\BNh}{\mathcal{H}}
\newcommand{\red}{\mathrm{red}}

\newcommand{\Ring}{\mathrm{R}}

\newcommand{\co}{\colon}
\newcommand{\bdy}{\partial}
\newcommand{\RR}{\R}

\newcommand{\DD}{\mathbb{D}}

\newcommand{\ZZ}{\mathbb{Z}}
\DeclareMathOperator{\cokernel}{coker}

\newcommand{\RP}{\mathbb{R}\mathrm{P}}

\newcommand{\mathcenter}[1]{\vcenter{\hbox{$#1$}}}

\newcommand*{\defeq}{\mathrel{\vcenter{\baselineskip0.5ex \lineskiplimit0pt
                     \hbox{\scriptsize.}\hbox{\scriptsize.}}}%
                     =}

\newcommand{\lab}[1]{$\scriptstyle #1$}

\newcommand{\MixedInvt}[1]{\Phi_{#1}}

\begin{document}


\title{A mixed invariant of nonorientable surfaces in equivariant Khovanov homology}

\author{Robert Lipshitz}
\thanks{\texttt{RL was supported by NSF Grant DMS-1810893.}}
\email{\href{mailto:lipshitz@uoregon.edu}{lipshitz@uoregon.edu}}
\address{Department of Mathematics, University of Oregon, Eugene, OR 97403}

\author{Sucharit Sarkar}
\thanks{\texttt{SS was supported by NSF Grant DMS-1905717.}}
\email{\href{mailto:sucharit@math.ucla.edu}{sucharit@math.ucla.edu}}
\address{Department of Mathematics, University of California, Los Angeles, CA 90095}



\date{\today}

\begin{abstract}
  We construct a mixed invariant of nonorientable surfaces from the
  Lee and Bar-Natan deformations of Khovanov homology and use it to
  distinguish pairs of surfaces bounded by the same knot, including
  some exotic examples.
\end{abstract}
\maketitle
\vspace{-1cm}


\tableofcontents


\section{Introduction}\label{sec:intro}
Ozsv\'ath-Szab\'o's Heegaard Floer homology associates
$\ZZ[U]$-modules $\HF^-(Y,\mathfrak{s})$ and $\HF^+(Y,\mathfrak{s})$
to a closed, connected, oriented $3$-manifold $Y$ and $\SpinC$-structure
$\mathfrak{s}$, and $\ZZ[U]$-module homomorphisms
$F^\pm(W,\mathfrak{t})\co
\HF^\pm(Y_1,\mathfrak{s}_1)\to\HF^\pm(Y_2,\mathfrak{s}_2)$ to a
$\SpinC$-cobordism $(W,\mathfrak{t})\co (Y_1,\mathfrak{s}_1)\to
(Y_2,\mathfrak{s}_2)$ \cite{OSz-hf-3manifolds,OSz-hf-4manifolds}. For
a closed $4$-manifold $W$ with $b_2^+>0$, viewed as a cobordism from
$S^3$ to itself by deleting two balls, the maps
$\HF^\pm(W,\mathfrak{t})$ vanish. Using the proof of vanishing,
Ozsv\'ath-Szab\'o define a Heegaard Floer-theoretic analogue of the
Seiberg-Witten invariant, the \emph{Heegaard Floer mixed invariant},
which is a map $\HF^-(Y_1,\mathfrak{s}_1)\to\HF^+(Y_2,\mathfrak{s}_2)$
associated to a $\SpinC$-cobordism with $b_2^+$ sufficiently large.

The goal of this paper is to give an analogous construction in
Khovanov homology, for smoothly embedded surfaces in $[0,1]\times S^3$
with crosscap number at least $3$.

Following Finashin-Kreck-Viro~\cite{FKV-top-bulletin}, call a pair of smoothly embedded
surfaces $F,F'\subset [0,1]\times S^3$ \emph{exotic} if there is a
self-homeomorphism of $[0,1]\times S^3$ which is the identity on
$\{0,1\}\times S^3$ and takes $F$ to $F'$, but there is no such
self-diffeomorphism of $[0,1]\times S^3$. (See also
Lemma~\ref{lem:MWW} and Section~\ref{sec:exotic}.) At the time of writing, we believe no
pairs of exotic closed, orientable surfaces in $[0,1]\times S^3$ are
known. There are, however, exotic nonorientable surfaces in
$[0,1]\times S^3$, as first
shown by Finashin-Kreck-Viro using results from Donaldson
theory~\cite{FKV-top-knot-surf}. Recently, examples of exotic
orientable cobordisms in $[0,1]\times S^3$ have also appeared, in work
of Juh\'asz-Miller-Zemke~\cite{MJZ-hf-exotic},
Hayden~\cite{Hay-kh-disks}, and
Hayden-Kjuchukova-Krishna-Miller-Powell-Sunukjian~\cite{HKKMPS-kh-Brun},
which use Heegaard Floer homology to distinguish them, and
Hayden-Sundberg~\cite{HS-kh-exotic}, which uses Khovanov
homology.

Many question about exotic pairs of surfaces remain open. For example,
Baykur-Sunukjian introduced stabilization operations for surfaces, and
showed that all known examples of exotic pairs of closed surfaces
become diffeomorphic after a single
stabilization~\cite{BS-top-stabilizations}; it is not known if this
holds in general. Building on these ideas, one can study the
\emph{total stabilization distance} between two surfaces $F$, $F'$,
the minimum number of stabilizations or destabilizations needed to
turn one into the other~\cite{Miy-86-stab,MP-top-stab}, or the
\emph{max stabilization distance}, the minimum over sequences
$F=F_0,F_1,\dots,F_n=F'$, where $F_i$ and $F_{i+1}$ are related by a
stabilization, destabilization, or taking the connected sum with a
knotted $2$-sphere, of $\max\{|g(F_1)-g(F)|,\dots,|g(F_n)-g(F)|\}$
(where $g$ denotes the genus)~\cite{JZ-hf-stab}. (See
also~\cite[p. 6]{Melvin-top-thesis}.) Another notion is the
\emph{generalized total stabilization distance}, which is defined the
same way as the total stabilization distance except that if two
surfaces differ by taking the connected sum with a $2$-sphere then
they are declared to be at distance $0$, so the generalized total
stabilization distance, like the max stabilization distance, focuses
on global, rather than local, knotting~\cite{MP-top-stab}. By using
the Alexander module, Miyazaki shows there are pairs of embedded
spheres in $S^4$ with arbitrarily high total stabilization
distance~\cite{Miy-86-stab}, and Miller-Powell show there are pairs of
embedded disks with arbitrarily high generalized total stabilization
distance~\cite{MP-top-stab}. Juh\'asz-Zemke use Heegaard Floer
homology to give pairs of disks with max stabilization distance at
least $3$~\cite{JZ-hf-stab}. The analogous questions for exotic pairs
are open.

A key strategy for distinguishing knotted closed surfaces in $S^4$ has
been to apply gauge theory, like the Heegaard Floer mixed invariant,
to their branched double covers.
Ozsv\'ath-Szab\'o showed that the Heegaard Floer homology of the
branched double cover of a knot $K$ is closely related to the Khovanov
homology of $K$~\cite{OSz-hf-branched}. So, it seems natural to look for an analogue of
the Heegaard Floer mixed invariant in Khovanov homology

In this paper, we give one such analogue. Khovanov homology admits a
family of deformations~\cite{Kho-kh-Frobenius}; we will focus on two
particular ones, the Lee deformation~\cite{Lee-kh-endomorphism} and
the Bar-Natan deformation~\cite{Bar-kh-tangle-cob}. Rasmussen showed
that the map of Lee homologies associated to a nonorientable
cobordism vanishes~\cite{Ras-kh-slice}. Viewing these deformations as
modules over polynomial algebras, analogous to the Heegaard Floer
invariant $\HF^-$, Rasmussen's result says that the map on the
analogue of $\HF^\infty$ associated to a nonorientable cobordism
vanishes. Using this, and a notion of admissible cuts analogous to
Heegaard Floer theory, we formulate a Khovanov mixed invariant
$\MixedInvt{F}$ of a surface $F$ with crosscap number $\geq 3$ in the
Lee and Bar-Natan deformations of Khovanov homology. Note that $F$
having crosscap number $\geq 3$ has no implication for $b_2^+$,
so the Khovanov mixed invariant is defined
for some surfaces $F$ where the Heegaard Floer mixed invariant of
$\Sigma(F)$ is not.

Verifying that the mixed invariant is well-defined (up to sign) has two steps:
observing that the map on (deformed) Khovanov homology associated to a
nonorientable cobordism is well-defined (up to sign), and verifying
independence of the choice of admissible cut. The proof of the first
statement is a straightforward extension of the literature~\cite{Jac-kh-cobordisms,Kho-kh-cobordism,Bar-kh-tangle-cob,MWW-kh-blobs}; see
Section~\ref{sec:behave}.
(Unlike the orientable case, the sign
ambiguity here is essential; see Remark~\ref{rem:Klein-TQFT}.)
To prove independence of the admissible cut, we use arguments
involving the one-sided curve complex of a nonorientable surface; see
Section~\ref{sec:cuts}.

It turns out that, unlike the Heegaard Floer mixed invariant, this
Khovanov mixed invariant does not distinguish closed, connected
surfaces (Section~\ref{sec:closed}), though the proof of this fact is somewhat
intricate. (We do not know if the mixed invariant distinguishes some
closed, disconnected surfaces, and in particular have not generalized
Gujral-Levine's results~\cite{LG-kh-split} to this setting.) On the
other hand, both the mixed invariant and the map on Khovanov homology
associated to a nonorientable cobordism do distinguish pairs of
nonorientable surfaces with common boundary. Indeed, this was essentially
already shown in computations of Sundberg-Swann~\cite{SS-kh-surf}:
combined with the functoriality result mentioned above, their
computations show the following.

\begin{theorem}\label{thm:intro}
  There is a pair of connected surfaces $F,F'$ with boundary on $3_1\# m(3_1)$
  with crosscap number $3$ and normal Euler number $-6$ which are not
  isotopic, and do not become isotopic after taking the connected sum
  with any knotted $2$-sphere. Further, $F$ is not obtained
  from a connected surface $F''$ by attaching a $1$-handle, or by taking the
  connected sum with a standard $\RP^2$ or $\overline{\RP}^2$.
\end{theorem}

Theorem~\ref{thm:intro} is proved in Section~\ref{sec:SunSwann}. The second half of the theorem uses the behavior of $\MixedInvt{F}$
under various local modifications to the surface, which are summarized
in Theorem~\ref{thm:vanishing}.

Hayden-Sundberg's examples of exotic pairs of slice disks
distinguished by Khovanov homology can be enhanced to give exotic
pairs of nonorientable surfaces distinguished by Khovanov
homology. In particular, we have:
\begin{theorem}\label{thm:intro-2}
  There is an exotic pair of surfaces with boundary $12^n_{309}$, crosscap
  number 3, and normal Euler number $-6$.
\end{theorem}
Theorem~\ref{thm:intro-2} is proved in Section~\ref{sec:exotic}. As far as we know, this is the first gauge theory-free proof that
there are pairs of exotic nonorientable surfaces.

This paper is organized as follows. We review the Lee and Bar-Natan
deformations of Khovanov homology in Section~\ref{sec:BN-background},
in an algebraic framework parallel to Heegaard Floer
homology. Section~\ref{sec:behave} shows that these deformed Khovanov
complexes are functorial with respect to nonorientable
cobordisms. For convenience later, we also allow our cobordisms to be
decorated with stars (following the notation
of~\cite{KhRo-kh-Frob2}). Section~\ref{sec:cuts} formulates the notion
of admissible cuts, and shows that for surfaces with crosscap number
$\geq 3$ all admissible cuts are equivalent in a suitable
sense. Section~\ref{sec:mixed} defines the mixed invariant and proves
it is well-defined. Section~\ref{sec:properties} gives basic
properties of the maps associated to nonorientable cobordisms and the
mixed invariant, and Section~\ref{sec:comps} gives some computations
and applications of these invariants, including proving
Theorems~\ref{thm:intro} and~\ref{thm:intro-2}, and concludes with some questions.

\subsection*{Acknowledgments} We thank Ian Zemke for helpful
comments on the first draft of this paper, and the referee for further
comments and corrections.

\section{Background on the Lee and Bar-Natan deformations}\label{sec:BN-background}

Khovanov homology has two well-known deformations, the Lee
deformation~\cite{Lee-kh-endomorphism} and the Bar-Natan
deformation~\cite{Bar-kh-tangle-cob}. The two theories are similar,
although they have some essential differences as well. Most of the
constructions and results of this paper work for either of the two
theories, so we will use the same notations for both the
theories. When the two theories diverge, we will explicitly mention
that in the text.

Fix a commutative ring $\Ring$ with unit. All chain complexes and
modules will be defined over $\Ring$, though we will often suppress
$\Ring$ from the notation. If we are using the Lee theory, we assume
$2$ is a unit in $\Ring$.

Fix an oriented link diagram $L$ with $N$ crossings, $N_+$ of which
are positive and $N_-$ of which are negative.  Consider the Kauffman
cube of resolutions of $L$. A \emph{Khovanov generator} $y$ is a choice of
vertex $v$ and a labeling $y(Z)\in\{1,X\}$ of
each circle $Z$ in the $v$-resolution. Denoting
the homological, quantum bigrading by $(\gr_h,\gr_q)$, a
Khovanov generator $y$ lying over a vertex $v\in\{0,1\}^N$ has
bigrading
\begin{align*}
  \gr_h(y)&=-N_-+|v|\\
  \gr_q(y)&=N_+-2N_-+|v|+\#\{Z\mid y(Z)=1\}-\#\{Z\mid y(Z)=X\}.
\end{align*}

The deformed Khovanov complex $\BNcx^-(L)$ is freely generated by these
generators over a polynomial algebra over $\Ring$, and is obtained by
feeding the cube of resolutions into a Frobenius algebra over that
polynomial algebra.
\begin{enumerate}
\item In the Lee theory, the polynomial algebra is $\Ring[T]$ with $T$
  in bigrading $(0,-4)$, and the Frobenius algebra is
  $\Ring[T,X]/(X^2=T)$, with co-multiplication given by
  \begin{align*}
    \Delta(1)&=1\otimes X + X\otimes 1 &  \Delta(X)&=X\otimes X + T 1\otimes 1
  \end{align*}
  and counit $\epsilon\co \Ring[T,X]/(X^2=T)\to \Ring[T]$ given by
  $\epsilon(1)=0$, $\epsilon(X)=1$.
\item In the Bar-Natan theory, the polynomial algebra is $\Ring[H]$
  with $H$ in bigrading $(0,-2)$, and the Frobenius algebra is
  $\Ring[H,X]/(X^2=XH)$, with co-multiplication given by
  \begin{align*}
    \Delta(1)&=1\otimes X + X\otimes 1 - H 1\otimes 1 &  \Delta(X)&=X\otimes X
  \end{align*}
  and counit $\epsilon\co \Ring[H,X]/(X^2=XH)\to \Ring[H]$ given by
  $\epsilon(1)=0$, $\epsilon(X)=1$.
\end{enumerate}
(In Khovanov's paper~\cite{Kho-kh-Frobenius}, these are denoted
$\mathcal{F}_7$ and $\mathcal{F}_3$, respectively.)
In either theory, the differential increases the bigrading by $(1,0)$.
To keep the notation the same, let $\Ring[U]$ denote the polynomial
algebra for either theory. That is, when discussing the Lee theory we
take $U=T$ (in bigrading $(0,-4)$), and when discussing the Bar-Natan
theory we take $U=H$ (in bigrading $(0,-2)$). In either case, the
original non-deformed Khovanov complex is obtained by setting $U=0$;
in analogy with Heegaard Floer homology, we will denote the
non-deformed complex $\wh\BNcx(L)=\BNcx^-(L)/\{U=0\}$. The homology
$\wh\BNh(L)$ of $\wh\BNcx(L)$ is ordinary Khovanov homology, often
denoted $\Kh(L)$.

Continuing the analogy with Heegaard Floer homology, let 
\begin{align*}
  \BNcx^\infty(L)&=U^{-1}\BNcx^-(L)\\
  \BNcx^+(L)&=\BNcx^\infty(L)/\BNcx^-(L),
\end{align*}
where the notation $U^{-1}$ denotes localization or, equivalently,
tensoring over $\Ring[U]$ with $\Ring[U^{-1},U]$.
(These conventions are not exactly parallel to Heegaard Floer
homology~\cite[Section 4.1]{OSz-hf-3manifolds}.)  Let $\BNh^-(L)$, $\BNh^\infty(L)$, and $\BNh^+(L)$
be the homologies of $\BNcx^-(L)$, $\BNcx^\infty(L)$, and
$\BNcx^+(L)$. See Figure~\ref{fig:compute} for an
example of $\BNcx^\infty(L)$ and its subcomplex $\BNcx^-(L)$ and
quotient complex $\BNcx^+(L)$ (up to quasi-isomorphism), and a
comparison with the usual formulation of the Lee deformation.

There are short exact sequences
\begin{equation}\label{eq:minfp}
  \begin{gathered}
   \begin{tikzpicture}[xscale=.9, >=To] 
      \node at (.65,0) (tl0) {$0$};
      \node at (2,0) (Cm) {$\BNcx^-(L)$};
      \node at (4,0) (Ci) {$\BNcx^\infty(L)$};
      \node at (6,0) (Cpt) {$\BNcx^+(L)$};
      \node at (7.35,0) (tr0) {$0\phantom{.}$};
      \draw[->] (tl0) to (Cm);
      \draw[->] (Cm) to node[above=-.05]{\lab{\iota}} (Ci);
      \draw[->] (Ci) to node[above=-.05]{\lab{\pi}} (Cpt);
      \draw[->] (Cpt) to (tr0);
    \end{tikzpicture}\\[-.75em]
   \begin{tikzpicture}[xscale=.9, >=To] 
      \node at (.65,0) (tl0) {$0$};
      \node at (2,0) (Cm) {$\BNcx^-(L)$};
      \node at (4,0) (Ci) {$\BNcx^-(L)$};
      \node at (6,0) (Cpt) {$\wh{\BNcx}(L)$};
      \node at (7.35,0) (tr0) {$0\phantom{.}$};
      \draw[->] (tl0) to (Cm);
      \draw[->] (Cm) to node[above=-.05]{\lab{U}} (Ci);
      \draw[->] (Ci) to node[above=-.05]{\lab{\pi}} (Cpt);
      \draw[->] (Cpt) to (tr0);
    \end{tikzpicture}\\[-.75em]
   \begin{tikzpicture}[xscale=.9, >=To] 
      \node at (.65,0) (tl0) {$0$};
      \node at (2,0) (Cm) {$\wh{\BNcx}(L)$};
      \node at (4,0) (Ci) {$\BNcx^+(L)$};
      \node at (6,0) (Cpt) {$\BNcx^+(L)$};
      \node at (7.35,0) (tr0) {$0.$};
      \draw[->] (tl0) to (Cm);
      \draw[->] (Cm) to node[above=-.05]{\lab{\iota}} (Ci);
      \draw[->] (Ci) to node[above=-.05]{\lab{U}} (Cpt);
      \draw[->] (Cpt) to (tr0);
    \end{tikzpicture}
  \end{gathered}
\end{equation}
and corresponding long exact sequences
\begin{equation}\label{eq:les}
  \begin{gathered}
   \begin{tikzpicture}[xscale=1, >=To] 
      \node at (0.5,0) (n1) {$\cdots$};
      \node at (2,0) (n2) {$\BNh^-(L)$};
      \node at (4,0) (n3) {$\BNh^\infty(L)$};
      \node at (6,0) (n4) {$\BNh^+(L)$};
      \node at (8,0) (n5) {$\BNh^-(L)$};
      \node at (9.55,0) (n6) {$\cdots,$};
      \draw[->] (n1) to  (n2);
      \draw[->] (n2) to node[above=-.05]{\lab{\iota_*}} (n3);
      \draw[->] (n3) to node[above=-.05]{\lab{\pi_*}} (n4);
      \draw[->] (n4) to  node[above=-.05]{\lab{\bdy}} (n5);
      \draw[->] (n5) to (n6);
    \end{tikzpicture}\\[-.75em]
   \begin{tikzpicture}[xscale=1, >=To] 
      \node at (0.5,0) (n1) {$\cdots$};
      \node at (2,0) (n2) {$\BNh^-(L)$};
      \node at (4,0) (n3) {$\BNh^-(L)$};
      \node at (6,0) (n4) {$\wh{\BNh}(L)$};
      \node at (8,0) (n5) {$\BNh^-(L)$};
      \node at (9.55,0) (n6) {$\cdots,$};
      \draw[->] (n1) to  (n2);
      \draw[->] (n2) to node[above=-.05]{\lab{U}} (n3);
      \draw[->] (n3) to node[above=-.05]{\lab{\pi_*}} (n4);
      \draw[->] (n4) to  node[above=-.05]{\lab{\bdy}} (n5);
      \draw[->] (n5) to (n6);
    \end{tikzpicture}\\[-.75em]
   \begin{tikzpicture}[xscale=1, >=To] 
      \node at (0.5,0) (n1) {$\cdots$};
      \node at (2,0) (n2) {$\wh{\BNh}(L)$};
      \node at (4,0) (n3) {$\BNh^+(L)$};
      \node at (6,0) (n4) {$\BNh^+(L)$};
      \node at (8,0) (n5) {$\wh{\BNh}(L)$};
      \node at (9.55,0) (n6) {$\cdots.$};
      \draw[->] (n1) to  (n2);
      \draw[->] (n2) to node[above=-.05]{\lab{\iota_*}} (n3);
      \draw[->] (n3) to node[above=-.05]{\lab{U}} (n4);
      \draw[->] (n4) to  node[above=-.05]{\lab{\bdy}} (n5);
      \draw[->] (n5) to (n6);
    \end{tikzpicture}\\[-.75em]
  \end{gathered}
\end{equation}
The homomorphisms $U$ decrease the bigrading by $(0,2)$ for the
Bar-Natan deformation and $(0,4)$ for the Lee deformation, the
homomorphism $\wh\BNh(L)\to\BNh^+(L)$ increases bigrading by $(0,2)$
for the Bar-Natan deformation and $(0,4)$ for the Lee deformation, the
connecting homomorphisms $\BNh^+(L)\to \BNh^-(L)$ and
$\BNh^+(L)\to \wh\BNh(L)$ increase the bigrading by $(1,0)$, the
connecting homomorphism $\wh\BNh(L)\to\BNh^-(L)$ increases bigrading
by $(1,2)$ for the Bar-Natan deformation and $(1,4)$ for the Lee
deformation, and the other maps preserve the bigrading.

Commutativity of the diagrams
  \[
   \mathcenter{\begin{tikzpicture}
      \node at (.75,0) (tl0) {$0$};
      \node at (2,0) (Cm) {$\BNcx^-(L)$};
      \node at (4,0) (Ci) {$\BNcx^\infty(L)$};
      \node at (6,0) (Cpt) {$\BNcx^+(L)$};
      \node at (7.25,0) (tr0) {$0$};
      \node at (0.75,-1) (bl0) {$0$};
      \node at (2,-1) (Ch) {$\wh\BNcx(L)$};
      \node at (4,-1) (Cpb1) {$\BNcx^+(L)$};
      \node at (6,-1) (Cpb2) {$\BNcx^+(L)$};
      \node at (7.25,-1) (br0) {$0$};
      \draw[->] (tl0) to (Cm);
      \draw[->] (Cm) to node[above]{\lab{\iota}} (Ci);
      \draw[->] (Ci) to node[above]{\lab{\pi}} (Cpt);
      \draw[->] (Cpt) to (tr0);
      \draw[->] (bl0) to (Ch);
      \draw[->] (Ch) to node[above]{\lab{\iota}} (Cpb1);
      \draw[->] (Cpb1) to node[above]{\lab{U}} (Cpb2);
      \draw[->] (Cpb2) to (br0);
      \draw[->] (Cm) to node[right]{\lab{\pi}} (Ch);
      \draw[->] (Ci) to node[right]{\lab{\pi\circ U^{-1}}} (Cpb1);
      \draw[->] (Cpt) to node[right]{\lab{\Id}} (Cpb2);
    \end{tikzpicture}}
   \text{\quad and\quad }
      \mathcenter{\begin{tikzpicture}
      \node at (.75,0) (tl0) {$0$};
      \node at (2,0) (Cm) {$\BNcx^-(L)$};
      \node at (4,0) (Ci) {$\BNcx^-(L)$};
      \node at (6,0) (Cpt) {$\wh{\BNcx}(L)$};
      \node at (7.25,0) (tr0) {$0$};
      \node at (0.75,-1) (bl0) {$0$};
      \node at (2,-1) (Ch) {$\BNcx^-(L)$};
      \node at (4,-1) (Cpb1) {$\BNcx^\infty(L)$};
      \node at (6,-1) (Cpb2) {$\BNcx^+(L)$};
      \node at (7.25,-1) (br0) {$0$};
      \draw[->] (tl0) to (Cm);
      \draw[->] (Cm) to node[above]{\lab{U}} (Ci);
      \draw[->] (Ci) to  node[above]{\lab{\pi}} (Cpt);
      \draw[->] (Cpt) to (tr0);
      \draw[->] (bl0) to (Ch);
      \draw[->] (Ch) to node[above]{\lab{\iota}} (Cpb1);
      \draw[->] (Cpb1) to node[above]{\lab{\pi}}(Cpb2);
      \draw[->] (Cpb2) to (br0);
      \draw[->] (Cm) to node[right]{\lab{\Id}} (Ch);
      \draw[->] (Ci) to node[right]{\lab{U^{-1}\circ\iota}} (Cpb1);
      \draw[->] (Cpt) to node[right]{\lab{\iota}} (Cpb2);
    \end{tikzpicture}}
  \]
  of short exact sequences and naturality of the snake lemma imply
  that the following diagrams commute:
\begin{equation}\label{eq:bdybdy-tri}
  \vcenter{\hbox{\begin{tikzpicture}[xscale=1.5]
        \node (p) at (0,0) {$\BNh^+(L)$};
        \node (h) at (2,0) {$\wh\BNh(L)$};
        \node (m) at (1,1) {$\BNh^-(L)$};
        \draw[->] (p) -- (h) node[midway,anchor=north] {\tiny $\bdy$};
        \draw[->] (p) -- (m) node[midway,anchor=south east] {\tiny $\bdy$};
        \draw[->] (m) -- (h) node[midway,anchor=south west] {\tiny $\pi_*$};
      \end{tikzpicture}}}
  \text{\quad\and\quad}
  \vcenter{\hbox{\begin{tikzpicture}[xscale=1.5]
        \node (p) at (0,0) {$\wh{\BNh}(L)$};
        \node (h) at (2,0) {$\BNh^-(L)$.};
        \node (m) at (1,1) {$\BNh^+(L)$};
        \draw[->] (p) -- (h) node[midway,anchor=north] {\tiny $\bdy$};
        \draw[->] (p) -- (m) node[midway,anchor=south east] {\tiny $\iota_*$};
        \draw[->] (m) -- (h) node[midway,anchor=south west] {\tiny $\bdy$};        
      \end{tikzpicture}}}
\end{equation}

For the empty link, $\BNh^-(\varnothing)\cong \Ring[U]$,
$\BNh^\infty(\varnothing)\cong \Ring[U^{-1},U]$,
$\BNh^+(\varnothing)\cong \Ring[U^{-1},U]/\Ring[U]$, and
$\wh\BNh(\varnothing)=\Ring$. More generally, for $\BNh^\infty$ we
have the following well-known result. (Recall that for the Lee
deformation we assume $2$ is invertible in $\Ring$.)

\begin{proposition}\label{prop:or-gens-part1}
  In the Bar-Natan theory, there is a canonical isomorphism
  \begin{equation}
    \BNh^\infty(L)\cong \bigoplus_{o\in o(L)}\Ring[H^{-1},H]
  \end{equation}
  where $o(L)$ is the set of orientations of $L$. In the Lee theory,
  after adding a formal square root of $T$, there is a canonical
  isomorphism
  \begin{equation}
    \BNh^\infty(L)\otimes_{\Ring[T]}\Ring[\sqrt{T}]\cong \bigoplus_{o\in o(L)}\Ring[T^{-\frac{1}{2}},T^{\frac{1}{2}}].
  \end{equation}
  In each case, the summand corresponding to an orientation $o$ is
  supported in homological grading $2\mathrm{lk}(L_o,L\setminus L_o)$,
  where $L_o$ is the sublink of $L$ consisting of components whose
  original orientations agree with $o$ and $\mathrm{lk}$ is the
  linking number.
\end{proposition}

\begin{proof}
  The proof is well-known
  (see~\cite{Lee-kh-endomorphism,BNM-kh-degeneration,Tur-kh-diag}), so
  we merely sketch it. The
  Bar-Natan Frobenius algebra $\Ring[H^{-1},H][X]/(X^2=XH)$ has a
  basis $\{ A\defeq X,\ B\defeq H-X\}$ over $\Ring[H^{-1},H]$,
  which diagonalizes it:
  \begin{equation*}
    \begin{split}
      A^2=HA,\qquad B^2=HB,\qquad AB=0\\
      \Delta(A)=A\otimes A,\qquad \Delta(B)=-B\otimes B.
    \end{split}
  \end{equation*}
  For the Lee case, after adding a formal square root of $T$, the
  Frobenius algebra $\Ring[T^{-\frac{1}{2}},T^{\frac{1}{2}}][X]/(X^2=T)$ has a
  basis $\{ A\defeq \sqrt{T}+X,\ B\defeq \sqrt{T}-X\}$, which
  diagonalizes it:
  \begin{equation*}
    \begin{split}
      A^2=2\sqrt{T} A,\qquad B^2=2\sqrt{T} B,\qquad AB=0\\
      \Delta(A)=A\otimes A,\qquad \Delta(B)=-B\otimes B.
    \end{split}
  \end{equation*}
  In the Lee case, note that the homology of
  $\BNcx^\infty(L)\otimes_{\Ring[T]}\Ring[\sqrt{T}]$ is isomorphic to
  $\BNh^\infty(L)\otimes_{\Ring[T]}\Ring[\sqrt{T}]$, since
  $\Ring[\sqrt{T}]$ is free over $\Ring[T]$.

  For any vertex $v\in\{0,1\}^N$ in the cube of resolutions, let $L_v$
  be the corresponding complete resolution of the link diagram
  $L$. With respect to the above basis, the chain group
  $\BNcx^\infty(L)$ is freely generated (over $\Ring[H^{-1},H]$ in the
  Bar-Natan case or $\Ring[T^{-\frac{1}{2}},T^{\frac{1}{2}}]$ in the
  Lee case) by all possible labelings of the circles of $L_v$ by
  $\{A,B\}$, for all $v$. Call two such generators \emph{equivalent}
  if one can be obtained from the other by changing the resolutions at
  some crossings ($0$ to $1$ or $1$ to $0$) so that the
  circles have consistent labelings before and after the change. That
  is, given resolutions $L_v$ and $L_w$, there is a cobordism
  $\Sigma_{v,w}$ from $L_v$ to $L_w$ consisting of saddles at the
  crossings where $v$ and $w$ differ; a generator over $v$ and a
  generator over $w$ are equivalent if for each component $\Sigma$ of
  $\Sigma_{v,w}$, all circles in the boundary of $\Sigma$ have the
  same label.

  Since the basis $\{ A,B\}$ diagonalizes the Frobenius
  algebra, the complex $\BNcx^\infty(L)$ decomposes as a direct sum
  along equivalence classes. Moreover, in each equivalence class, the
  complex is isomorphic to the tensor product of some number of copies
  of the two-step complex
  $\Ring[H^{-1},H]\stackrel{\Id}{\too}\Ring[H^{-1},H]$ in the
  Bar-Natan case or
  $\Ring[T^{-\frac{1}{2}},T^{\frac{1}{2}}]\stackrel{\Id}{\too}
  \Ring[T^{-\frac{1}{2}},T^{\frac{1}{2}}]$ in the Lee case. These
  complexes are acyclic, unless the tensor product is over zero
  copies, that is, the equivalence class contains just a single
  element.  Therefore, the homology $\BNh^\infty(L)$ is generated by
  equivalence classes containing a single element, which are
  generators where every crossing connects two circles in the
  resolution with different labels.

  Such generators, in turn, are in canonical correspondence with
  orientations of $L$, as follows. For any resolution $L_v$, consider
  the checkerboard coloring of $\RR^2\setminus L_v$ where the
  unbounded region is colored white. For any generator over $v$,
  orient each circle in $L_v$ as the boundary of the black
  (respectively, white) region if it is labeled $A$ (respectively,
  $B$). This orientation of $L_v$ induces an orientation of $L$
  precisely for the above type of generators. The statement about the
  homological gradings is straightforward from the description of the
  generators above.
\end{proof}

\section{Behavior under (possibly nonorientable) cobordisms}\label{sec:behave}

We will study the maps induced on these deformed Khovanov complexes and
their homologies by a (possibly nonorientable) cobordism
$F\subset [0,1]\times S^3$ from an oriented link
$L_0\subset\{0\}\times S^3$ to an oriented link
$L_1\subset \{1\}\times S^3$. Throughout the paper, all link
cobordisms will be assumed to be products near the boundary.

Recall the definition of the
\emph{normal Euler number}. Pick Seifert surfaces for the $L_i$ and take a
transverse pushoff $F'$ of $F$ so that the pushoff of $L_i$ is in the
direction of its Seifert surface. (It follows from the Mayer-Vietoris
theorem applied to the decomposition
$S^3=\nbd(L_i)\cup (S^3\setminus L_i)$ that these pushoffs are
independent of the choice of Seifert surfaces.)
Then, the normal Euler number $e$ of $F$ is the signed count of
intersection points between $F$ and $F'$, where the signs come from
picking a local orientation of $F$ near each intersection point and
using the induced local orientation of $F'$. This number is
independent of the choice of pushoff. The normal Euler number is zero
for oriented cobordisms from $L_0$ to $L_1$
and is some even number in general.

We will consider compact link cobordisms decorated with finitely many
marked points which, to be consistent with Khovanov-Robert~\cite{KhRo-kh-Frob2}, we
will call \emph{stars}. So, 
an \emph{elementary cobordism} between link diagrams is one of the
following moves:
\begin{enumerate}[label=(EC-\arabic*)]
\item\label{item:EC1} A planar isotopy of the diagram. 
\item A Reidemeister move.
\item\label{item:EC3} A birth or death of an unknot disjoint from the
  rest of the diagram.
\item\label{item:EC4} No change to the link diagram but a choice of
  a distinguished point (star) in the interior of an arc of the diagram; 
  we interpret this as the identity cobordism with a single star in its
  interior, lying over this distinguished point.
\item\label{item:EC5} A planar saddle.
\item\label{item:EC6} The identity cobordism from a link $L$ to the
  same link with a different orientation on some components.
\end{enumerate}
Associated to each elementary cobordism is a map of the Khovanov
complexes. For Reidemeister moves, these are the maps from the
proof of invariance of these theories. Specifically, Bar-Natan
associates particular picture-world maps to each Reidemeister
move~\cite{Bar-kh-tangle-cob}, and feeding these pictures into the Lee
or Bar-Natan Frobenius algebra gives the map of deformed Khovanov complexes.
The map associated to a birth is the unit $1$ and associated to a death is
the counit $\epsilon$. The map associated to a saddle is obtained by
applying the corresponding multiplication $m$ or comultiplication $\Delta$ to each
resolution. The map associated to the identity cobordism with a star
on some arc $A$ multiplies the label of $A$, in each
resolution, by $2X$ for the Lee deformation and $2X-H$ for the
Bar-Natan deformation (compare~\cite[Formula
(16)]{KhRo-kh-Frob2}). This map depends only on the arc containing
the star, not the location of the star on that arc. The map associated
to the identity cobordism with inconsistent orientations is the
identity map.

Suppose $F$ is a (possibly nonorientable) cobordism from $L_0$ to $L_1$, with a
finite number of marked stars in the interior of $F$. If $F$ is represented by a
movie of elementary cobordisms, then there is an induced map
$\BNcx^-(F)\co \BNcx^-(L_0)\to\BNcx^-(L_1)$, obtained by composing the maps
associated to elementary cobordisms. This induces maps on all the four versions
$\BNcx^\bullet$ of the Khovanov complexes,
$\bullet\in\{+,-,\infty,\widehat{\ }\}$, as well as their homologies
$\BNh^\bullet$.

\begin{lemma}\label{lem:les-natural}
  The maps $\BNcx^\bullet(F)\co \BNcx^\bullet(L_0)\to \BNcx^\bullet(L_1)$,
  $\bullet\in\{+,-,\infty,\widehat{\ }\}$, induce maps
  $\BNh^\bullet\co \BNh(L_0)\to\BNh(L_1)$, and the long exact sequences from
  Formula~\eqref{eq:les} are natural with respect to these maps.
\end{lemma}
\begin{proof}
  This is immediate from the definitions.
\end{proof}

Assuming the link diagrams are oriented coherently before and after
the move, for planar isotopies and Reidemeister moves, the maps
preserve the bigrading, and for births and deaths, the maps preserve
$\gr_h$ and increase $\gr_q$ by $1$. The map associated to a star
preserves $\gr_h$ and decreases $\gr_q$ by $2$. The behavior of saddles is more
complicated.

\begin{lemma}\label{lem:bigrading-shift-saddle}
  Let $F$ be a planar saddle from an oriented link diagram $L_0$ to an
  oriented link diagram $L_1$, which is not necessarily orientable
  coherently with the orientations of $L_0$ and $L_1$. Let $e$ be its
  normal Euler number. Then, the map
  $\BNcx^-(F)\co \BNcx^-(L_0)\to\BNcx^-(L_1)$ decreases $\gr_h$ by
  $e/2$ and decreases $\gr_q$ by $1+3e/2$.
\end{lemma}
\begin{proof}
  Ozsv\'ath-Stipsicz-Szab\'o show that the normal Euler number $e$ of
  the planar saddle $F$ is $w(L_0)-w(L_1)$, where
  $w(L_i)=N_+(L_i)-N_-(L_i)$ is the writhe of the link diagram
  $L_i$~\cite[Proof of Lemma 4.3]{OSS-hf-unoriented}. They write their
  proof only for knots but, as we sketch in the next paragraph, it
  works equally well for links.

  Fix any normal direction to the plane of projection of the link
  diagrams and consider a small pushoff of $L_i$ in this normal
  direction; call this the \emph{blackboard pushoff}. Since the total
  linking number of $L_i$ with its blackboard pushoff is the writhe
  $w(L_i)$ while the total linking number of $L_i$ with its Seifert
  pushoff is zero, the identity cobordism from $L_i$ to $L_i$ has a
  pushoff which intersects itself $w(L_i)$ times, so that the
  pushoff restricts to the Seifert pushoff at one end and the
  blackboard pushoff at the other. The planar saddle also has a
  (similarly defined) blackboard pushoff without any
  self-intersection, and it restricts to the blackboard pushoffs of
  $L_0$ and $L_1$ at the two ends. Putting these pieces together, we
  get a pushoff with $w(L_0)-w(L_1)$ self-intersections connecting the
  Seifert pushoffs of $L_0$ and $L_1$.

  Let $N$ be the total number of crossings in either $L_0$ or
  $L_1$. Recall that the complex $\BNcx^-(L_i)$ is obtained from the
  total complex of a cube-shaped diagram---call it $\BNcx'(L_i)$---by
  increasing $\gr_h$ by $-N_-(L_i)=(w(L_i)-N)/2$ and $\gr_q$ by
  $N_+(L_i)-2N_-(L_i)=(3w(L_i)-N)/2$. The saddle $F$ induces a map
  $\BNcx'(L_0) \to \BNcx'(L_1)$ that preserves the homological grading
  and decreases the quantum grading
  by $1$. Therefore, after the grading shifts, the map
  $\BNcx^-(F)\co \BNcx^-(L_0)\to\BNcx^-(L_1)$ decreases $\gr_h$ by
  $(w(L_0)-w(L_1))/2=e/2$ and decreases $\gr_q$ by
  $1+3(w(L_0)-w(L_1))/2=1+3e/2$.
\end{proof}

\begin{corollary}\label{cor:hq-gr-change}
  (Compare~\cite[Proposition 4.7]{Bal-kh-E1}) The map associated to a
  cobordism with Euler characteristic $\chi$, normal Euler number
  $e$, and $s$ stars decreases $\gr_h$ by $e/2$ and increases $\gr_q$ by
  $\chi-3e/2-2s$.
\end{corollary}
\begin{proof}
  Consider a movie decomposition into elementary cobordisms
  \ref{item:EC1}--\ref{item:EC6}. Choose orientations of all the link
  diagrams that appear in the movie, so that link diagrams before and
  after each of the moves~\ref{item:EC1}--\ref{item:EC4} are oriented
  coherently. (We may choose the orientations inductively, starting
  with the given orientation of $L_0$, and by using a
  move~\ref{item:EC6} if necessary, we may ensure that the chosen
  orientation of $L_1$ agrees with the given orientation of $L_1$.)

  We now check that the statement holds for each of the elementary
  cobordisms.  The elementary
  cobordisms~\ref{item:EC1}--\ref{item:EC3} have $e=s=0$, and the
  associated maps increase the bigrading by
  $(0,\chi)=(-e/2,\chi-3e/2-2s)$. The elementary
  cobordism~\ref{item:EC4} has $e=\chi=0$ and $s=1$, and the
  associated map increases the bigrading by
  $(0,-2)=(-e/2,\chi-3e/2-2s)$. The elementary
  cobordism~\ref{item:EC5} has $\chi=-1$ and $s=0$, and by
  Lemma~\ref{lem:bigrading-shift-saddle}, the associated map increases
  the bigrading by $(-e/2,-1-3e/2)=(-e/2,\chi-3e/2-2s)$.  Finally, for
  cobordisms of type~\ref{item:EC6}, we have $\chi=s=0$, and the
  associated identity map increases the bigrading by
  $(-e/2,-3e/2)=(-e/2,\chi-3e/2-2s)$ as well. (The proof is similar
  to, but easier than, the proof of
  Lemma~\ref{lem:bigrading-shift-saddle}.)  Since $e$, $\chi$, and $s$
  are additive, the composition of these maps also increases the
  bigrading by $(-e/2,\chi-3e/2-2s)$.
\end{proof}

\begin{remark}
  The corollary suggests that another natural grading is
  $\gr_\gamma=\gr_q-3\gr_h$: the map associated to any cobordism
  (possibly nonorientable) increases $\gr_\gamma$ by the Euler
  characteristic of the cobordism minus twice the number of stars.
\end{remark}

We also recall a well-known result of
Rasmussen's~\cite{Ras-kh-slice}. Given a link $L$, let $o(L)$ be
the set of orientations of $L$. Similarly, given a cobordism $F$ from
$L_0$ to $L_1$, let $o(F)$ be the set of orientations of $F$. There
are restriction maps $o(L_0)\leftarrow o(F)\rightarrow o(L_1)$; that
is, $o(F)$ is a \emph{correspondence} from $o(L_0)$ to
$o(L_1)$. (Here, we choose orientation conventions so that if
$F=[0,1]\times L$ then $o(F)$ is the identity correspondence of
$o(L)$.)
\begin{proposition}\label{prop:or-gens-part2}
  Given a cobordism $F$ from $L_0$ to $L_1$, for the Bar-Natan
  and Lee theories, respectively, we have commutative
  diagrams
  \begin{align*}
    &\vcenter{\hbox{\begin{tikzpicture}[xscale=6,yscale=1.5]
          \node (a0) at (0,1) {$\BNh^\infty(L_0)$};
          \node (a1) at (1,1) {$\BNh^\infty(L_1)$};
          \node (b0) at (0,0) {$\displaystyle\bigoplus_{o\in o(L_0)}\Ring[H^{-1},H]$};
          \node (b1) at (1,0) {$\displaystyle\bigoplus_{o\in o(L_1)}\Ring[H^{-1},H]$};
          \draw[->] (a0) -- (a1) node[midway,anchor=south] {\tiny $\BNh^\infty(F)$};
          \draw[->] (b0) -- (b1) node[midway,anchor=south] {\tiny $F_*$};
          \draw[->] (a0) -- (b0) node[midway,anchor=west] {\tiny $\cong$};
          \draw[->] (a1) -- (b1) node[midway,anchor=east] {\tiny $\cong$};
        \end{tikzpicture}}}\\
    &\vcenter{\hbox{\begin{tikzpicture}[xscale=6,yscale=1.5]
          \node (a0) at (0,1) {$\BNh^\infty(L_0)\otimes_{\Ring[T]}\Ring[\sqrt{T}]$};
          \node (a1) at (1,1) {$\BNh^\infty(L_1)\otimes_{\Ring[T]}\Ring[\sqrt{T}]$};
          \node (b0) at (0,0) {$\displaystyle\bigoplus_{o\in o(L_0)}\Ring[T^{-\frac{1}{2}},T^{\frac{1}{2}}]$};
          \node (b1) at (1,0) {$\displaystyle\bigoplus_{o\in o(L_1)}\Ring[T^{-\frac{1}{2}},T^{\frac{1}{2}}]$};
          \draw[->] (a0) -- (a1) node[midway,anchor=south] {\tiny $\BNh^\infty(F)\otimes\Id$};
          \draw[->] (b0) -- (b1) node[midway,anchor=south] {\tiny $F_*$};
          \draw[->] (a0) -- (b0) node[midway,anchor=west] {\tiny $\cong$};
          \draw[->] (a1) -- (b1) node[midway,anchor=east] {\tiny $\cong$};
        \end{tikzpicture}}}
  \end{align*}
  where the vertical arrows are from
  Proposition~\ref{prop:or-gens-part1}, and the bottom arrow is some
  map $F_*$ that refines the correspondence
  $o(L_0)\leftarrow o(F)\rightarrow o(L_1)$. That is, for
  $i\in\{0,1\}$ and for any orientation $o_i\in o(L_i)$ and any
  generator $g_i$ of the $\Ring[H^{-1},H]$ (respectively,
  $\Ring[T^{-\frac{1}{2}},T^{\frac{1}{2}}]$) summand corresponding to
  $o_i$, the coefficient of $g_1$ in $F_*(g_0)$ is a sum
  \(  \sum_{o\in o(F),\\ \ o|_{L_i}=o_i} e_o, \) where
  each $e_o$ is a unit in $\Ring[H^{-1},H]$ (respectively,
  $\Ring[T^{-\frac{1}{2}},T^{\frac{1}{2}}]$). In particular, if $F$ is
  nonorientable, then in either theory, the map
  $\BNh^\infty(F)\from \BNh^\infty(L_0)\to\BNh^\infty(L_1)$ is zero.
\end{proposition}
\begin{proof}
  For the first part, Rasmussen proved~\cite{Ras-kh-slice} the
  result for Lee's deformation with $\Ring=\QQ$ using a change of
  basis that diagonalized Lee's Frobenius algebra (after adding a
  square root of $T$). The change of basis from the proof of
  Proposition~\ref{prop:or-gens-part1} diagonalizes the Frobenius
  algebra, both for the Bar-Natan theory (over any ring $\Ring$) and
  the Lee theory (over any ring $R$ with $2$ invertible, and after
  adding a square root of $T$). Using this diagonalized basis,
  Rasmussen's proof goes through without any essential
  changes. (For an elementary star cobordism, the map is
  multiplication by $A-B$, which sends each orientation to itself with
  coefficient $\pm 1$, and hence fits into Rasmussen's
  framework.)

  The last assertion is automatic for the Bar-Natan theory. For the
  Lee theory, note that if the map
  $\BNh^\infty(F)\otimes\Id\from
  \BNh^\infty(L_0)\otimes_{\Ring[T]}\Ring[\sqrt{T}] \to
  \BNh^\infty(L_1)\otimes_{\Ring[T]}\Ring[\sqrt{T}]$ is zero, then the
  map $\BNh^\infty(F)\from \BNh^\infty(L_0) \to \BNh^\infty(L_1)$ must
  be zero as well. This follows from commutativity of the diagram
  \[
    \begin{tikzpicture}[xscale=6,yscale=1.5]
      \node (a0) at (0,0)   {$\BNh^\infty(L_0)\otimes_{\Ring[T]}\Ring[\sqrt{T}]$};
      \node (a1) at (1,0) {$\BNh^\infty(L_1)\otimes_{\Ring[T]}\Ring[\sqrt{T}]$};
      \node (b0) at (0,1) {$\BNh^\infty(L_0)$};
      \node (b1) at (1,1) {$\BNh^\infty(L_1)$};

      \draw[->] (a0)--(a1)  node[midway,anchor=south] {\tiny $\BNh^\infty(F)\otimes\Id$};
      \draw[->] (b0)--(b1) node[midway,anchor=south] {\tiny $\BNh^\infty(F)$};

      \draw[->] (b0)--(a0);
      \draw[->] (b1)--(a1);
    \end{tikzpicture}
  \]
  and noting that the rightmost vertical map
  \[
    \BNh^\infty(L_1)\to \BNh^\infty(L_1)\otimes_{\Ring[T]}\Ring[\sqrt{T}]\cong \BNh^\infty(L_1)\otimes_{\Ring[T]}\big(\Ring[T]\oplus\sqrt{T}\Ring[T]\big)\cong \BNh^\infty(L_1)\oplus \sqrt{T}\BNh^\infty(L_1)
  \]
  is the inclusion as the first factor, and therefore is injective.
\end{proof}

Finally, we confirm well-definedness of the cobordism maps. Before
stating the main result, we note some relations involving elementary
star cobordisms (cobordisms of type~\ref{item:EC4}).
\begin{lemma}\label{lem:star-commute}
  Up to sign, the map on $\BNcx^-$ associated to an elementary star
  cobordism commutes with the map associated to any elementary
  cobordism disjoint from the star, and commutes with planar isotopies
  in general in the obvious sense. If $p$ and $q$ are points on
  opposite sides of a crossing then the map associated to the
  elementary star cobordism at $p$ is chain homotopic to $-1$ times
  the map associated to the elementary star cobordism at $q$; in
  particular, these two maps also agree up to homotopy and sign.
\end{lemma}
\begin{proof}
  The first statement is straightforward from the definitions. The
  second is immediate from a lemma of Hedden-Ni's \cite[Lemma
  2.3]{HN-kh-detects} in the Lee case and a lemma of
  Alishahi's~\cite[Lemma 2.2]{Ali-kh-unknotting} in the Bar-Natan
  case.
\end{proof}

Well-definedness of the cobordism maps is the following:
\begin{proposition}\label{prop:BNh-functorial}
  Let $F\subset [0,1]\times S^3$ be a (possibly nonorientable) cobordism from
  $L_0$ to $L_1$. For $\bullet\in\{+,-,\infty,\widehat{\ }\}$, the induced map
  $\BNcx^\bullet(F)\from\BNcx^\bullet(L_0)\to\BNcx^\bullet(L_1)$ in the homotopy
  category of complexes over $\Ring[U]$ is well-defined up to sign, and
  invariant under isotopy of $F$ in $[0,1]\times S^3$ rel.~boundary. In fact, if
  $\Phi\co [0,1]\times S^3\to [0,1]\times S^3$ is a diffeomorphism which is the
  identity near the boundary, then $\BNcx^\bullet(F)$ and
  $\BNcx^\bullet(\Phi(F))$ are chain homotopic.
\end{proposition}

\begin{proof}
  Since the maps on $\BNcx^+$, $\BNcx^\infty$, and $\wh{\BNcx}$ are induced by
  the map on $\BNcx^-$, it suffices to prove the result for $\BNcx^-$ (compare
  Lemma~\ref{lem:les-natural}).
  For isotopies of oriented cobordisms in $[0,1]\times \RR^3$, this follows from
  Bar-Natan's result~\cite[Theorem 4]{Bar-kh-tangle-cob} since both
  the Lee perturbation and the Bar-Natan perturbation can be obtained
  functorially from Bar-Natan's diagrammatic invariants. His proof
  works equally well for nonorientable cobordisms: any two movies
  representing isotopic nonorientable cobordisms also differ by a
  sequence of Carter-Saito's movie moves~\cite{CS-knot-movie},
  every local movie move (between sequences of tangles) can be given a
  consistent orientation, and the map induced by a movie is
  independent of the choice of orientations. (In particular, we can
  suppress cobordisms of type~\ref{item:EC6} in these movies.) Finally, functoriality
  for starred cobordisms follows easily from
  Lemma~\ref{lem:star-commute}. (See, for
  instance,~\cite[Lemma 2.1 and 4.1]{Sar-ribbon} for more details.)

  To verify invariance under isotopies in $[0,1]\times S^3$ we must also check
  invariance under Morrison-Walker-Wedrich's \emph{sweep-around
    move}~\cite[Formula (1.1)]{MWW-kh-blobs}. Their proof works
  mutatis mutandis for the Lee and Bar-Natan deformations~\cite[Remark
  2.2]{MWW-kh-blobs}. Nevertheless, for the sake of completeness, we
  present their proof adapted to our setting  below. We will mostly use their
  notation, but a slightly different language.

  Consider Morrison-Walker-Wedrich's picture~\cite[Formula (3.1)]{MWW-kh-blobs}. The
  picture shows two ways of moving a strand from top to bottom to get
  from a link diagram $L$ to a link diagram $L'$: in the first method,
  this strand moves in front of the rest of the link, while in the
  second, it moves behind. Let $L^i_+$ (respectively, $L^i_-$) denote
  the link diagram at the $i\th$ stage in the first (respectively,
  second) method. The two sequences of link diagram $L^i_\pm$ produce
  two chain maps $\BNcx^-(L)\to\BNcx^-(L')$ by composing maps
  associated to Reidemeister moves. By choosing the Reidemeister maps
  carefully, we will show that the two maps agree on the nose.

  Recall that for any link diagram, the homological grading of any
  resolution is given by the number of crossings that have been
  resolved as the $1$-resolution minus the total number of negative
  crossings in the diagram. For the link diagrams that appear above,
  call a crossing \emph{external} (respectively,
  \emph{internal}) if it involves (respectively, does not involve) the
  moving horizontal strand, and define the \emph{external grading}
  (respectively, \emph{internal grading}) of any resolution to be the
  number of external (respectively, internal) crossings that have been
  resolved as the $1$ resolution minus the total number of negative
  external (respectively, internal) crossings. The differential
  preserves or increases the external grading.

  For any of the above link diagrams, let $\BNcx^-_0$ denote the
  subgroup of the chain group that lives in external grading
  $0$. Note that $\BNcx^-_0(L)=\BNcx^-(L)$ and $\BNcx^-_0(L')=\BNcx^-(L')$
  since these diagrams have no external crossings. Also note that there is
  a natural isomorphism of groups
  $\BNcx^-_0(L^i_+)\cong\BNcx^-_0(L^i_-)$ since any resolution of
  $L^i_+$ in external grading $0$, when viewed as a
  resolution of $L^i_-$, is also in external grading $0$. (This uses
  the fact that Morrison-Walker-Wedrich start with a braid
  closure.)

  The Reidemeister maps will be chosen in such a way that they will
  preserve or decrease the external grading. Since the composition is
  a map $\BNcx^-_0(L)\to\BNcx^-_0(L')$, it is therefore enough to
  consider the portion of the maps that preserve the external grading,
  namely, the maps
  \[
    \BNcx^-(L)\to\BNcx^-_0(L^0_\pm)\to\dots\to\BNcx^-_0(L^i_\pm)\to\dots\to\BNcx^-_0(L^\ell_\pm)\to\BNcx^-(L').
  \]
  (These individual maps are typically not chain maps; however,
  perhaps surprisingly, that is irrelevant to the proof.)
  Furthermore, these components of the maps will commute with the isomorphisms
  $\BNcx^-_0(L^i_+)\cong\BNcx^-_0(L^i_-)$, that is, the following
  diagram will commute:
  \[
    \begin{tikzpicture}[xscale=1.5,yscale=0.7]
      \node (L) at (-0.5,0) {$\BNcx^-(L)$};

      \node (L0p) at (1,1) {$\BNcx^-_0(L^0_+)$};
      \node (L0m) at (1,-1) {$\BNcx^-_0(L^0_-)$};

      \node (L1p) at (2,1) {$\cdots$};
      \node (L1m) at (2,-1) {$\cdots$};

      \node (L2p) at (3,1) {$\BNcx^-_0(L^i_+)$};
      \node (L2m) at (3,-1) {$\BNcx^-_0(L^i_-)$};

      \node (L3p) at (4,1) {$\cdots$};
      \node (L3m) at (4,-1) {$\cdots$};

      \node (L4p) at (5,1) {$\BNcx^-_0(L^\ell_+)$};
      \node (L4m) at (5,-1) {$\BNcx^-_0(L^\ell_-)$};

      \node (LL) at (6.5,0) {$\BNcx^-(L')$};

      \draw[->] (L)--(L0p);
      \draw[->] (L)--(L0m);

      \draw[<-] (LL)--(L4p);
      \draw[<-] (LL)--(L4m);

      \foreach\i in{0,2,4}{
        \draw[->] (L\i p)--(L\i m) node[pos=0.5,anchor=west] {\tiny $\cong$};
      }

      \foreach\i [count=\j from 1] in {0,1,2,3}{
        \draw[->] (L\i p)--(L\j p);
        \draw[->] (L\i m)--(L\j m);
      }
    \end{tikzpicture}
  \]
  That will establish that the two compositions agree on the nose.

  For the Reidemeister I and II moves at the beginning and end of the
  sequence, the above check is almost
  automatic. The maps preserve the homological grading.  Since all
  the crossings involved are external, the maps also preserve the internal
  grading, and therefore preserve the external grading as well.

  \begin{figure}
    \centering
    \begin{tikzpicture}[scale=1]

      \foreach\mode/\modesym [count=\y from 0] in {over/+,under/-}{
        \foreach\time/\timesym [count=\x from 0] in {before/,after/+1}{

          \begin{scope}[yshift=-300*\y,xshift=200*\x,yscale=1.5,xscale=2]

            \node at (-1.5+3*\x,3) {$=$};
            
            \node (\mode\time main) at (-2+4*\x,3) {%
              \begin{tikzpicture}[scale=0.4]

                \coordinate (leftleft) at (-1.5,0);
                \coordinate (left) at (-1,0.5-\x);
                \coordinate (centerleft) at (-0.5,0.5-\x);
                \coordinate (center) at (0,0.5-\x);
                \coordinate (centerright) at (0.5,0.5-\x);
                \coordinate (right) at (1,0.5-\x);
                \coordinate (rightright) at (1.5,0);

                \coordinate (topleft) at (-0.5,1);
                \coordinate (topright) at (0.5,1);
                \coordinate (bottomleft) at (-0.5,-1);
                \coordinate (bottomright) at (0.5,-1);

                \coordinate (intnw) at (-0.5,0+\x);
                \coordinate (intne) at (0.5,0+\x);
                \coordinate (intsw) at (-0.5,-1+\x);
                \coordinate (intse) at (0.5,-1+\x);

                \ifnum\x=0
                \node at (0,-1) {\tiny 3};
                \else
                \node at (0,1) {\tiny 3};
                \fi

                \ifnum\y=0
                \node at (-0.9,0) {\tiny 1};
                \node at (0.9,0) {\tiny 2};
                \else
                \node at (-0.9,0) {\tiny 2};
                \node at (0.9,0) {\tiny 1};
                \fi
                
                \ifnum\y=1
                \draw[knot] (leftleft) to[out=0,in=180] (left)--(right) to[out=0,in=180] (rightright);
                \fi
                \draw[knot] (topright)--(intne) to[out=-90,in=90] (intsw)--(bottomleft);
                \draw[knot] (topleft)--(intnw) to[out=-90,in=90] (intse)--(bottomright);
                \ifnum\y=0
                \draw[knot] (leftleft) to[out=0,in=180] (left)--(right) to[out=0,in=180] (rightright);
                \fi

              \end{tikzpicture}
              };

              \node[above=0pt of \mode\time main] {$L^{i\timesym}_{\modesym}$};
            
            \foreach\i in {0,1}{
              \foreach\j in {0,1}{
                    \pgfmathtruncatemacro\extgr{\i+\j}
                    \ifnum\extgr=0
                    \def\col{red}
                    \def\llstyle{dashed}
                    \else
                    \ifnum\extgr=1
                    \def\col{blue}
                    \def\llstyle{solid}
                    \else
                    \def\col{green!50!black}
                    \def\llstyle{densely dotted}
                    \fi
                    \fi

                    \ifnum\y=0
                    \pgfmathtruncatemacro\ii{\i}
                    \pgfmathtruncatemacro\jj{\j}
                    \else
                    \pgfmathtruncatemacro\ii{\j}
                    \pgfmathtruncatemacro\jj{\i}
                    \fi
                    
                \foreach\k in {0,1}{

                  \node[inner sep=0,outer sep=0,anchor=center] (\mode\time\i\j\k) at ($\i*(1,1)+\j*(0,2)+\k*(-1,3)$) {%
                    \begin{tikzpicture}[scale=0.3]
                      \draw[\llstyle,\col] (-1.5,-1) rectangle (1.5,1);

                      \coordinate (intnw) at (-0.5,0+\x);
                      \coordinate (intne) at (0.5,0+\x);
                      \coordinate (intsw) at (-0.5,-1+\x);
                      \coordinate (intse) at (0.5,-1+\x);
                      \coordinate (intw) at (-0.2,-0.5+\x);
                      \coordinate (inte) at (0.2,-0.5+\x);

                      \coordinate (leftleft) at (-1.5,0);
                      \coordinate (left) at (-1,0.5-\x);
                      \coordinate (centerleft) at (-0.5,0.5-\x);
                      \coordinate (center) at (0,0.5-\x);
                      \coordinate (centerright) at (0.5,0.5-\x);
                      \coordinate (right) at (1,0.5-\x);
                      \coordinate (rightright) at (1.5,0);

                      \coordinate (left0) at (-0.5,1-\x-\y);
                      \coordinate (left1) at (-0.5,-\x+\y);
                      \coordinate (right1) at (0.5,1-\x-\y);
                      \coordinate (right0) at (0.5,-\x+\y);

                      \ifnum\ii=0
                      \draw[resol] (leftleft) to[out=0,in=180] (left) ..controls(centerleft).. (left0);
                      \draw[resol] (center) ..controls(centerleft).. (left1);
                      \else
                      \draw[resol] (leftleft) to[out=0,in=180] (left) ..controls(centerleft).. (left1);
                      \draw[resol] (center) ..controls(centerleft).. (left0);
                      \fi

                      \ifnum\jj=0
                      \draw[resol] (rightright) to[out=180,in=0] (right) ..controls(centerright).. (right0);
                      \draw[resol] (center) ..controls(centerright).. (right1);
                      \else
                      \draw[resol] (rightright) to[out=180,in=0] (right) ..controls(centerright).. (right1);
                      \draw[resol] (center) ..controls(centerright).. (right0);
                      \fi

                      \ifnum\k=0
                      \draw[resol] (intnw) to[out=-90,in=-90,looseness=1.3] (intne);
                      \draw[resol] (intsw) to[out=90,in=90,looseness=1.3] (intse);
                      \else
                      \draw[resol] (intnw) to[out=-90,in=90] (intw) to[out=-90,in=90] (intsw);
                      \draw[resol] (intne) to[out=-90,in=90] (inte) to[out=-90,in=90] (intse);
                      \fi
                      
                    \end{tikzpicture}
                  };

                  \begin{scope}[zlevel=foreground]
                    \node[above =0pt of \mode\time\i\j\k,inner sep=1pt,outer sep=1pt,fill=white] {\tiny \i\j\k};
                  \end{scope}

                }}}

            \draw[khdiff] (\mode\time000)--(\mode\time001) node[pos=0.5,linelabel] {\tiny $s$};
            \draw[khdiff] (\mode\time000)--(\mode\time010) node[pos=0.3,linelabel] {\tiny $s$};
            \draw[khdiff] (\mode\time000)--(\mode\time100) node[pos=0.5,linelabel] {\tiny $s$};

            \draw[khdiff] (\mode\time001)--(\mode\time011) node[pos=0.3,linelabel] {\tiny $s$};
            \draw[khdiff] (\mode\time001)--(\mode\time101) node[pos=0.3,linelabel] {\tiny $s$};
            \draw[khdiff] (\mode\time010)--(\mode\time011) node[pos=0.3,linelabel] {\tiny $-s$};
            \draw[khdiff] (\mode\time010)--(\mode\time110) node[pos=0.3,linelabel] {\tiny $s$};
            \draw[khdiff] (\mode\time100)--(\mode\time101) node[pos=0.7,linelabel] {\tiny $-s$};
            \draw[khdiff] (\mode\time100)--(\mode\time110) node[pos=0.6,linelabel] {\tiny $-s$};

            \draw[khdiff] (\mode\time011)--(\mode\time111) node[pos=0.5,linelabel] {\tiny $s$};
            \draw[khdiff] (\mode\time101)--(\mode\time111) node[pos=0.7,linelabel] {\tiny $-s$};
            \draw[khdiff] (\mode\time110)--(\mode\time111) node[pos=0.5,linelabel] {\tiny $s$};

          \end{scope}
          
        }

      \draw[preservemap] (\mode before001) to[out=30,in=150] node[pos=0.45,linelabel] {\tiny $\Id$} (\mode after001);

      \draw[preservemap] (\mode before010)--(\mode after100) node[pos=0.4,linelabel] {\tiny $-dssb$};
      \draw[dropmap] (\mode before010)--(\mode after001) node[pos=0.5,linelabel] {\tiny $-ds$};
      \draw[preservemap] (\mode before010)--(\mode after010) node[pos=0.5,linelabel] {\tiny $-ds$};
      \draw[preservemap] (\mode before100)--(\mode after010) node[pos=0.6,linelabel] {\tiny $\Id$};
      \draw[preservemap] (\mode before100)--(\mode after100) node[pos=0.5,linelabel] {\tiny $sb$};

      \draw[dropmap] (\mode before110)--(\mode after011) node[pos=0.7,linelabel] {\tiny $\Id$};
      \draw[preservemap] (\mode before011)--(\mode after011) node[pos=0.5,linelabel] {\tiny $\Id$};
      \draw[preservemap] (\mode before101)--(\mode after101) node[pos=0.4,linelabel] {\tiny $\Id$};
      
      \draw[preservemap] (\mode before111)--(\mode after111) node[pos=0.5,linelabel] {\tiny $\Id$};

      }

    \end{tikzpicture}
    \caption{The Reidemeister III move during the proof of invariance
      under the sweep-around move.}\label{fig:sweeparound-RIII}
    \end{figure}
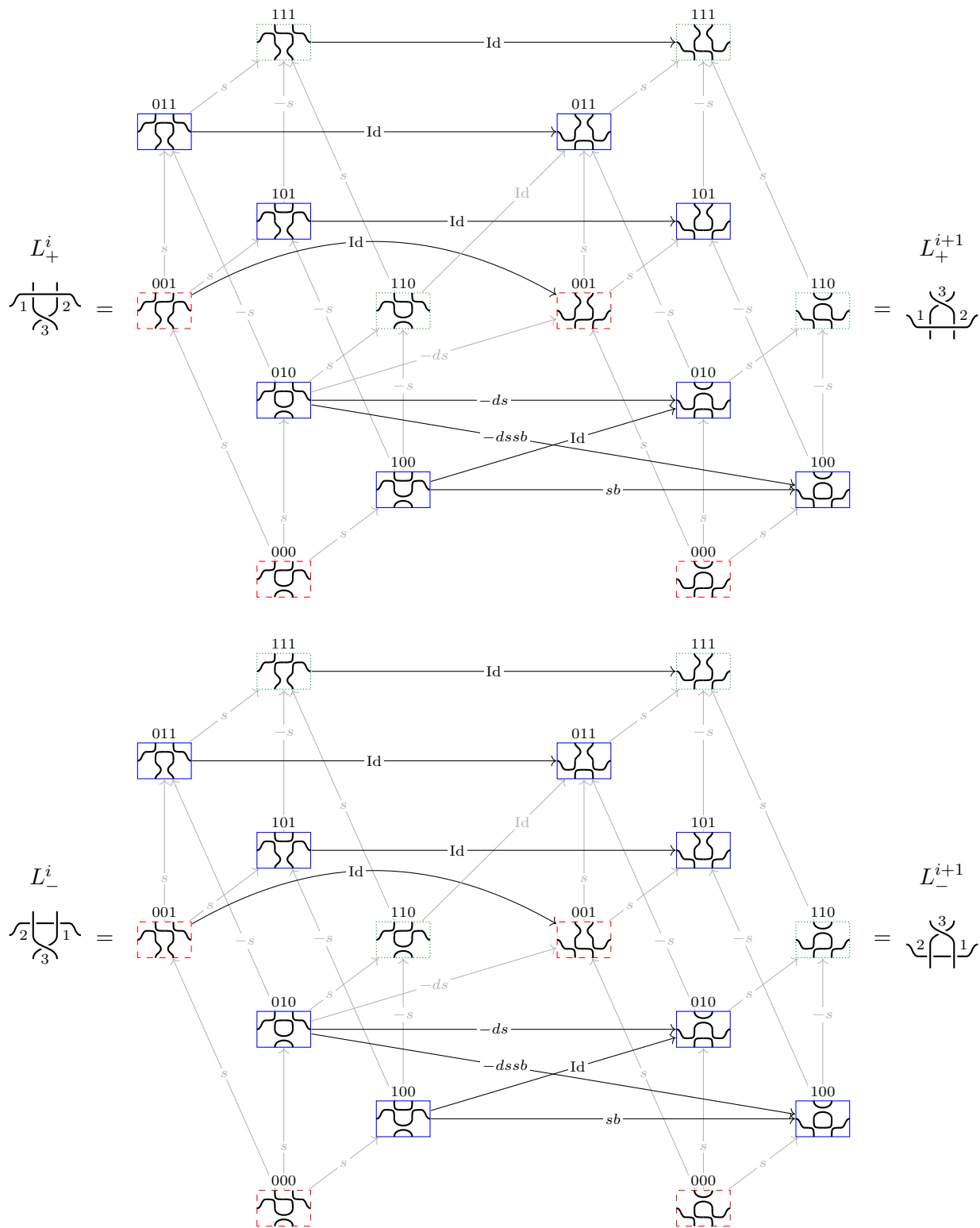

    The Reidemeister III move requires more care. The proof is
    illustrated in Figure~\ref{fig:sweeparound-RIII}. Assume
    $L^{i+1}_\pm$ is obtained from $L^i_\pm$ by moving the horizontal
    strand past an internal crossing, as shown in the figure. The
    other Reidemeister III move is obtained by mirroring all the link
    diagrams, and the proof in that case follows formally from the
    following proof by reversing all arrows.

    The 3-dimensional cubes of resolution for the four link diagrams
    $L^*_\pm$, $*\in\{i,i+1\}$, are shown. The two external crossings
    are numbered 1 and 2---left to right for $L^*_+$ and right to left
    for $L^*_-$---and the internal crossing is numbered 3. The eight
    vertices in each cube of resolutions decompose according to the
    local external grading, which is the sum of the first two
    coordinates of the vertices (up to a shift); this is shown by
    boxing them with a \textcolor{red}{dashed},
    \textcolor{blue}{solid}, or \textcolor{green!50!black}{dotted}
    line. The differentials are shown in light gray.

    The Reidemeister maps go from the cube of resolution of $L^i_\pm$
    to the cube of resolution of $L^{i+1}_\pm$. The maps either
    preserve or decrease the external grading. We are only interested
    in the maps which preserve the external grading, so we have drawn
    them in black, and the other maps in light gray. The maps (and the
    differentials) are decorated with the cobordisms that induce them,
    with $s$, $b$, $d$ being shorthand for saddle, birth, and death,
    respectively.

    The top row (corresponding to $L^*_+$) is essentially a copy
    of Bar-Natan's picture~\cite[Figure 9]{Bar-kh-tangle-cob}; we have merely rotated
    Bar-Natan's tangle so that his vertical over-strand has become our
    horizontal moving strand, and we have reordered the crossings as
    well, so our signs differ from Bar-Natan's. For example, the
    surface highlighted in Bar-Natan's picture corresponds to our map
    from the $010$ vertex of $L^i_+$ to the $100$ vertex of
    $L^{i+1}_-$; it is decorated $-dssb$, so it is the negative of a
    death, followed by two saddles (which are easy to figure out from
    the diagrams), followed by a birth. 

    The bottom row (corresponding to $L^*_-$) is also obtained from
    Bar-Natan's picture~\cite[Figure 9]{Bar-kh-tangle-cob}; this time we have rotated
    Bar-Natan's tangle so that his northwest-to-southeast under-strand
    has become our horizontal moving strand, and once again, we have
    reordered the crossings. The diagram thus obtained is not quite
    the bottom row of our figure: it does not have the map from the
    $100$ vertex of $L^i_-$ to the $100$ vertex of $L^{i+1}_-$, nor
    the map from the $101$ vertex of $L^i_-$ to the $101$ vertex of
    $L^{i+1}_-$, but instead has maps from the $001$ vertex of $L^i_-$
    to the $100$ vertex of $L^{i+1}_-$ and from the $101$ vertex of
    $L^i_-$ to the $110$ vertex of $L^{i+1}_-$. The latter maps
    increase the external gradings, so we modify the chain map by a
    null-homotopy $\bdy f+f\bdy$ to get to our diagram, where $f$ is
    the map from the $101$ vertex of $L^i_-$ to the $100$ vertex of
    $L^{i+1}_-$ corresponding to a birth.

    The natural isomorphism
    $\BNcx^-_0(L^*_+)\stackrel{\cong}{\too}\BNcx^-_0(L^*_-)$ sends the
    \textcolor{green!50!black}{dotted} vertices to the \textcolor{red}{dashed} vertices and vice-versa, and sends the
    \textcolor{blue}{solid} vertices to the corresponding \textcolor{blue}{solid} vertices. These
    isomorphisms commute with the Reidemeister maps that preserve
    external gradings (which are the black arrows in the figure). This
    gives invariance under the sweep-around move.

    Finally, the proof that the map is invariant under diffeomorphisms
    follows from another argument of
    Morrison-Walker-Wedrich~\cite[Section 4.2]{MWW-kh-blobs}; we refer
    the reader there, though a key point in the proof is quoted as
    Lemma~\ref{lem:MWW}, below.
  \end{proof}
  
  \begin{remark}\label{rem:Klein-TQFT}
    The Khovanov Frobenius algebra---the $U=0$ specialization of the
    Lee and Bar-Natan algebras---corresponds to a (1+1)-dimensional
    TQFT. It is natural to ask if this TQFT extends to non-orientable
    cobordisms. TQFTs for oriented 1-manifolds but allowing
    non-orientable cobordisms are called \emph{Klein
      TQFTs}~\cite{AN-top-Klein-tqft} or \emph{unoriented
      TQFTs}~\cite{TT-top-un-tqft}. Unoriented (1+1)-dimensional TQFTs
    correspond to Frobenius algebras $V$ with extra structure: an
    element $\theta\in V$ corresponding to a M\"obius band and an
    involution $\phi$ corresponding to the mapping cylinder of an
    orientation-reversing involution of $S^1$, satisfying the
    conditions that $\phi(m(\theta,v))=m(\theta,v)$ for all $v\in V$
    and $\bigl(m\circ (\phi\otimes\Id)\circ
    \Delta\bigr)(1)=m(\theta,\theta)$~\cite[Proposition 2.9]{TT-top-un-tqft}.
    For the Khovanov TQFT, the fact that
    $\phi$ respects the unit and counit implies that $\phi=\Id$, so
    the second identity implies that $m(\theta,\theta)=2X$ which is
    impossible (cf.~\cite[Section 4.2]{TT-top-un-tqft}). It is
    possible to extend $V$ to a projective unoriented TQFT, by
    defining $\theta=0$, $\phi(1)=1$, and $\phi(X)=-X$; the map $\phi$
    only respects the counit up to sign.
    
    Imitating part of this argument, we can see that it is impossible
    to remedy the sign ambiguity in for non-orientable surfaces
    (without equipping the surfaces with some extra data). There is a
    movie which starts with a 0-crossing unknot, performs a
    Reidemeister I move on half of it, introducing one crossing, then
    performs a Reidemeister I move on the other half eliminating the
    crossing. The induced map $V\to V$ is either
    $(1\mapsto 1,\ X\mapsto -X)$ or $(1\mapsto -1,\ X\mapsto X)$. One
    can compute this directly, but it is also forced by the fact that
    the invariant of a once-punctured Klein bottle is zero (see the
    proof of Corollary~\ref{cor:star-0}): the map associated to a
    once-punctured Klein bottle factors as a birth, then a split, then
    applying the map just described to one of the two circles, and
    then a merge. (This is an embedded version of the proof of the
    relation
    $\bigl(m\circ (\phi\otimes\Id)\circ
    \Delta\bigr)(1)=m(\theta,\theta)$.) However, following the first
    option by a death (counit) gives a cobordism isotopic to a death,
    but sending $X\mapsto -1$ instead of $X\mapsto 1$; and preceding
    the second option by a birth gives a cobordism isotopic to a
    birth, but sending $1\mapsto -1$.
    
    (Note that, in the construction of the Khovanov cube, all the
    surfaces that arise are orientable; in fact, a checkerboard
    coloring of the knot projection induces an orientation of
    them. So, to construct the Khovanov cube one does not need the
    extension of the TQFT to non-orientable surfaces. See
    also~\cite{TT-top-un-tqft} for further discussion.)
    
    Mikhail Khovanov informs us that Greg Kuperberg mentioned to him around 2003
    that Khovanov homology is functorial with respect to nonorientable
    cobordisms in $[0,1]\times\RR^3$, up to a sign, which is part of
    Proposition~\ref{prop:BNh-functorial}, above.
  \end{remark}
  
\section{Admissible cuts}\label{sec:cuts}

Any compact, connected, nonorientable surface $F$ is diffeomorphic to
$(\#^g\RP^2)\#(\#^k \DD^2)$, where $k=|\pi_0(\bdy F)|$ is the number
of boundary components. The number $g=2-\chi(F)-k$ is called the
\emph{crosscap number} of $F$. For any surface $F$ (not necessarily
connected), define its crosscap number to be the sum of the crosscap
numbers of its nonorientable components.

\begin{definition}\label{def:admissible-cut}
  Fix a small $\epsilon>0$. Let $F\subset [0,1]\times S^3$ be a nonorientable cobordism from
  $L_0$ to $L_1$, which is a product near the boundary. An
  \emph{admissible cut} for $F$ consists of the data $(S,V,\phi)$, where:
  \begin{itemize}
  \item $S\subset (0,1)\times S^3$ is a smoothly embedded 3-manifold;
  \item $V\subset (0,1)\times S^3$ is a tubular neighborhood of $S$; and
  \item $\phi\from V\to (\tfrac{1}{2}-\epsilon,\tfrac{1}{2}+\epsilon)\times S^3$ is a
    diffeomorphism, 
  \end{itemize}
  satisfying:
  \begin{enumerate}[label=(AC-\arabic*)]
  \item $\phi$ takes $S$ to $\{\tfrac{1}{2}\}\times S^3$ and $F\cap V$
    to a product cobordism;
  \item the intersection of $F$ with each of the 2 components of
    $([0,1]\times S^3)\setminus S$ is nonorientable; and
  \item\label{item:AC-diffeo} there exists a diffeomorphism $\Phi\co ([0,1]\times S^3,V)\stackrel{\cong}{\longrightarrow}
  ([0,1]\times S^3, (\tfrac{1}{2}-\epsilon,\tfrac{1}{2}+\epsilon)\times S^3)$, which is the
  identity near the boundary and agrees with $\phi$ on $V$.
  \end{enumerate}
  Call a pair of admissible cuts $(S,V,\phi)$ and $(S',V',\phi')$ for
  $F$ \emph{elementary equivalent} if $V\cap V'=\emptyset$ and there
  is a diffeomorphism 
  \begin{equation}\label{eq:admis-cut-equiv}
    \begin{split}
      \bigl([0,1]\times S^3,V,V'\bigr)&\cong \bigl([0,1]\times S^3, (\tfrac{1}{3}-\epsilon,\tfrac{1}{3}+\epsilon)\times
      S^3,(\tfrac{2}{3}-\epsilon,\tfrac{2}{3}+\epsilon)\times S^3\bigr) \text{\qquad or }\\
      \bigl([0,1]\times S^3,V',V\bigr)&\cong \bigl([0,1]\times S^3, (\tfrac{1}{3}-\epsilon,\tfrac{1}{3}+\epsilon)\times
      S^3,(\tfrac{2}{3}-\epsilon,\tfrac{2}{3}+\epsilon)\times S^3\bigr),
    \end{split}
  \end{equation}
  which is the identity near the boundary and which agrees with $\phi$
  and $\phi'$ on $V$ and $V'$, respectively, after post-composition by
  a translation in the first factor.  Call admissible cuts
  $(S,V,\phi)$ and $(S',V',\phi')$ for a pair of surfaces $F$ and $F'$
  \emph{diffeomorphic} if there is a diffeomorphism
  $\Psi\from ([0,1]\times S^3,F,V)\stackrel{\cong}{\longrightarrow}
  ([0,1]\times S^3,F',V')$ which is the identity near the boundary and
  satisfies $\phi'\circ\Psi=\phi$ on $V$.  Call admissible cuts
  $(S,V,\phi)$ and $(S',V',\phi')$ for a pair of surfaces $F$ and $F'$
  \emph{equivalent} if they differ by a sequence of elementary
  equivalences and diffeomorphisms.
\end{definition}

\begin{proposition}\label{prop:admis-cut}
  Suppose $F$ is a cobordism (not necessarily connected) with crosscap
  number $\geq 2$. Then, $F$ has an admissible cut. Further, if $F$ has
  crosscap number $\geq 3$, and if $F'$ is obtained from $F$ by a
  self-diffeomorphism of $[0,1]\times S^3$ which is the identity near the
  boundary, then any admissible cut for $F$ is equivalent to any
  admissible cut for $F'$.
\end{proposition}

We recall some results about the curve complex of nonorientable
surfaces before proving Proposition~\ref{prop:admis-cut}.

Let $F$ be a compact, nonorientable surface. Consider the long
exact sequence for the pair $(F,\bdy F)$,
\[
  \wt{H}^0(\bdy F;\FF_2)\to H^1(F,\bdy F;\FF_2)\to H^1(F;\FF_2)\to H^1(\bdy F;\FF_2).
\]
The first Stiefel-Whitney class $w_1(TF)\in H^1(F;\FF_2)$ maps to zero
in $H^1(\bdy F;\FF_2)$ (since $TF|_{\bdy F}$ is orientable), hence is
in the image of $H^1(F,\bdy F;\FF_2)$. Call a closed curve $\alpha$ in
the interior of $F$ \emph{complement-orientable} if
$\mathrm{PD}([\alpha])\in H^1(F,\bdy F;\FF_2)$ maps to $w_1(TF)$;
assuming $\alpha$ is embedded, this is equivalent to the condition
that $F\setminus\alpha$ is orientable, since for any other curve
$\beta$, $\langle w_1(TF),\beta\rangle=\alpha\cdot\beta\pmod{2}$. Call the other curves
\emph{complement-nonorientable}. Call a closed curve
$\alpha\subset F$ \emph{one-sided} if
$\langle w_1(TF),[\alpha]\rangle=1$; assuming $\alpha$ is embedded,
this is equivalent to $TF|_\alpha$ being a M\"obius band.

\begin{lemma}\label{lem:comp-or-one-side}
  Let $\alpha\subset F$ be a complement-orientable, embedded
  circle. Then, $F$ has a single nonorientable component
  $F_0$. Moreover, $\alpha$ is one-sided if and only if the crosscap
  number of $F_0$ (equivalently, $F$) is odd.
\end{lemma}

\begin{proof}
  By hypothesis, $[\alpha]=\mathrm{PD}(w_1(TF))$. If $F$ has multiple
  nonorientable components, then $[\alpha]$ cannot be represented by a
  single curve.  For the second part, we have to calculate
  $\langle w_1(TF),[\alpha]\rangle=\langle w_1(TF_0)\cup
  w_1(TF_0),[F_0]\rangle=\alpha\cdot\alpha\pmod{2}$. By classification of surfaces and a direct
  computation, this number equals the parity of the crosscap number of
  $F_0$.
\end{proof}

The \emph{one-sided curve complex} of $F$ is the graph with vertices
isotopy classes of embedded, one-sided curves $\alpha$ in the interior
of $F$ and an edge from $\alpha$ to $\beta$ if and only if there are
disjoint representatives of $\alpha$ and $\beta$. The \emph{restricted
  one-sided curve complex} is the full sub-graph spanned by the
complement-nonorientable one-sided curves $\alpha$.  (By
Lemma~\ref{lem:comp-or-one-side}, if $F$ has multiple nonorientable
components or if the crosscap number of $F$ is even, then the
restricted one-sided curve complex is the same as the one-sided curve
complex.)

\begin{proposition}\label{prop:cc-connect}
  Let $F$ be a compact, nonorientable surface (with boundary) of crosscap
  number $\geq 2$. Then, the restricted one-sided curve complex of $F$ is
  connected.
\end{proposition}
\begin{proof}
  This is essentially due to Pieloch~\cite[Proposition 2.7]{Pie-top-nonor}, and
  we follow his argument.

  If $F$ has more than one nonorientable connected component then the statement
  is obvious. So, it suffices to prove the result when $F$ is a connected
  surface with crosscap number $\geq 2$.

  \begin{figure}
    \centering
    \includegraphics{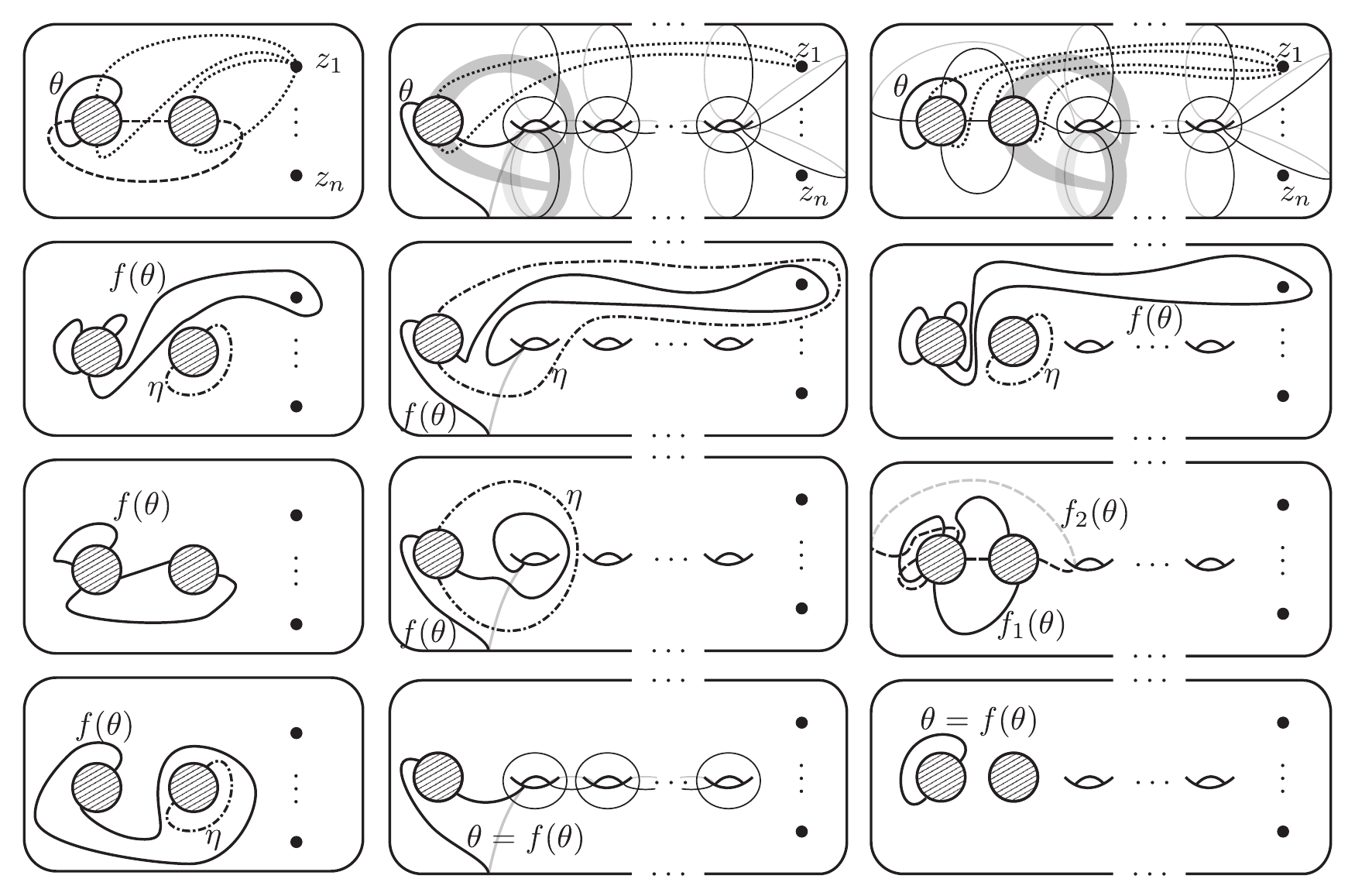}
    \caption{\textbf{Generators for the mapping class group of a
        punctured surface.} The left column is the crosscap number $2$ case, center is the crosscap number $2k+1$ ($k\geq 1$) case, and right is the crosscap number $2k$ ($k>1$) case. Crosscaps are shaded, and hidden lines are gray. Top: the mapping class group is generated by elementary braids in the $z_i$, Dehn twists around the thin curves and the dashed curve, boundary slides (push maps) along the dotted curves, and a crosscap slide. In the left picture, the crosscap slide pulls one crosscap along the dashed curve. In the other pictures, the crosscap slide occurs in the shaded region (a punctured Klein bottle).  Second row: the image of $\theta$ under the boundary slide that moves it. Third row: the images of $\theta$ under the Dehn twist(s) that move it. Bottom row: the image of $\theta$ under the crosscap slide. When $f(\theta)$ is neither disjoint from nor equal to $\theta$, a third curve $\eta$ disjoint from both is shown (dash-dotted).}
    \label{fig:MCG-gens}
  \end{figure}
  
  Let $\theta$ be the one-sided curve shown in Figure~\ref{fig:MCG-gens} and let $\lambda$ be any
  other complement-nonorientable one-sided curve. From the classification of
  surfaces, there is a homeomorphism from $F$ to itself sending $\theta$ to
  $\lambda$. So, it suffices to show that the mapping class group of $F$ takes the
  path component of $\theta$ in the restricted one-sided curve complex to
  itself. For that, it suffices to show that a set of generators for the mapping
  class group take this path component to itself, i.e., take $\theta$ to curves
  which can be connected to $\theta$ in the restricted one-sided curve complex.

  Here, we will not require homeomorphisms to be the identity on $\partial F$ or
  to take boundary components to themselves. In fact, since deleting
  $\partial F$ has no effect on the restricted one-sided curve complex, we can
  view $F$ as a punctured surface and the homeomorphism as an element of the
  mapping class group of the punctured surface $F$. This mapping class group was
  studied by Korkmaz~\cite{Kork-top-mcg}, who denoted it $\mathcal{M}_{g,n}$,
  where $g$ is the crosscap number and $n$ is the number of punctures $z_1,\dots,z_n$.

  In particular, Korkmaz gave a set of generators for this mapping class
  group~\cite[Section 4]{Kork-top-mcg}. There are three cases, depending on the
  crosscap number: crosscap number $2$, $2k+1$ for $k\geq 1$, or $2k$ for
  $k>1$. In each case, the mapping class group is generated by a finite set of
  Dehn twists, braid generators in the $z_i$, one crosscap slide (see~\cite[Section 2]{Kork-top-mcg} and the
  references he gives for the definition), and one or two boundary slides
  (again, see~\cite[Section 2]{Kork-top-mcg}); the generators are shown in
  Figure~\ref{fig:MCG-gens}. In each case, most of the generators fix
  $\theta$. The remaining ones either take $\theta$ to a curve disjoint from
  $\theta$ (up to isotopy) or to a curve $f(\theta)$ so that there is a third
  curve $\eta$ disjoint from both $\theta$ and $f(\theta)$. See
  Figure~\ref{fig:MCG-gens}. So, in all cases, $f(\theta)$ lies in the same path
  component as $\theta$, as desired.
\end{proof}

We note an easier lemma:
\begin{lemma}\label{lem:easy-surfaces}
  If $F$ is a nonorientable surface of crosscap number $>1$ (i.e., is
  not a punctured $\RP^2$ union an orientable surface) then the
  restricted curve complex has at least two points. If $F$ has
  crosscap number $>2$ (i.e., is also neither a punctured Klein bottle
  nor a punctured $\RP^2\amalg\RP^2$, union an orientable surface)
  then the restricted curve complex contains a $3$-cycle.
\end{lemma}
\begin{proof}
  This is straightforward from the classification of surfaces, and is
  left to the reader.
\end{proof}

Proposition~\ref{prop:cc-connect} and Lemma~\ref{lem:easy-surfaces}
together imply the following.
\begin{lemma}\label{lem:even-length-walk}
  Let $F$ be a compact, nonorientable surface with crosscap number
  $>2$. For any complement-nonorientable one-sided embedded curves
  $\alpha,\beta$ in $F$, there exists an even-length sequence
  $\alpha=\gamma_0,\gamma_1,\dots,\gamma_{2n}=\beta$ of
  complement-nonorientable one-sided embedded curves connecting
  $\alpha$ to $\beta$ so that every pair of consecutive curves
  $\gamma_i,\gamma_{i+1}$ are disjoint.
\end{lemma}

\begin{proof}
  Using Proposition~\ref{prop:cc-connect}, we may choose a walk in the
  restricted curve complex connecting $\alpha$ to $\beta$. Using the
  $3$-cycle from the second part of Lemma~\ref{lem:easy-surfaces} if
  needed, we may ensure that the walk has even length. Choose embedded
  curves representing the vertices of the walk to get a sequence
  $\gamma_0,\gamma_1,\dots,\gamma_{2m}$.

  We may choose $\gamma_0=\alpha$ and we may choose $\gamma_i$
  inductively to ensure that it is disjoint from $\gamma_{i-1}$. The
  final curve $\gamma_{2m}$ will be isotopic to $\beta$, but need not
  equal $\beta$. Let $\phi_t$ for $t\in[0,1]$ be an ambient isotopy
  taking $\gamma_{2m}$ to $\beta$.

  Using the first part of Lemma~\ref{lem:easy-surfaces}, choose a
  complement-nonorientable one-sided embedded curve $\delta$ which is
  disjoint from $\gamma_{2m}$.  To finish the proof, we will construct
  a sequence $\gamma_{2m},\gamma_{2m+1},\dots,\allowbreak\gamma_{2n}=\beta$ of
  complement-nonorientable one-sided embedded curves with every
  consecutive pair disjoint, as follows:
  \[
    \gamma_{2i}=\phi_{\frac{i-m}{n-m}}(\gamma_{2m}),\ m\leq i\leq n\qquad\qquad
    \gamma_{2i+1}=\phi_{\frac{i-m}{n-m}}(\delta),\ m\leq i <n.
  \]
  Clearly, $\gamma_{2i+1}$ is disjoint from $\gamma_{2i}$, and by
  compactness, for $n$ large enough, it will be disjoint from
  $\gamma_{2i+2}$ as well. For instance, fix a metric on $F$ so that
  length of $\frac{\del}{\del t}\phi_t$ is bounded above by $1$; let
  $D=\min_{t\in[0,1]}\mathrm{dist}(\phi_t(\gamma_{2m}),\phi_t(\delta))$; then
  $n>m+(1/D)$ suffices.
\end{proof}

Finally we need the following analogue of a result of Morrison-Walker-Wedrich's~\cite[Lemma~4.7]{MWW-kh-blobs}:
\begin{lemma}\label{lem:MWW}
  Let $\Sigma\subset [0,1]\times S^3$ be a link cobordism and
  $f\co [0,1]\times S^3\to[0,1]\times S^3$ a diffeomorphism which is the
  identity near the boundary. Then, $\Sigma$ is isotopic to $f(\Sigma)$,
  and the isotopy may also be assumed to be the identity near the
  boundary.
\end{lemma}

\begin{proof}
  This is proved by replacing $\RR^3$ by $S^3$ throughout
  Morrison-Walker-Wedrich's proof~\cite[Lemma~4.7]{MWW-kh-blobs}, but for completeness, we sketch
  the proof below.

  Let $U$ be a collar neighborhood of $\{0,1\}\times S^3$ on which $f$
  is the identity.  Fix a point $p\in S^3$ so that $[0,1]\times\{p\}$
  is disjoint from $\Sigma$. By postcomposing $f$ by an isotopy that
  takes $f([0,1]\times\{p\})$ to $[0,1]\times\{p\}$ (and is the
  identity on $U$), we may assume $f$ is the identity on
  $[0,1]\times\{p\}$. Let $B\subset S^3$ be a ball around $p$ with
  $[0,1]\times B$ disjoint from $\Sigma$. Let $\phi$ be a
  diffeomorphism of $[0,1]\times S^3$ which is the identity on
  $U\cup\big([0,1]\times\big(\{p\}\cup(S^3\setminus B)\big)\big)$, so
  that on the normal bundle of $[0,1]\times\{p\}$ inside
  $[0,1]\times S^3$ (which is a trivial $\RR^3$ bundle in an obvious
  way) $d\phi$ induces the non-trivial element of
  $\pi_1(\mathrm{SO}(3))$. By post-composing $f$ by $\phi$ if necessary, and a
  further small isotopy, we may assume $f$ is the
  identity on $[0,1]\times B$. Now let $g_t$ be an isotopy of
  $[0,1]\times S^3$ which is the identity near the boundary, so that
  $g_0=\Id$ and $g_1(\Sigma)\subset U\cup([0,1]\times B)$. Then, the
  isotopy $g_t$ takes $\Sigma$ to $g_1(\Sigma)$ and the isotopy
  $f\circ g_{1-t}$ takes $g_1(\Sigma)=f(g_1(\Sigma))$ to $f(\Sigma)$.
\end{proof}

\begin{proof}[Proof of Proposition~\ref{prop:admis-cut}]
  Let $\pi\co [0,1]\times S^3\to S^3$ be projection. Given a curve
  $\gamma\subset F$, let
  \begin{align*}
    C_{\leq \gamma}&=\{(t,p)\in [0,1]\times S^3\mid \exists
    t'\in[0,1]\text{ so that } (t',p)\in \gamma\text{ and }t\leq t'\}\\
    C_{\geq \gamma}&=\{(t,p)\in [0,1]\times S^3\mid \exists
    t'\in[0,1]\text{ so that } (t',p)\in \gamma\text{ and }t\geq t'\},
  \end{align*}
  so $\pi^{-1}(\pi(\gamma))=C_{\leq \gamma}\cup C_{\geq \gamma}$ and
  if $\pi|_{\gamma}$ is injective then
  $C_{\leq \gamma}\cap C_{\geq \gamma}=\gamma$.
    
  For the first statement, that $F$ has an admissible cut, choose disjoint one-sided embedded curves
  $\gamma,\eta\subset F$; this is possible by
  Lemma~\ref{lem:easy-surfaces}.  Perturbing $F$ slightly, we may
  assume:
  \begin{enumerate}[label=(G-\arabic*),ref=(G-\arabic*)]
  \item\label{item:generic-gamma-first} $\pi|_{\gamma}$ is injective,
  \item\label{item:generic-gamma-second} $d\pi$ restricted to $TF|_\gamma$ has rank 2 everywhere,
  \item\label{item:generic-gamma-last} $F$ intersects $C_{\leq\gamma}$ transversely away from $\gamma$,
  \item\label{item:generic-eta-first} $\pi|_{\eta}$ is injective,
  \item $d\pi$ restricted to $TF|_\eta$ has rank 2 everywhere,
  \item\label{item:generic-eta-last} $F$ intersects $C_{\geq\eta}$ transversely away from $\eta$, and
  \item\label{item:generic-last} $\pi(\gamma)\cap\pi(\eta)=\emptyset$.
  \end{enumerate}
  Let $U_{\leq\gamma},U_{\geq\eta}$ be tubular neighborhoods
  of $C_{\leq \gamma}$ and $C_{\geq \eta}$ with disjoint closures. Choose $U_{\leq\gamma}$
  small enough that $U_{\leq\gamma}\cap F$ consists of a M\"obius
  band around $\gamma$ and disks around
  the other points in $C_{\leq\gamma}\cap F$. Choose
  $U_{\geq\eta}$ similarly. Choose $\epsilon>0$ small enough that
  \[
    \bigl(([0,\epsilon]\times S^3)\cup \overline{U}_{\leq\gamma}\bigr)\cap\bigl(([1-\epsilon,1]\times S^3)\cup \overline{U}_{\geq\eta}\bigr)=\emptyset.
  \]
  Let
  \[
    S_{\leq\gamma}=\bdy \bigl(([0,\epsilon]\times S^3)\cup
    \overline{U}_{\leq\gamma}\bigr)\setminus (\{0\}\times S^3),\qquad
    S_{\geq\eta}=\bdy \bigl(([1-\epsilon,1]\times S^3)\cup
  \overline{U}_{\geq\eta}\bigr)\setminus (\{1\}\times S^3).
  \]
  Then, after smoothing corners, both $S_{\leq\gamma}$ and $S_{\geq\eta}$
  are admissible cuts for $F$, since either side of either cut is
  nonorientable (containing one of the one-sided curves $\gamma$ or
  $\eta$). The diffeomorphisms $\phi$ from
  Definition~\ref{def:admissible-cut} for $S_{\leq\gamma}$ and
  $S_{\geq\eta}$ are induced from the natural isotopies that take
  $\gamma$ to $\{\epsilon\}\times \pi(\gamma)$ along $C_{\leq\gamma}$ and $\eta$ to
  $\{1-\epsilon\}\times \pi(\eta)$ along $C_{\geq\eta}$.

  For the second statement, fix admissible cuts $S$ and $S'$ for $F$
  and $F'$ and choose diffeomorphisms
  $\Phi\co ([0,1]\times S^3,S)\stackrel{\cong}{\longrightarrow}
  ([0,1]\times S^3, \{\tfrac{1}{2}\}\times S^3)$ and
  $\Phi'\co ([0,1]\times S^3,S')\stackrel{\cong}{\longrightarrow}
  ([0,1]\times S^3, \{\tfrac{1}{2}\}\times S^3)$ as in
  Definition~\ref{def:admissible-cut}. It is enough to show that the admissible cuts
  $\{\tfrac{1}{2}\}\times S^3$ for $\Phi(F)$ and $\Phi'(F')$ are
  equivalent.

  Choose one-sided curves $\gamma,\eta\subset \Phi(F)$ on the left and
  right side of $\{\tfrac{1}{2}\}\times S^3$ and choose a one-sided
  curve $\alpha\subset \Phi'(F')$ on the left side of
  $\{\tfrac{1}{2}\}\times S^3$. Consider the admissible cuts
  $S_{\leq\gamma}$ and $S_{\geq\eta}$ for $\Phi(F)$ as defined
  above. They are both elementary equivalent to the cut
  $\{\tfrac{1}{2}\}\times S^3$ (as well as to each other). Similarly the
  admissible cut $S_{\leq\alpha}$ for $\Phi'(F')$
  is elementary equivalent to the cut $\{\tfrac{1}{2}\}\times S^3$. In particular,
  it is enough to show that the cut $S_{\leq\gamma}$ for $\Phi(F)$
  is equivalent to the cut $S_{\leq\alpha}$ for $\Phi'(F')$.

  Using Lemma~\ref{lem:MWW}, choose an ambient isotopy
  $\psi_t\from[0,1]\times S^3\to[0,1]\times S^3$ which is the identity
  near the boundary, with $\psi_0=\Id$, and
  $\psi_1(\Phi(F))=\Phi'(F')$. We may perturb the isotopy
  to ensure that the genericity
  assumptions~\ref{item:generic-gamma-first}--\ref{item:generic-last}
  hold for the curves
  $\psi_t(\gamma),\psi_t(\eta)\subset \psi_t(\Phi(F))$, except for
  finitely many $t\in(0,1)$ when exactly one of them fails. Choose
  non-exceptional points $0=t_0<t_1<\dots<t_{m-1}<t_m=1$ such that
  each interval $[t_i,t_{i+1}]$ contains at most one such exceptional
  point; call such an interval \emph{$\gamma$-good} (respectively,
  \emph{$\eta$-good}) if the genericity
  conditions~\ref{item:generic-gamma-first}--\ref{item:generic-gamma-last}
  (respectively,~\ref{item:generic-eta-first}--\ref{item:generic-eta-last})
  hold on the interval. Since at most one genericity condition fails
  at each exceptional point, each of these intervals $[t_i,t_{i+1}]$
  is either $\gamma$-good or $\eta$-good (or both, if~\ref{item:generic-last} fails).

  At non-exceptional $t$ (such as $t_0,\dots,t_m$), both
  the cuts $S_{\leq\psi_t(\gamma)}$ and $S_{\geq\psi_t(\eta)}$ for
  $\psi_t(\Phi(F))$ are admissible, and they are elementary equivalent
  to each other. We need the following lemma.
  \begin{lemma}\label{lem:inside-prop}
    For any $\gamma$-good (respectively, $\eta$-good) interval
    $[t_i,t_{i+1}]$, the cuts $S_{\leq\psi_{t_i}(\gamma)}$
    (respectively, $S_{\geq\psi_{t_i}(\eta)}$) for
    $\psi_{t_i}(\Phi(F))$ and $S_{\leq\psi_{t_{i+1}}(\gamma)}$
    (respectively, $S_{\geq\psi_{t_{i+1}}(\eta)}$) for
    $\psi_{t_{i+1}}(\Phi(F))$ are diffeomorphic.
  \end{lemma}

  \begin{proof}
    We will explain the $\gamma$-good case; the $\eta$-good case is
    similar. For notational convenience, let $a=t_i$, $b=t_{i+1}$,
    $F_a=\psi_{a}(\Phi(F))$, $\gamma_a=\psi_a(\gamma)$,
    $F_b=\psi_b(\Phi(F))$, and $\gamma_b=\psi_b(\gamma)$. Note also
    that, in the definition of $S_{\leq \gamma_a}$ and
    $S_{\leq \gamma_b}$, up to diffeomorphism, we can choose the
    collar neighborhoods of the boundary and the tubular neighborhoods
    of $C_{\leq \gamma_a}$ and $C_{\leq \gamma_b}$ to be as small as we like.
   
    To construct the diffeomorphism
    \[
      ([0,1]\times
      S^3,F_a,S_{\leq\gamma_a})\stackrel{\cong}{\too}([0,1]\times
      S^3,F_b,S_{\leq\gamma_b})
    \]
    we first construct an ambient isotopy $\theta_t$, $t\in[a,b]$, of
    $[0,1]\times S^3$ which carries $F_a\cup C_{\leq\gamma_a}$ to
    $F_b\cup C_{\leq\gamma_b}$.

    The map $\psi_t\circ \psi_a^{-1}$
    restricts to an isotopy $\theta^F_t$ from $F_a$ to $F_b$, and the
    isotopy $(\psi_t\circ\psi_a^{-1})|_{\gamma_a}$ induces an isotopy
    $\theta^\gamma_t$ from $C_{\leq \gamma_a}$ to $C_{\leq
      \gamma_b}$. (This uses conditions~\ref{item:generic-gamma-first}
    and~\ref{item:generic-gamma-second}.) The isotopy
    $\theta^\gamma_t$ will not be the identity on the boundary, but we
    can choose it so that, for all small $s$, it sends
    $C_{\leq \gamma_a}\cap (\{s\}\times S^3)$ to $C_{\leq
      \gamma_t}\cap (\{s\}\times S^3)$ for each $t$.    
    The maps $\theta^F_t$ and $\theta^\gamma_t$ typically will not
    agree on $F_a\cap C_{\leq \gamma_a}$, but by
    Condition~\ref{item:generic-gamma-last}, for any $t\in[a,b]$,
    $\theta^F_t(F_a)$ and the interior of
    $\theta^\gamma_t(C_{\leq\gamma_a})$ intersect transversally in
    finitely many points, say $P_t$.  So, we get one-parameter families
    of points $(\theta^F_t)^{-1}(P_t)$ on $F_a$ and
     $(\theta^\gamma_t)^{-1}(P_t)$ on
    $C_{\leq\gamma_a}$.  Let $\tilde{\theta}^F_t$ be an isotopy of
    $F_a$ which preserves it setwise, is the identity near the boundary
    and near $\gamma_a$, and satisfies
    $\tilde{\theta}^F_t(P_a)=(\theta^F_t)^{-1}(P_t)$.  Similarly, let
    $\tilde{\theta}^\gamma_t$ be an isotopy of $C_{\leq\gamma_a}$
    which preserves it setwise, is the identity near the boundary, and
    satisfies
    $\tilde{\theta}^\gamma_t(P_a)=(\theta^\gamma_t)^{-1}(P_t)$. 
    On $F_a$, set $\theta_t=\theta^F_t\circ\tilde{\theta}^F_t$, and
    on $C_{\leq\gamma_a}$, set
    $\theta_t=\theta^\gamma_t\circ\tilde{\theta}^\gamma_t$.  These two
    isotopies do agree on $F_a\cap C_{\leq \gamma_a}$, so let
    $\theta_t\co F_a\cup C_{\leq\gamma_a}\to [0,1]\times S^3$,
    $t\in[a,b]$, be their union. 
    By the isotopy extension lemma, we can extend $\theta_t$ to a
    smooth ambient isotopy $\theta_t$ of $[0,1]\times S^3$. (This again uses Conditions~\ref{item:generic-gamma-second}
    and~\ref{item:generic-gamma-last}.) We can
    ensure that this extension preserves the slices $\{s\}\times S^3$
    for $s\in [0,2\epsilon]\cup [1-2\epsilon,1]$, for some
    sufficiently small $\epsilon$; shrinking $\epsilon$ if necessary,
    we may assume that $\theta_t$ is the identity on $F_a\cap
    ([0,2\epsilon]\cup[1-2\epsilon,1])\times S^3$.

    The ambient isotopy $\theta_t$ is not the identity on
    the whole boundary, but $\theta_t|_{\{0,1\}\times S^3}$ is, of course,
    isotopic to the identity.  So, we can modify $\theta_t$ to a new
    isotopy $\theta'_t$ so that:
    \begin{itemize}
    \item $\theta'_t$ is the identity on a small collar neighborhood
      $S^3\times ([0,\epsilon]\cup[1-\epsilon,1])$ of the boundary.
    \item $\theta'_t$ preserves $\{s\}\times S^3$ for $s\in
      [0,2\epsilon]\cup[1-2\epsilon,1]$.
    \item $\theta'_t$ agrees with $\theta_t$ on
      $[2\epsilon,1-2\epsilon]\times S^3$.
    \item $\theta'_t$ is the identity on $F_a\cap
      ([0,2\epsilon]\cup[1-2\epsilon,1])\times S^3$. So, in
      particular, $\theta'_b(F_a)=F_b$.
    \end{itemize}
    Choose the collar neighborhoods of the boundary, in the definition
    of $S_{\leq \gamma_a}$ and $S_{\leq \gamma_b}$, to be
    $([0,2\epsilon]\cup[1-2\epsilon,1])\times S^3$.  Then, for
    appropriate choices of tubular neighborhoods of $C_{\leq\gamma_a}$
    and $C_{\leq \gamma_b}$, the map $\theta'_b$ sends $F_a$ to $F_b$
    and $S_{\leq\gamma_a}$ to $S_{\leq \gamma_b}$, as desired.
  \end{proof}    

  We can now conclude that for any interval $[t_i,t_{i+1}]$, the cuts
  $S_{\leq\psi_{t_i}(\gamma)}$ for $\psi_{t_i}(\Phi(F))$ and
  $S_{\leq\psi_{t_{i+1}}(\gamma)}$ for $\psi_{t_{i+1}}(\Phi(F))$ are
  equivalent. If the interval is $\gamma$-good, then this follows from
  Lemma~\ref{lem:inside-prop}. On the other hand, if the interval is $\eta$-good,
  then the cuts $S_{\geq\psi_{t_i}(\eta)}$ for $\psi_{t_i}(\Phi(F))$
  and $S_{\geq\psi_{t_{i+1}}(\eta)}$ for $\psi_{t_{i+1}}(\Phi(F))$ are
  diffeomorphic, again by Lemma~\ref{lem:inside-prop}; however, the former is elementary equivalent to
  $S_{\leq\psi_{t_i}(\gamma)}$, while the latter is elementary
  equivalent to $S_{\leq\psi_{t_{i+1}}(\gamma)}$.  Therefore, we get
  that the cut $S_{\leq\psi_0(\gamma)}=S_{\leq\gamma}$ for
  $\psi_0(\Phi(F))=\Phi(F)$ and the cut $S_{\leq\psi_1(\gamma)}$ for
  $\psi_1(\Phi(F))=\Phi'(F')$ are equivalent. So all that remains is
  to show that the two cuts $S_{\leq\psi_1(\gamma)}$ and
  $S_{\leq\alpha}$ for $\Phi'(F')$ are equivalent.

  Using Lemma~\ref{lem:even-length-walk}, choose an even-length
  sequence $\alpha=\delta_0,\delta_1,\dots,\delta_{2n}=\psi_1(\gamma)$
  of complement-nonorientable one-sided embedded curves on $\Phi'(F')$
  so that every pair of consecutive curves are disjoint. Perturbing
  slightly, we may assume that the genericity
  conditions~\ref{item:generic-gamma-first}--\ref{item:generic-gamma-last}
  hold for each of these curves, and the genericity
  condition~\ref{item:generic-last} holds for each pair of consecutive
  curves. Then,
  \(
    S_{\leq\delta_0},S_{\geq\delta_1},S_{\leq\delta_2},\dots,S_{\leq\delta_{2n}}
  \)
  is a sequence of admissible cuts connecting $S_{\leq\alpha}$ and
  $S_{\leq\psi_1(\gamma)}$ where every consecutive pair is elementary
  equivalent.
\end{proof}

\begin{remark}\label{remark:OSz}
  The proof that all admissible cuts are equivalent is inspired by the
  $b_2^+\geq 3$ case of Ozsv\'ath-Szab\'o's
  argument~\cite[Theorem~8.5]{OSz-hf-4manifolds}. Their argument is
  terse, so for comparison with the proof of
  Proposition~\ref{prop:admis-cut} we expand the $b_2^+\geq 3$ case of
  their argument here.

  Given a compact, oriented, connected 4-dimensional cobordism
  $W\co Y_0\to Y_1$, $Y_i$ connected, an \emph{admissible cut} for $W$
  is a decomposition $W=W_0\cup_NW_1$ along a closed, connected
  3-manifold $N$ so that both $b_2^+(W_0)>0$ and $b_2^+(W_1)>0$ (and
  $Y_i\subset W_i$). Ozsv\'ath-Szab\'o show that any 4-manifold $W$
  with $b_2^+(W)\geq 2$ has an admissible cut, as follows. Choose
  closed, connected, oriented surfaces $\Sigma_0,\Sigma_1\subset W$
  with $[\Sigma_0]^2,[\Sigma_1]^2>0$ and
  $[\Sigma_0]\cdot[\Sigma_1]=0$. One can make $\Sigma_0$ and
  $\Sigma_1$ disjoint by repeatedly performing embedded surgery on
  $\Sigma_0$ to cancel pairs of points $p,q\in \Sigma_0\cap \Sigma_1$
  of opposite sign, along an arc in $\Sigma_1$ from $p$ to $q$. Let
  $W_0$ be a neighborhood of the union of $\Sigma_0$, an arc from
  $\Sigma_0$ to $Y_0$, and $Y_0$. Choose the arc generically and its
  neighborhood small enough to be disjoint from $\Sigma_1$, and let
  $W_1$ be the complement of the interior of $W_0$. Then, $W_0\cup W_1$
  is an admissible cut for $W$.

  \begin{figure}
    \centering
    \includegraphics{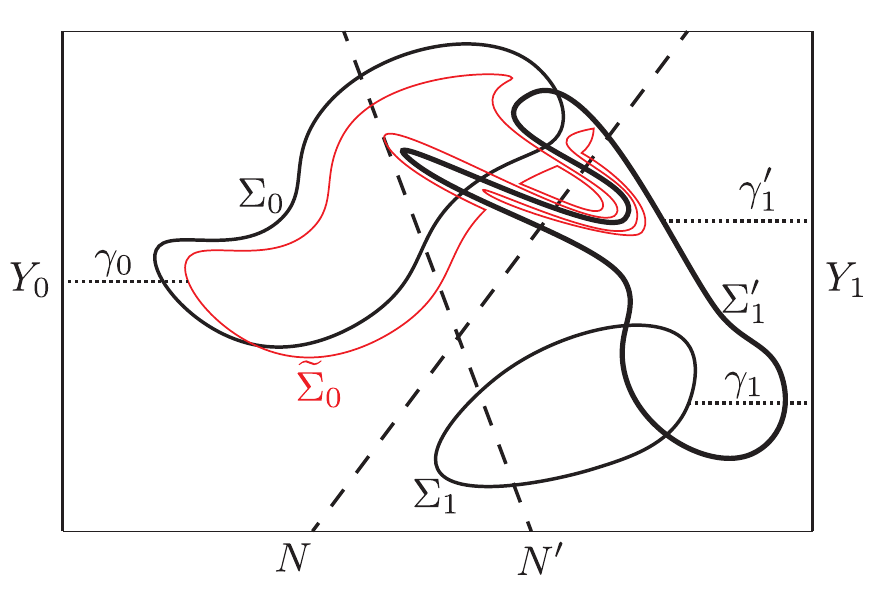}
    \caption{\textbf{Equivalence of admissible cuts in Ozsv\'ath-Szab\'o's
        setting.} This is a schematic illustration of the surfaces and paths in
      Remark~\ref{remark:OSz}. The cuts $N$ and $N'$ are dashed, the paths
      $\gamma$ are dotted, the surface $\Sigma_1'$ is thick, and the surface
      $\wt{\Sigma}_0$ is \textcolor{red}{thin}. A key point is that $\wt{\Sigma}_0$ is disjoint
      from $\Sigma_1$ and $\Sigma'_1$.}
    \label{fig:OSz}
  \end{figure}

  Next, call admissible cuts $W=W_0\cup_NW_1=W'_0\cup_{N'}W'_1$
  \emph{elementary equivalent} if $N\cap N'=\emptyset$, and
  \emph{equivalent} if they differ by a sequence of elementary
  equivalences.  If $b_2^+(W)\geq 3$ then Ozsv\'ath-Szab\'o argue that
  any two admissible cuts for $W$ are equivalent. Fix admissible cuts
  $W=W_0\cup_NW_1=W_0'\cup_{N'}W_1'$. For convenience, assume that
  $b_2^+(W_1')\geq 2$; the other case is symmetric. Choose connected
  surfaces $\Sigma_0\subset \interior{W}_0$ with $[\Sigma_0]^2>0$ and
  $\Sigma'_1\subset \interior{W}'_1$ with $[\Sigma'_1]^2>0$, and so
  that $[\Sigma_0]\cdot[\Sigma'_1]=0$ (this uses the assumption on
  $b_2^+(W'_1)$). (This part of the argument is illustrated schematically in Figure~\ref{fig:OSz}.)  Performing surgery on $\Sigma_0$ along arcs $\Gamma$
  in $\Sigma'_1$ with interiors disjoint from
  $\Sigma_0$ gives a new surface $\wt{\Sigma}_0$ homologous to
  $\Sigma_0$ and disjoint from $\Sigma'_1$.  Let
  $\Sigma_1\subset \interior{W}_1$ be a surface with
  $[\Sigma_1]^2>0$. Since $\Sigma_0\subset \interior{W}_0$ and
  $\Sigma_1\subset \interior{W}_1$,
  $\Sigma_0\cap\Sigma_1=\emptyset$. Perturbing $\Sigma_1$ slightly, we
  can assume that $\Sigma_1$ is also disjoint from the arcs $\Gamma$,
  and hence from $\wt{\Sigma}_0$.

  Choose an arc $\gamma_0\subset W_0$
  connecting $\wt{\Sigma}_0$ to $Y_0$, disjoint from $\Sigma'_1$;
  $\gamma'_1\subset W'_1$ connecting $\Sigma'_1$ to $Y_1$, disjoint
  from $\wt{\Sigma}_0\cup\gamma_0$; and $\gamma_1\subset W_1$ connecting $\Sigma_1$ to
  $Y_1$, disjoint from $\Gamma$. Let $\wt{N}_0$ be the boundary of a neighborhood of
  $Y_0\cup\gamma_0\cup\wt{\Sigma}_0$, $N'_1$ the boundary of a
  neighborhood of $Y_1\cup \gamma'_1\cup \Sigma'_1$, and $N_1$ the
  boundary of a neighborhood of $Y_1\cup\gamma_1\cup\Sigma_1$.
  Observe that $\wt{N}_0$, $N_1$, and $N'_1$ are all admissible
  cuts. Further, $N$ is equivalent to $N_1$, which is equivalent to
  $\wt{N}_0$, which is equivalent to $N'_1$, which is equivalent to
  $N'$, proving the result.
\end{remark}

\section{The mixed invariant}\label{sec:mixed}

Consider the long exact sequence
$\cdots\to
\BNh^-(L)\stackrel{\iota_*}{\lra}\BNh^\infty(L)\stackrel{\pi_*}{\lra}
\BNh^+(L)\stackrel{\partial}{\lra}\BNh^-(L)\to\cdots$ from
Formula~\eqref{eq:les}. Define
$\BNh^{\red}(L)=\ker(\iota_*)\cong \cokernel(\pi_*)$, and give it the
grading it inherits as a submodule of $\BNh^-(L)$; this is $(1,0)$
higher than its grading as a quotient module of $\BNh^+(L)$. (This is
analogous to the reduced Heegaard Floer invariant $\HF_\red$, and is
not immediately related to reduced Khovanov homology.)

\begin{definition}\label{def:compat-with-cut}
  Fix an admissible cut $(S,V,\phi)$ for a cobordism
  $F\subset [0,1]\times S^3$, decomposing $F$ into $F_0$ and
  $F_1$. Let $\Phi$ be a diffeomorphism as in
  Condition~\ref{item:AC-diffeo} of
  Definition~\ref{def:admissible-cut}. Given a movie $M_0$
  representing $\Phi(F_0)\subset [0,1/2]\times S^3$ (identified with
  $[0,1]\times S^3$ in the obvious way) and a movie $M_1$ representing
  $\Phi(F_1)\subset [1/2,1]\times S^3$, we say that the concatenated
  movie $(M_0,M_1)$ is a \emph{movie compatible with the admissible
    cut $(S,V,\phi)$}.
\end{definition}

\begin{lemma}\label{lem:non-or-red}
  Let $F$ be a nonorientable cobordism in $[0,1]\times S^3$ from
  $L_0$ to $L_1$. Then, the image of the induced map
  $\BNh^-(F)\co \BNh^-(L_0)\to \BNh^-(L_1)$ lies in
  $\BNh^{\red}(L_1)\subset \BNh^-(L_1)$, and the map
  $\BNh^+(F)\co \BNh^+(L_0)\to \BNh^+(L_1)$ descends to a map
  $\BNh^{\red}(L_0)\to\BNh^{+}(L_1)$.
\end{lemma}
\begin{proof}
  By Proposition~\ref{prop:or-gens-part2},
  $\BNh^\infty(F)\co \BNh^\infty(L_0)\to \BNh^\infty(L_1)$ vanishes. So, both
  claims follow from the first long exact sequence in Formula~\eqref{eq:les} and
  Lemma~\ref{lem:les-natural}.
\end{proof}

\begin{definition}\label{def:mixed-invt}
  Let $F$ be a nonorientable cobordism from $L_0$ to $L_1$  with crosscap number $\geq 3$. Let $S$ be an admissible
  cut for $F$, decomposing $F$ as $F_1\circ F_0$ and $[0,1]\times S^3$
  as $W_1\circ W_0$. Choose a movie compatible with the admissible
  cut, and let $\BNh^\pm(F_i)$ be the induced maps.
  By Lemma~\ref{lem:non-or-red}, the map
  $\BNh^-(F_0)$ induces a map
  $\BNh(F_0)\co\BNh^-(L_0)\to \BNh^{\red}(L)$ and $\BNh^+(F_1)$
  descends to a map $\BNh(F_1)\co\BNh^{\red}(L)\to\BNh^+(L_1)$. Define
  the mixed invariant $\MixedInvt{F}\co \BNh^-(L_0)\to\BNh^+(L_1)$ to be
  the composition $\BNh(F_1)\circ\BNh(F_0)$. That is, $\MixedInvt{F}$ is
  is the composition of the two dashed arrows in the following diagram:
\begin{equation}\label{eq:mixed-invt-diagram}
\mathcenter{\begin{tikzpicture}[xscale=5,yscale=1.8]
\node (minus-0) at (0,1) {$\BNh^-(L_0)$};
\node (inf-0) at (0,.45) {$\BNh^\infty(L_0)$};

\node (inff-1) at (1,2.55) {$\BNh^\infty(L)$};
\node (plus-1) at (1,2) {$\BNh^+(L)$};
\node (minus-1) at (1,1) {$\BNh^-(L)$};
\node (inf-1) at (1,.45) {$\BNh^\infty(L)$};
\node (red) at (1,1.5) {$\BNh^\red(L)$};

\node (inff-2) at (2,2.55) {$\BNh^\infty(L_1)$};
\node (plus-2) at (2,2) {$\BNh^+(L_1)$};

\draw[->] (minus-0) -- (inf-0);

\draw[->] (inff-1) -- (plus-1);
\draw[->] (plus-1) -- (red);
\draw[->] (red)--(minus-1);
\draw[->] (minus-1) -- (inf-1);

\draw[->] (inff-2) -- (plus-2);

\draw[->] (inf-0) -- (inf-1) node[midway,below] {\tiny $\BNh^\infty(F_0)=0$};
\draw[->] (minus-0) -- (minus-1) node[midway,below] {\tiny $\BNh^-(F_0)$};

\draw[->,dashed] (minus-0) -- (red) node[midway,above] {\tiny $\BNh(F_0)$};
\draw[->,dashed] (red) -- (plus-2) node[midway,below] {\tiny $\BNh(F_1)$};

\draw[->] (inff-1) -- (inff-2) node[midway,above] {\tiny $\BNh^\infty(F_1)=0$};
\draw[->] (plus-1) -- (plus-2) node[midway,above] {\tiny $\BNh^+(F_1)$};
\end{tikzpicture}}
\end{equation}
\end{definition}

\begin{remark}\label{rem:2-crosscaps}
  Definition~\ref{def:mixed-invt} also makes sense if $F$ has crosscap
  number $2$, but the proof that $\MixedInvt{F}$ is independent of the
  choice of admissible cut (Theorem~\ref{thm:invt}) requires crosscap
  number at least $3$. That is, in the case $F\co L_0\to L_1$ has
  crosscap number $2$, there is a map
  $\MixedInvt{F,S}\co \BNh^-(L_0)\to\BNh^+(L_1)$ which, as far as we
  know, depends on both the surface $F$ and the admissible cut $S$.
\end{remark}

\begin{theorem}\label{thm:invt}
  Let $F$ be a cobordism from $L_0$ to $L_1$, with crosscap number
  $\geq 3$. Then, the mixed invariant
  $\MixedInvt{F}\co\BNh^-(L_0)\to\BNh^+(L_1)$ is independent of the
  choices in its construction, up to an overall sign. Further, If $F$
  is isotopic to $F'$ or, more generally, if there is a
  self-diffeomorphism of $[0,1]\times S^3$ which is the identity on
  the boundary and sends $F$ to $F'$ then
  $\MixedInvt{F}=\pm\MixedInvt{F'}\co \BNh^-(L_0)\to\BNh^+(L_1)$.
\end{theorem}
\begin{proof}
  We will show (in order):
  \begin{enumerate}[label=(\arabic*)]
  \item\label{item:inv-movie} Independence of the choice of movies
    compatible with a fixed admissible cut and, in particular, of
    isotopies of $F_0$ and $F_1$ rel boundary, and of the choice of diffeomorphism $\Phi$ in
    Definition~\ref{def:compat-with-cut}. 
  \item\label{item:inv-cut} Invariance under elementary equivalences
    of admissible cuts.
  \item\label{item:inv-cob} Invariance under diffeomorphisms of
    surfaces and admissible cuts.
  \end{enumerate}
  By Proposition~\ref{prop:admis-cut}, this implies the result.
  
  Throughout the proof, ``equal'' or ``homotopic'' will mean equal or homotopic up to an overall sign.

  For point~\ref{item:inv-movie}, by
  Proposition~\ref{prop:BNh-functorial}, difference choices of movies
  compatible with the same admissible cut give chain homotopic maps
  $\BNcx^-(F_0)$ and $\BNcx^-(F_1)$. (For independence of $\Phi$, this
  uses the last statement in
  Proposition~\ref{prop:BNh-functorial}.) If $\BNcx^-(F_0)\sim \BNcx^-(F_0')$, then
  $\BNh^-(F_0)=\BNh^-(F'_0)$; similarly, if
  $\BNcx^-(F_1)\sim \BNcx^-(F_1')$, then
  $\BNcx^+(F_1)\sim \BNcx^+(F_1')$, and hence
  $\BNh^+(F_1)=\BNh^+(F'_1)$. Notice that the lift
  $\BNh(F_0)\co \BNh^-(L_0)\to\BNh^{\red}(L)$ of $\BNh^-(F_0)$ is
  canonical: $\BNh^{\red}(L)$ is a canonical submodule of
  $\BNh^-(L)$. Similarly, the induced map
  $\BNh(F_1)\co \BNh^{\red}(L)\to \BNh^+(L_1)$ is canonical, since
  $\BNh^{\red}(L)$ is a canonical quotient module of $\BNh^+(L)$. So,
  different choices of movies for $F_0$ and $F_1$ give the same mixed
  invariant.
  
  For point~\ref{item:inv-cut}, if the
  admissible cuts $(S,V,\phi)$ and $(S',V',\phi')$ are elementary
  equivalent, let $L=\phi(S\cap F)$ and $L'=\phi'(S\cap F)$. Assume,
  without loss of generality, that we are in the first case of
  Formula~\eqref{eq:admis-cut-equiv}. Choose a movie for $F$
  compatible with this decomposition (in a sense analogous to
  Definition~\ref{def:compat-with-cut}), so $L$ and $L'$ are
  frames in the movie.
  Then, it follows from commutativity of the diagram
  \[
    \begin{tikzpicture}[xscale=3,yscale=2.4]
      \node (plus-0) at (0,2) {$\BNh^+(L;R)$};
      \node (minus-0) at (0,1) {$\BNh^-(L;R)$};
      \node (red-0) at (0,1.5) {$\BNh^\red(L;R)$};
      
      \node (plus-1) at (1,2) {$\BNh^+(L';R)$};
      \node (minus-1) at (1,1) {$\BNh^-(L';R)$};
      \node (red-1) at (1,1.5) {$\BNh^\red(L';R)$};
      
      \draw[->] (plus-0) -- (red-0);
      \draw[->] (plus-1) -- (red-1);
      \draw[->] (red-0) -- (minus-0);
      \draw[->] (red-1) -- (minus-1);
      \draw[->] (plus-0) -- (plus-1);
      \draw[->] (minus-0) -- (minus-1);
      \draw[->] (red-0) -- (red-1);
    \end{tikzpicture}
  \]
  (and point~\ref{item:inv-movie}) that the mixed invariants with
  respect to $(S,V,\phi)$ and $(S',V',\phi')$ agree.
  
  Finally, for point~\ref{item:inv-cob}, suppose an admissible cut
  $(S,V,\phi)$ for $F$ is diffeomorphic to an admissible cut
  $(S',V',\phi')$ for $F'$, via a diffeomorphism $\Psi$. Fix a
  movie for $F$ compatible with $(S,V,\phi)$, with respect to a
  diffeomorphism $\Phi$ extending $\phi$. Then, the same movie is
  compatible with $(S',V',\phi')$, via the diffeomorphism
  $\Phi\circ\Psi$. By points~\ref{item:inv-movie} and~\ref{item:inv-cut}, we can compute the mixed
  invariants of $F$ and $F'$ using these movies, so the mixed
  invariants agree.
\end{proof}

\section{Properties}\label{sec:properties}
\subsection{First observations}\label{sec:first-obs}
We start by noting the mixed invariant's grading:
\begin{lemma}\label{lem:mixed-grading}
  The mixed invariant $\MixedInvt{F}\co \BNh^-(L_0)\to\BNh^+(L_1)$
  increases the bigrading by $(-1-e/2,\chi-3e/2-2s)$, where $\chi$ is the
  Euler characteristic of $F$, $e$ is the normal Euler number of
  $F$, and $s$ is the number of stars on $F$.
\end{lemma}

\begin{proof}
  This follows immediately from Corollary~\ref{cor:hq-gr-change}. The
  additional downward grading shift by $(1,0)$ comes from the
  identification of $\BNh^{\red}(L)$ as a quotient module of
  $\BNh^+(L)$.
\end{proof}

There is a simple condition guaranteeing the mixed invariant vanishes:
\begin{lemma}\label{lem:simple-vanishing}
  If $F$ has an admissible cut $S$ so that the link $L\subset S^3$
  corresponding to $S\cap F$ has $\BNh^\red(L)=0$, then the
  mixed invariant $\MixedInvt{F}$ vanishes.
\end{lemma}

\begin{proof}
  This is immediate from Definition~\ref{def:mixed-invt}.
\end{proof}

\begin{remark}
  The analogue of Lemma~\ref{lem:simple-vanishing} for Heegaard Floer
  homology is factoring through an $L$-space. Note, however, that
  $L$-spaces seem to be much more common than links with vanishing
  $\BNh^\red$.
\end{remark}

The mixed invariant behaves simply with respect to (a particular kind
of) mirroring. Let $F\co L_0\to L_1$ be a cobordism, so $F$ is
smoothly embedded in $[0,1]\times S^3$. Applying the orientation-preserving
diffeomorphism $[0,1]\times S^3\to [0,1]\times S^3$,
$(t,x,y,z,w)\mapsto (1-t,-x,y,z,w)$ gives a new cobordism
$m(F)\co m(L_1)\to m(L_0)$, where $m(L_i)$ denotes the mirror of
$L_i$.

The statement is a little simpler for the Lee deformation than the
Bar-Natan deformation, so we separate the two cases. Given an
$\Ring[T]$-module $M$, $\Hom_{\Ring}(M,\Ring)$ inherits the structure
of an $\Ring[T]$-module, as well.
\begin{lemma}\label{lem:mirror-Lee}
  Let $\BNcx^\pm$ denote the Lee deformation. 
  Given a link $L$, there is an isomorphism $\BNcx^+(m(L))\cong
  \Hom_{\Ring}(\BNcx^-(L),\Ring)$ of complexes over $\Ring[T]$ so that for any cobordism $F\co L_0\to
  L_1$,
  \[
    \BNcx^+(m(F))\co \BNcx^+(m(L_1))\cong
    \Hom_{\Ring}(\BNcx^-(L_1),\Ring)\to
    \Hom_{\Ring}(\BNcx^-(L_0),\Ring)\cong\BNcx^+(m(L_0))
  \]
  is the dual to the map $\BNcx^-(F)\co \BNcx^-(L_0)\to\BNcx^-(L_1)$. 

  Further, if $F$ has crosscap number $\geq 3$ and $\Ring$ is a field
  then the mixed invariant $\MixedInvt{m(F)}$ is given by the
  composition
  \[
    \BNh^-(m(L_1))\cong \Hom(\BNh^+(L_1),\Ring)\stackrel{\MixedInvt{F}^*}{\lra}\Hom(\BNh^-(L_0),\Ring)\cong\BNh^+(m(L_0)).
  \]
  In particular, $\MixedInvt{m(F)}$ vanishes if and only if
  $\MixedInvt{F}$ does.
\end{lemma}
(Compare~\cite[Proposition
32]{Kho-kh-categorification},~\cite[pp. 184--185]{Kho-kh-Frobenius},
and~\cite[Proposition 3.1]{CMW-kh-functoriality}.)
\begin{proof}
  The isomorphism $\mu\co \Hom(\BNcx^-(L),\Ring)\to \BNcx^+(m(L))$
  is defined as follows. Given a generator $T^n(v,y)$ of $\BNcx^-(L)$
  (over $\Ring$), let $[T^n(v,y)]^*$ denote the dual generator of
  $\Hom(\BNcx^-(L),\Ring)$. The isomorphism is given by
  $\mu\bigl([T^n(v,y)]^*)=T^{-n-1}(\vec{1}-v,y^*)$ where
  $\vec{1}-v=(1-v_1,\dots,1-v_c)$ and 
  $y^*$ is the result of reversing the label of every circle. (That
  is, if $y$ labels a circle $Z$ by $X$ then $y^*$ labels the
  corresponding circle by $1$, and vice-versa.)  It is straightforward
  to check that this defines a chain isomorphism.

  A movie for $F$ induces a movie for $m(F)$ by mirroring each frame
  and reversing the order of the frames. So, it suffices to check the
  second statement for a single elementary cobordism (pair of adjacent
  frames in a movie). This is a straightforward case
  check.

  For the statement about the mixed invariant, a choice of admissible
  cut for $F$, along some link $L$, and movie compatible with it induce an admissible cut
  for $m(F)$, along $m(L)$, and a movie compatible with it. The map
  $\mu$ induces an isomorphism of short exacts sequences
  \[
   \mathcenter{\begin{tikzpicture}[xscale=1.65]
      \node at (.75,0) (tl0) {$0$};
      \node at (2,0) (Cm) {$\Hom(\BNcx^+(L),\Ring)$};
      \node at (4,0) (Ci) {$\Hom(\BNcx^\infty(L),\Ring)$};
      \node at (6,0) (Cpt) {$\Hom(\BNcx^-(L),\Ring)$};
      \node at (7.25,0) (tr0) {$0$};
      \node at (0.75,-1) (bl0) {$0$};
      \node at (2,-1) (Ch) {$\BNcx^-(m(L))$};
      \node at (4,-1) (Cpb1) {$\BNcx^\infty(m(L))$};
      \node at (6,-1) (Cpb2) {$\BNcx^+(m(L))$};
      \node at (7.25,-1) (br0) {$0$};
      \draw[->] (tl0) to (Cm);
      \draw[->] (Cm) to (Ci);
      \draw[->] (Ci) to (Cpt);
      \draw[->] (Cpt) to (tr0);
      \draw[->] (bl0) to (Ch);
      \draw[->] (Ch) to (Cpb1);
      \draw[->] (Cpb1) to  (Cpb2);
      \draw[->] (Cpb2) to (br0);
      \draw[->] (Cm) to (Ch);
      \draw[->] (Ci) to (Cpb1);
      \draw[->] (Cpt) to (Cpb2);
    \end{tikzpicture}}
   \]
   so naturality of the snake lemma implies that the diagram
   \[
   \mathcenter{\begin{tikzpicture}[xscale=1.75, yscale=1.1]
      \node at (0,0) (Cmdual) {$\Hom(\BNh^-(L),\Ring)$};
      \node at (4,0) (Cpdual) {$\Hom(\BNh^+(L),\Ring)$};
      \node at (0,-3) (Cp) {$\BNh^+(m(L))$};
      \node at (4,-3) (Cm) {$\BNh^-(m(L))$};
      \node at (2,-1) (reddual) {$\Hom(\BNh^{\red}(L),\Ring)$};
      \node at (2,-2) (redu) {$\BNh^{\red}(m(L))$};
      \draw[->] (Cmdual) to node[above]{\lab{\bdy^*}} (Cpdual);
      \draw[->] (Cp) to node[above]{\lab{\bdy}} (Cm);
      \draw[->] (Cmdual) to node[left]{\lab{\cong}} (Cp);
      \draw[->] (Cpdual) to node[right]{\lab{\cong}} (Cm);
      \draw[->] (Cp) to (redu);
      \draw[->] (redu) to (Cm);
      \draw[->, dashed] (reddual) to (redu);
      \draw[->] (Cmdual) to (reddual);
      \draw[->] (reddual) to (Cpdual);
    \end{tikzpicture}}
   \]
   commutes. Combining this with the definition of the mixed invariant
   in Diagram~\eqref{eq:mixed-invt-diagram} gives the result.
\end{proof}

For the analogous results for the Bar-Natan complex, there is an extra
sign. There is a ring automorphism $\sigma\co \Ring[H]\to\Ring[H]$
induced by $\sigma(H)=-H$. Given a module $M$ over $\Ring[H]$, let
$M_\sigma$ be the result of restricting (or extending) scalars by
$\sigma$. That is, $M_\sigma$ is the same as $M$ except $H$ acts on
$M_\sigma$ the way $-H$ acts on $M$. Given another
$\Ring[H]$-module $N$ and a homomorphism $f\co M\to N$, there is an
induced homomorphism $f\co M_\sigma\to N_\sigma$ (which, as a map of
sets, is the same as $f\co M\to N$).

Then, the following is the analogue of Lemma~\ref{lem:mirror-Lee}:
\begin{lemma}
  Let $\BNcx^\pm$ denote the Bar-Natan deformation.  Given a link $L$, there
  is an isomorphism
  $\BNcx^+(m(L))_\sigma\cong \Hom_{\Ring}(\BNcx^-(L),\Ring)$ of complexes
  over $\Ring[H]$ so that for any cobordism $F\co L_0\to L_1$,
  \[
    \BNcx^+(m(F))\co \BNcx^+(m(L_1))_\sigma\cong
    \Hom_{\Ring}(\BNcx^-(L_1),\Ring)\to
    \Hom_{\Ring}(\BNcx^-(L_0),\Ring)\cong\BNcx^+(m(L_0))_\sigma
  \]
  is the dual to the map $\BNcx^-(F)\co \BNcx^-(L_0)\to\BNcx^-(L_1)$.   

  Further, if $F$ has crosscap number $\geq 3$ and $\Ring$ is a field
  then the mixed invariant $\MixedInvt{m(F)}$, viewed as a map of
  modules twisted by $\sigma$, is given by the composition
  \[
    \BNh^-(m(L_1))_\sigma\cong \Hom(\BNh^+(L_1),\Ring)\stackrel{\MixedInvt{F}^*}{\lra}\Hom(\BNh^-(L_0),\Ring)\cong\BNh^+(m(L_0))_\sigma.
  \]
  In particular, $\MixedInvt{m(F)}$ vanishes if and only if
  $\MixedInvt{F}$ does.
\end{lemma}
\begin{proof}
  The isomorphism sends $[H^n(v,y)]^*$ to
  $(-1)^nH^{-n-1}(\vec{1}-v,y^*)$. The rest of the proof is a
  straightforward adaptation of the Lee case.
\end{proof}

\begin{remark}
  If $\Ring$ is a field, it follows from the classification of modules
  over a PID that there is a (perhaps unnatural) isomorphism over
  $\Ring[H]$ between $\BNh^+(m(L))$ and $\Hom(\BNh^+(L),\Ring)$.
\end{remark}

\begin{remark}
  There is another mirror one might consider: the map
  $(t,x,y,z,w)\mapsto (t,-x,y,z,w)$ which mirrors each frame in the
  movie but does not reverse the order of the frames. Neither
  $\BNh^-(F)$ nor the mixed invariant seems to behave simply with
  respect to this operation, as the example in
  Section~\ref{sec:direct-comp} shows. (Gauge-theoretic invariants of
  the branched double cover also do not behave well with respect to
  this operation.)
\end{remark}

The mixed invariant also respects composition, as follows:
\begin{lemma}\label{lem:composition}
  Let $F_0\co L_0\to L_1$, $F_1\co L_1\to L_2$, and $F_2\co L_2\to L_3$
  be cobordisms, so that $F_1$ has crosscap number $\geq 3$. Then,
  \begin{align}
    \MixedInvt{F_2\circ F_1}&=\BNh^+(F_2)\circ \MixedInvt{F_1}\\
    \MixedInvt{F_1\circ F_0}&=\MixedInvt{F_1}\circ\BNh^-(F_0).
  \end{align}
  The same result holds if $F_1$ has crosscap number $2$, as long as
  either $F_1\circ F_0$ and $F_2\circ F_1$ have crosscap number $\geq
  3$, and we define $\MixedInvt{F_1}$ with respect to any choice of admissible cut for $F_1$
  (compare \ Remark~\ref{rem:2-crosscaps}).
\end{lemma}
\begin{proof}
  This is immediate from the definitions.
\end{proof}

There is an easy criterion for the mixed invariant to be
non-vanishing, in terms of the induced map on ordinary Khovanov
homology $\wh{\BNh}$:
\begin{lemma}\label{lem:hat}
  Let $F\co L_0\to L_1$ be a cobordism with crosscap number
  $\geq 3$.  Then, the following diagrams commute:
  \[
    \mathcenter{\begin{tikzpicture}[xscale=3.5,yscale=1.2]
      \node at (0,0) (Hm) {$\BNh^-(L_0)$};
      \node at (1,0) (Hp) {$\BNh^+(L_1)$};
      \node at (0,-1) (Hh0) {$\wh{\BNh}(L_0)$};
      \node at (1,-1) (Hh1) {$\wh{\BNh}(L_1)$};
      \draw[->] (Hm) to node[above]{\lab{\MixedInvt{F}}} (Hp);
      \draw[->] (Hh0) to node[above]{\lab{\widehat{\BNh}(F)}} (Hh1);
      \draw[->] (Hm) to node[left]{\lab{\pi_*}} (Hh0);
      \draw[->] (Hp) to node[right]{\lab{\bdy}} (Hh1);
    \end{tikzpicture}}
  \text{\quad and \quad}
    \mathcenter{\begin{tikzpicture}[xscale=3.5,yscale=1.2]
      \node at (0,0) (Hm) {$\BNh^-(L_0)$};
      \node at (1,0) (Hp) {$\BNh^+(L_1).$};
      \node at (0,1) (Hh0) {$\wh{\BNh}(L_0)$};
      \node at (1,1) (Hh1) {$\wh{\BNh}(L_1)$};
      \draw[->] (Hm) to node[above]{\lab{\MixedInvt{F}}} (Hp);
      \draw[->] (Hh0) to node[above]{\lab{\widehat{\BNh}(F)}} (Hh1);
      \draw[->] (Hh0) to node[left]{\lab{\bdy}} (Hm);
      \draw[->] (Hh1) to node[right]{\lab{\iota_*}} (Hp);
    \end{tikzpicture}}
  \]
  In particular, if $\wh{\BNh}(F)\circ\pi_*$ or
  $\iota_*\circ \wh{\BNh}(F)$ is non-zero then the mixed
  invariant $\MixedInvt{F}$ is also non-zero.
\end{lemma}
\begin{proof}
  Let $L$ be an admissible cut for $F$, decomposing $F$ as
  $F_1\circ F_0$.
  Define the map $\bdy\co \BNh^{\red}(L)\to \wh{\BNh}(L)$ to be the
  composition
  $\BNh^\red(L)\to\BNh^-(L)\stackrel{\pi_*}{\too} \wh{\BNh}(L)$. Using
  the first commutative triangle in Formula~(\ref{eq:bdybdy-tri}), the
  map $\bdy\from \BNh^+(L)\to\wh{\BNh}(L)$ is the composition
  $\BNh^+(L)\to \BNh^\red(L)\to\BNh^-(L)\stackrel{\pi_*}{\too}
  \wh{\BNh}(L)$, so is also the composition
  $\BNh^+(L)\to \BNh^\red(L)\stackrel{\bdy}{\too} \wh{\BNh}(L)$.
  
  To see that $\bdy\circ \Phi_F=\wh{\BNh}(F)\circ\pi_*$, 
  consider the larger diagram
  \[
    \begin{tikzpicture}
      \node at (1,0) (hatL0) {$\wh{\BNh}(L_0)$};
      \node at (0,2) (mL0) {$\BNh^-(L_0)$};
      \node at (3,1.5) (mL) {$\BNh^-(L)$};
      \node at (4,0) (hatL) {$\wh{\BNh}(L)$};
      \node at (4,3) (Hred) {$\BNh^{\red}(L)$};
      \node at (5,1.5) (pL) {$\BNh^+(L)$};
      \node at (7,0) (hatL1) {$\wh{\BNh}(L_1)$};
      \node at (8,2) (pL1) {$\BNh^+(L_1)$};
      \draw[->] (mL0) to node[left]{\lab{\pi_*}} (hatL0);
      \draw[->] (hatL0) to node[above]{\lab{\wh{\BNh}(F_0)}} (hatL);
      \draw[->] (hatL) to node[above]{\lab{\wh{\BNh}(F_1)}} (hatL1);
      \draw[->,sloped] (mL0) to node[below]{\lab{\BNh^-(F_0)}} (mL);
      \draw[->, dashed] (mL0) to node[above,sloped]{\lab{\BNh(F_0)}} (Hred);
      \draw[->, dashed] (Hred) to node[above,sloped]{\lab{\BNh(F_1)}} (pL1);
      \draw[->] (Hred) to (mL);
      \draw[->] (pL) to (Hred);
      \draw[->] (mL) to node[left]{\lab{\pi_*}} (hatL);
      \draw[->] (pL) to node[right]{\lab{\bdy}} (hatL);
      \draw[->,sloped] (pL) to node[below]{\lab{\BNh^+(F_1)}} (pL1);
      \draw[->] (pL1) to node[right]{\lab{\bdy}} (hatL1);
      \draw[->] (Hred) to node[left]{\lab{\bdy}} (hatL);
    \end{tikzpicture}
  \]
  The middle triangles commute by the discussion above. The outer
  squares commute by naturality of the long exact
  sequences~\eqref{eq:les}, Lemma~\ref{lem:les-natural}. The triangles at the top commute by the
  definition of the dashed lifts. Since the map $\BNh^+(L)\to
  \BNh^{\red}(L)$ is surjective, commutativity of the right square and
  triangles implies that $\bdy\circ \BNh(F_1)=\wh{\BNh}(F_1)\circ
  \bdy\co \BNh^{\red}(L)\to\wh{\BNh}(L_1)$. Commutativity of the
  square and two triangles on the left then implies the result.

  The proof that $\iota_*\circ \wh{\BNh}(F)=\Phi_F\circ\bdy$ is
  similar, but instead uses the commutative diagram
  \[
    \begin{tikzpicture}
      \node at (1,0) (hatL0) {$\wh{\BNh}(L_0)$};
      \node at (0,-2) (mL0) {$\BNh^-(L_0)$};
      \node at (3,-1.5) (mL) {$\BNh^-(L)$};
      \node at (4,0) (hatL) {$\wh{\BNh}(L)$};
      \node at (4,-3) (Hred) {$\BNh^{\red}(L)$};
      \node at (5,-1.5) (pL) {$\BNh^+(L)$};
      \node at (7,0) (hatL1) {$\wh{\BNh}(L_1)$};
      \node at (8,-2) (pL1) {$\BNh^+(L_1)$};
      \draw[->] (hatL0)  to node[left]{\lab{\bdy}} (mL0);
      \draw[->] (hatL0) to node[above]{\lab{\wh{\BNh}(F_0)}} (hatL);
      \draw[->] (hatL) to node[above]{\lab{\wh{\BNh}(F_1)}} (hatL1);
      \draw[->,sloped] (mL0) to node[above]{\lab{\BNh^-(F_0)}} (mL);
      \draw[->, dashed] (mL0) to node[below,sloped]{\lab{\BNh(F_0)}} (Hred);
      \draw[->, dashed] (Hred) to node[below,sloped]{\lab{\BNh(F_1)}} (pL1);
      \draw[->] (Hred) to (mL);
      \draw[->] (pL) to (Hred);
      \draw[->]  (hatL) to node[left]{\lab{\bdy}} (mL);
      \draw[->] (hatL) to node[right]{\lab{\iota_*}} (pL);
      \draw[->,sloped] (pL) to node[above]{\lab{\BNh^+(F_1)}} (pL1);
      \draw[->] (hatL1) to node[right]{\lab{\iota_*}} (pL1);
      \draw[->] (hatL)  to (Hred);
    \end{tikzpicture}
  \]
  and the fact that the map $\BNh^{\red}(L)\to \BNh^-(L)$ is injective.
\end{proof}

\begin{remark}\label{rem:mixed-strong}
  Since $\pi_*\co \BNh^-(\emptyset)\to\wh{\BNh}(\emptyset)$ is
  surjective, it follows from Lemma~\ref{lem:hat} that if $F$ is a cobordism from $\emptyset$ to $L$ and
  $\wh{\BNh}(F)\neq 0$ then $\MixedInvt{F}\neq 0$, as well. Similarly,
  if $F$ is a cobordism from $L$ to $\emptyset$ and
  $\wh{\BNh}(F)\neq 0$ then $\MixedInvt{F}\neq 0$.
\end{remark}

\subsection{Stabilizations}
Next, we turn to the behavior of the mixed invariant under various
local changes to the knot. For example, we will study the behavior
under Baykur-Sunukjian's \emph{stabilizations} (attaching arbitrary
$1$-handles to a surface) and \emph{standard stabilizations} (local
connect sums with a standard $T^2$)~\cite{BS-top-stabilizations},
\emph{crosscap stabilizations} (taking local connected sums with a
standard $\RP^2$ or $\overline{\RP}^2$), \emph{local knotting} (taking
a local connected sum with a knotted $S^2$), and more general local
connected sums. The main results are summarized in
Theorem~\ref{thm:vanishing}, though some technical results along the
way (e.g., Corollaries~\ref{cor:star-0} and~\ref{cor:H-vanish}) may
also be of interest.

Most of the results in this section work for all four versions of
Khovanov homology, $\BNh^-$, $\BNh^+$, $\BNh^\infty$, or $\wh{\BNh}$,
so we will use the symbol $\BNh^\bullet$ to denote any of these
four versions.  The key technical property we will use, as usual for
these kinds of arguments, is a neck-cutting relation.

\begin{proposition}\label{prop:neck-cut}
  Let $F\co L_0\to L_1$ be a cobordism, and let $A$ be an arc in
  $[0,1]\times S^3$ with endpoints on a single component $F_0$ of $F$ and
  interior disjoint from $F$. Let $F^\cap$ be the result of attaching a
  $1$-handle to $F$ along $A$ and $F^\star$ the result of adding a star
  to $F$ at one endpoint of $A$. If $F_0$ is orientable, assume that
  the $1$-handle is attached in such a way that the resulting
  component is still orientable. Then, $\BNh^\bullet(F^\cap)=\BNh^\bullet(F^\star)$.
\end{proposition}
\begin{proof}
  For the Lee deformation, this was essentially shown by
  Levine-Zemke~\cite[Proposition 7]{LV-kh-ribbon}, so we focus on the
  case of the Bar-Natan deformation and comment on the Lee deformation
  at the end.

  Let $B$ be an arc in $F$ connecting the endpoints of $A$, chosen so
  that the loop $A\cup B$ is two-sided, i.e., so that $TF^\cap|_{A\cup B}$
  is trivial. (If $F_0$ is orientable, any arc $B$ has $TF^\cap|_{A\cup B}$
  trivial by the assumption that orientability was preserved; if $F_0$ is nonorientable, given any arc $B$
  connecting the endpoints of $A$ we can
  modify $B$ by taking the connected sum with a one-sided loop to
  achieve this property.) Let $H=F^\cap\setminus F$ denote the new
  $1$-handle. Perform an isotopy to $F^\cap$ so that:
  \begin{itemize}
  \item The projection $[0,1]\times S^3\to [0,1]$ restricts to a Morse
    function $f$ on $F^\cap$, and induces a movie decomposition of $F^\cap$.
  \item The restriction $f|_H$ has two critical points, both of index
    $1$, corresponding to a pair of saddles in the movie.
  \item The two saddles corresponding to $f|_H$ in the movie are
    adjacent, happening at times $t+\epsilon,t+2\epsilon\in (0,1)$, and there
    are no other elementary cobordisms between $t-\epsilon$ and
    $t+3\epsilon$. Further, these frames are obtained by gluing the
    following local model to the identity cobordism of the rest of the
    link:
      \begin{equation}\label{eq:neck-tangle}
    \mathcenter{\begin{tikzpicture}[scale=0.8]
      \draw[dashed] (0,0) circle (1);
      \draw[bend right=45,knot] (45:1) to (.7071,-.7071);
      \draw[bend left=45,knot] (-.7071,.7071) to (-.7071,-.7071);
    \end{tikzpicture}}
  \longrightarrow
  \mathcenter{
    \begin{tikzpicture}[scale=0.8]
      \draw[dashed] (0,0) circle (1);
      \draw[bend left=45,knot] (.7071,.7071) to (-.7071,.7071);
      \draw[bend right=45,knot] (.7071,-.7071) to (-.7071,-.7071);
    \end{tikzpicture}
    }
    \longrightarrow
    \mathcenter{
    \begin{tikzpicture}[scale=0.8]
      \draw[dashed] (0,0) circle (1);
      \draw[bend right=45,knot] (.7071,.7071) to (.7071,-.7071);
      \draw[bend left=45,knot] (-.7071,.7071) to (-.7071,-.7071);
    \end{tikzpicture}
    }
  \end{equation}
  \item The arc $B$ satisfies $f(B)=t$.
  \end{itemize}
  (To arrange this, first isotope $F^\cap$ so that $H$ is small, then make
  $H$ standard with respect to the projection to $[0,1]$, and then
  isotope $B$ to lie in the desired level set and use the isotopy
  extension lemma to push the rest of $F^\cap$ out of the way.) 

  Given a dotted (not starred) cobordism, decomposed as a movie, there
  is a corresponding map of Bar-Natan complexes, where the map associated to
  a dot is multiplication by $X$. Let $F^{\smbullet,0}$ and
  $F^{\smbullet,1}$ be the result of placing a dot on $F$ at each endpoint of $A$.
  Then, an easy local computation shows that
  \begin{equation}\label{eq:neck-cut-helper}
    \pm\BNh^\bullet(F^\cap)=\BNh^\bullet(F^{\smbullet,0})+\BNh^\bullet(F^{\smbullet,1})-H\BNh^\bullet(F).
  \end{equation}
  The surface $F^{\smbullet,0}$ can be transformed into $F^{\smbullet,1}$
  by moving the dot along the arc $B$. Since $B$ is contained in a
  single level, this corresponds to moving the dot along an arc in a
  single link diagram in the movie.

  Let $F^{\smbullet,(i)}$ be the surface after we have moved the dot
  through $i$ crossings.  By Alishahi's lemma~\cite[Lemma 2.2]{Ali-kh-unknotting},
  \[
    \BNh^\bullet(F^{\smbullet,(i)})=H\BNh^\bullet(F)-\BNh^\bullet(F^{\smbullet,(i+1)})
  \]
  So, it suffices to show that the arc $B$ has an even number of
  crossings on it: then
  $\BNh^\bullet(F^{\smbullet,0})=\BNh^\bullet(F^{\smbullet,1})$ so the
  right side of Formula~\eqref{eq:neck-cut-helper} is equal to
  $\BNh^\bullet(F^\star)$.

  This is where the assumption that $A\cup B$ is two-sided is
  used. Orient the arc $B$. This induces an orientation of the arcs in
  the first frame of~\eqref{eq:neck-tangle}. The fact that
  $TF^\cap|_{A\cup B}$ is trivial implies that this orientation is
  compatible with the saddle cobordisms in~\eqref{eq:neck-tangle};
  that is, $B$ connects the bottom-left endpoint to the bottom-right
  one, or the top-left endpoint to the top-right one. Without loss of
  generality, assume $B$ runs from the top-left endpoint to the
  top-right one. Choose a checkerboard coloring of the left link
  diagram so that the region between the arcs shown is black. Then, $B$
  starts with a black region to its right, and ends with a black
  region to its right. Each time $B$ passes through a crossing, the
  black region switches between the left and right side of $B$. So,
  $B$ passes through an even number of crossings, as claimed.

  The proof for the Lee case is the same, except that the analogue of
  Equation~\eqref{eq:neck-cut-helper} does not have the $H\BNh(F)$
  term, and Alishahi's lemma is replaced by Hedden-Ni's \cite[Lemma
  2.3]{HN-kh-detects}.
\end{proof}

The standard $\RP^2$ inside the $4$-ball is the surface represented by
the following movie:
\begin{equation}\label{eq:std-rp2}
  \vcenter{\hbox{\begin{tikzpicture}[xscale=2]
        \foreach \i in {1,6}{
          \node (m\i) at (\i,0) {$\varnothing$};
        }

        \foreach \i in {2,5}{
          \node (m\i) at (\i,0) {\begin{tikzpicture}[scale=0.5]
              \draw[knot] (0,0) circle (1);
            \end{tikzpicture}};
        }

        \node (m3) at (3,0) {\begin{tikzpicture}[scale=0.5]
            \draw[knot] (45:1) arc (45:315:1);
            \draw[knot] (45:1) to [out=-45,in=-90] (-0.7,0);
            \draw[knot] (-45:1) to [out=45,in=90] (-0.7,0);
          \end{tikzpicture}};

        \node (m4) at (4,0) {\begin{tikzpicture}[scale=0.5]
            \draw[knot] (45:1) arc (45:135:1);
            \draw[knot] (-45:1) arc (-45:-135:1);
            \draw[knot] (45:1) to [out=-45,in=135] (225:1);
            \draw[knot] (-45:1) to [out=45,in=-135] (-225:1);
          \end{tikzpicture}};

        \foreach \i/\l in {1/b,2/RI,3/s,4/RI,5/d}{
          \pgfmathtruncatemacro{\j}{\i+1}%
          \node (a\i) at ($(m\i)!0.5!(m\j)$) {$\too$};
          \node [anchor=south] at (a\i) {\tiny $\l$};
        }
          
  \end{tikzpicture}}}
\end{equation}
The standard $\RP^2$ has $e=-2$.  The standard $\overline{\RP}^2$ is
the mirror of the above, and has $e=2$. (Our conventions are chosen to
agree with~\cite{FKV-top-knot-surf}.)
Define a \emph{crosscap stabilization} to be the result of taking the connected sum
with a standard $\RP^2$ or $\overline{\RP}^2$.

\begin{lemma}\label{lem:cc-stab-vanish}
  If $F^\otimes$ is obtained from $F$ by a crosscap stabilization then
  $\BNh^\bullet(F^\otimes)$ vanishes.
\end{lemma}
\begin{proof}
  A movie for $F^\otimes$ is obtained from a movie for $F$ by taking the
  disjoint union with the movie in Formula~\eqref{eq:std-rp2} or its
  mirror, but replacing $d$ with a saddle map between the unknot shown
  and $F$. It is straightforward to see that the map on $\BNh^\bullet$
  induced by $s\circ RI\circ b$ vanishes (as does the map associated
  to the mirror of this movie), so the map associated to the whole
  movie vanishes, as well.
\end{proof}

\begin{corollary}\label{cor:star-0}
  Let $F\co L_0\to L_1$ be a cobordism, and suppose that some
  nonorientable component $F_0$ of $F$ has a star on it. Then,
  $\BNh^\bullet(F)=0$. Similarly, if $F$ is obtained from another surface by
  attaching a $1$-handle to a nonorientable component then $\BNh^\bullet(F)=0$.
\end{corollary}
\begin{proof}
  For the first statement, let $F^\curlywedge$ be the result of taking the
  connected sum of $F$ with a local Klein bottle (with normal Euler
  number $0$) at the star on $F_0$, and forgetting the star. That is,
  $F^\curlywedge$ is obtained from $F$ by attaching a local $1$-handle with both
  feet near the star, in a locally-nonorientable way (and forgetting
  the star). By Proposition~\ref{prop:neck-cut},
  $\BNh^\bullet(F)=\BNh^\bullet(F^\curlywedge)$. On the other hand, $F^\curlywedge$ is also obtained
  from $F$ by taking the connect sum with $\RP^2\#\overline{\RP}^2$
  (and forgetting the star). So, by Lemma~\ref{lem:cc-stab-vanish},
  $\BNh^\bullet(F^\curlywedge)$ vanishes.

  The second statement follows from the first and
  Proposition~\ref{prop:neck-cut}.
\end{proof}

Adding an even number of stars has a predictable effect on the cobordism
maps and the mixed invariant:
\begin{lemma}\label{lem:starstar}
  Let $F\co L_0\to L_1$ be a cobordism, and let $F^{\star\star}$ be the result of adding two stars to the same component of $F$. Then,
  \[
    \BNh^\bullet(F^{\star\star})=
    \begin{cases}
      4T\BNh^\bullet(F) & \text{for the Lee deformation}\\
      H^2\BNh^\bullet(F) & \text{for the Bar-Natan deformation.}
    \end{cases}
  \]
  Further, if the crosscap number of $F$ is at least $3$ then
  \[
    \MixedInvt{F^{\star\star}}=
    \begin{cases}
      4T\MixedInvt{F} & \text{for the Lee deformation}\\
      H^2\MixedInvt{F} & \text{for the Bar-Natan deformation.}
    \end{cases}
  \]
\end{lemma}
\begin{proof}
  Since the maps are invariant under isotopy of the cobordisms, we can
  arrange that the two elementary star cobordisms are adjacent. Then,
  the result is immediate from the definition of the map associated to
  an elementary star cobordism.
\end{proof}

\begin{corollary}\label{cor:H-vanish}
  If $F$ is nonorientable then for the Lee deformation,
  $4T\BNh^\bullet(F)=0$, and for the Bar-Natan deformation,
  $H^2\BNh^\bullet(F)=0$. If $F$ has crosscap number at least $3$
  then additionally, $4T\MixedInvt{F}=0$ and
  $H^2\MixedInvt{F}=0$ for the Lee and Bar-Natan deformations,
  respectively.
\end{corollary}
\begin{proof}
  This is immediate from Corollary~\ref{cor:star-0} and Lemma~\ref{lem:starstar}.
\end{proof}

We note a very mild extension of a result of
Rasmussen~\cite{Rasmussen-kh-closed} and
Tanaka~\cite{Tanaka-kh-closed}:
\begin{lemma}\label{lem:closed-surf}
  If $F\subset [0,1]\times S^3$ is a closed, connected, orientable
  surface of genus $g$ with $s$ stars then $\BNh^\bullet(F)$ is
  \begin{itemize}
  \item $0$ if $g+s$ is even,
  \item $2H^{g+s-1}$ if $g+s$ is odd and we are considering the
    Bar-Natan deformation, and
  \item $2^{g+s}T^{\frac{g+s-1}{2}}$ if $g+s$ is odd and we are
    considering the Lee deformation.
  \end{itemize}
  If $F\subset [0,1]\times S^3$ is a closed, connected, nonorientable
  surface (possibly with stars) then $\BNh^\bullet(F)=0$.
\end{lemma}
(In both cases, the surface may be knotted.)
\begin{proof}
  We start with the orientable case. Any such surface becomes isotopic
  to a standardly-embedded one after attaching some number of
  $1$-handles (see, e.g.,~\cite[Theorem 1]{BS-top-stabilizations}). By
  adding an extra one if necessary, we may assume the number of
  $1$-handles added is even, say $2k$. By
  Proposition~\ref{prop:neck-cut}, adding these $1$-handles has the
  same effect as adding the same number of stars which, by
  Lemma~\ref{lem:starstar}, multiplies $\BNh^\bullet(F)$ by $(4T)^k$
  or $H^{2k}$.  Considering first $\BNh^-(F)$,
  multiplication by $(4T)^k$ or $H^{2k}$ is an injective map
  $\Ring[H]\to \Ring[H]$ or $\Ring[T]\to\Ring[T]$, so the
  element $\BNh^-(F)$ depends only on $g$ and $s$, not the embedding
  of $F$. Thus, the result for $\BNh^-(F)$ follows from an easy model computation for a
  standardly-embedded surface (which can be made even easier by
  applying Proposition~\ref{prop:neck-cut} to trade the genus for
  stars and then applying Lemma~\ref{lem:starstar}). The results for
  the other versions---$\BNh^\infty(F)$, $\BNh^+(F)$, and
  $\wh{\BNh}(F)$---follow formally from the case of $\BNh^-(F)$ and
  the natural long exact sequences~(\ref{eq:les}).

  The proof for the nonorientable case is essentially the same. By a
  result of Baykur-Sunukjian \cite[Theorem 6]{BS-top-stabilizations},
  after a finite number of stabilizations, $F$ becomes isotopic to a
  connected sum of copies of the standard $\RP^2$ and
  $\overline{\RP}^2$. A straightforward model computation, or
  Lemma~\ref{lem:cc-stab-vanish}, shows the map associated to a
  connected sum of copies of the standard $\RP^2$ or
  $\overline{\RP}^2$ vanishes, so by Proposition~\ref{prop:neck-cut}
  the map associated to $F$ vanishes, as well.
\end{proof}

Given a link cobordism $F\co L_0\to L_1$, a closed surface
$E\subset S^4$, and points $p\in E$ and $q\in F$, there is a
\emph{standard connected sum} of $F$ and $E$ at $q$ and $p$; this
is the connected sum of pairs $([0,1]\times S^3,F)\#(S^4,E)$.

\begin{proposition}\label{prop:cyl-sum}
  Let $F\co L_0\to L_1$ be a cobordism and $E\subset S^4$ a
  closed, connected surface with no stars on it. Let
  $F^\#\defeq F\# E$ be a standard connected sum of $F$ and
  $E$, and let $F^\star$ be the result of adding a star to $F$,
  on the component where the connect sum is occurring. Then,
  \begin{enumerate}
  \item\label{item:BNh-or} If $E$ is an orientable surface of genus
    $g>0$ then for the Lee deformation,
    \begin{equation}
      \BNh^\bullet(F^\#)=
      \begin{cases}
        (4T)^{\frac{g}{2}}\BNh^\bullet(F) & \text{$g$ even}\\
        (4T)^{\frac{g-1}{2}}\BNh^\bullet(F^\star) & \text{$g$ odd}
      \end{cases}
    \end{equation}
    while for the Bar-Natan deformation
    \begin{equation}
      \BNh^\bullet(F^\#)=
      \begin{cases}
        H^{g}\BNh^\bullet(F) & \text{$g$ even}\\
        H^{g-1}\BNh^\bullet(F^\star) & \text{$g$ odd}.
      \end{cases}
    \end{equation}
  \item\label{item:BNh-non-or} If $E$ is a nonorientable surface then $\BNh^\bullet(F^\#)=0$.
  \end{enumerate}
\end{proposition}
\begin{proof}
  Let $D$ be a small disk on $E$, so $E\setminus D$ is a cobordism
  from the empty set to the unknot $U$. We will show that
  $\BNh^\bullet(E\setminus D)$ is the same as the invariant of a disk with
  $g$ stars in the orientable case, and vanishes in the nonorientable
  case.  The result then follows from Lemma~\ref{lem:starstar} and
  functoriality, since $F^\#$ is obtained from $F$ by replacing a
  small disk by $E\setminus D$.

  Consider first the version $\BNh^-$ for the Lee deformation, for the
  case that $E$ is orientable. Let $E^\star$ be the result of adding
  a star to $E$. Write
  \[
    \BNh^-(E\setminus D)=p(T)1+q(T)X\in\BNh^-(U)=\Ring[T]\langle 1,X\rangle.
  \]
  By Lemma~\ref{lem:closed-surf} applied to $E$, $q(T)=0$ if $g$ is
  even and $q(T)=2^{g}T^{\frac{g-1}{2}}$ if $g$ is odd. Also,
  \[
    \BNh^-(E^\star\setminus D)=2p(T)X+2q(T),
  \]
  so by Lemma~\ref{lem:closed-surf} applied to $E^\star$, $p(T)=0$
  if $g$ is odd and $p(T)=2^{g}T^{\frac{g}{2}}$ if $g$ is even. So,
  $\BNh^-(E\setminus D)$ is $2^{g}T^{\frac{g}{2}}$ times the
  invariant of a disk if $g$ is even, and $2^{g-1}T^{\frac{g-1}{2}}$
  times the invariant of a disk with a star if $g$ is odd. The results for
  the other versions---$\BNh^\infty$, $\BNh^+$, and
  $\wh{\BNh}$---follow formally from the case of $\BNh^-$ since
  $\BNh^-(U)$ is free over $\BNh^-(\varnothing)$.

  The proof for the Bar-Natan deformation in the orientable case is
  similar.

  Now, suppose $E$ is nonorientable and again, for definiteness,
  consider the Lee deformation. We can again write
  $\BNh^-(E\setminus D)=p(T)1+q(T)X$, but now
  $\BNh^-(E)=\BNh^-(E^\star)=0$. Consequently, $p(T)=q(T)=0$.
\end{proof}

To summarize, both the map $\BNh^\bullet(F)$ and the mixed invariant
$\MixedInvt{F}$ obstruct surfaces being stabilizations and crosscap
stabilizations, and are independent of local knotting.
\begin{theorem}\label{thm:vanishing}
  Let $F\co L_0\to L_1$ be a cobordism.
  \begin{enumerate}[label=(\arabic*)]
  \item\label{item:van-star} If $F$ has at least one star on some
    nonorientable component, then $\BNh^\bullet(F)=0$; and if in addition
    $F$ has crosscap number $\geq3$ then $\MixedInvt{F}=0$, as well.
  \item\label{item:van-stab} If $F$ is a (possibly nonstandard)
    stabilization, obtained from another cobordism $F'$ by attaching a
    handle to some nonorientable component of $F'$, then
    $\BNh^\bullet(F)=0$.
  \item\label{item:van-sum-nonor} If $F$ is obtained from another
    cobordism $F'$ by taking a standard connected sum with a closed,
    nonorientable surface then $\BNh^\bullet(F)=0$; if $F'$ has
    crosscap number $\geq 2$ then $\MixedInvt{F}=0$, as well. In
    particular, this applies if $F$ is a crosscap stabilization of a
    surface (with crosscap number $\geq 2$ in the case of $\MixedInvt{F}=0$).
  \item\label{item:van-sum-or} If $F$ is obtained from another
    cobordism $F'$ by taking a standard connected sum of some
    nonorientable component of $F'$ with a closed, orientable surface
    of genus $g>0$, then $\BNh^\bullet(F)=0$; if $F'$ has crosscap
    number $\geq 2$ then $\MixedInvt{F}=0$, as well. In particular, this
    applies if $F$ is a standard stabilization of a surface (with
    crosscap number $\geq 2$ in the case of $\MixedInvt{F}$) at some nonorientable component.
  \item\label{item:van-knot} If $F$ is obtained from another cobordism
    $F'$ by taking a standard connected sum with a knotted $2$-sphere
    then $\BNh^\bullet(F)=\BNh^\bullet(F')$; if in addition $F$ has crosscap number $\geq 3$
    then $\MixedInvt{F}=\MixedInvt{F'}$.
  \end{enumerate}
\end{theorem}
\begin{proof}
  For $\BNh^\bullet(F)$, Points~\ref{item:van-star}
  and~\ref{item:van-stab} are Corollary~\ref{cor:star-0},
  Points~\ref{item:van-sum-nonor} and~\ref{item:van-knot} are
  Proposition~\ref{prop:cyl-sum}, while Point~\ref{item:van-sum-or} is
  Proposition~\ref{prop:cyl-sum} together with
  Corollaries~\ref{cor:star-0} and~\ref{cor:H-vanish}.

  For $\MixedInvt{F}$, Points~\ref{item:van-sum-nonor},
  \ref{item:van-sum-or}, and~\ref{item:van-knot} follow by the same
  methods, after isotoping the surface so the connect sum happens
  entirely on one side of the admissible cut. For
  Point~\ref{item:van-star}, choose disjoint one-sided embedded curves
  $\gamma,\eta\subset F$ so that $\eta$ contains a star, and then
  consider the admissible cut $S_{\leq\gamma}$, as in the proof of
  Proposition~\ref{prop:admis-cut}; then one side of the admissible
  cut contains a nonorientable component with a star, and so
  $\MixedInvt{F}$ vanishes by Corollary~\ref{cor:star-0}.
\end{proof}

We conclude the section by singling out one consequence of
Point~\ref{item:van-knot}.  Miller-Powell introduced the notion of the
\emph{generalized stabilization distance}~\cite{MP-top-stab} (see also~\cite{Miy-86-stab,JZ-hf-stab}). In
particular, surfaces $F$ and $F'$ have generalized stabilization
distance $0$ if and only if they are related by taking the connected
sums with embedded $2$-spheres. (See also~\cite{SS-kh-surf}.) While
they work in the topological category, we will continue to assume all
surfaces are smoothly embedded.

\begin{corollary}
  If $F$ and $F'$ are cobordisms with $\BNh^\bullet(F)\neq \BNh^\bullet(F')$ or
  $\MixedInvt{F}\neq \MixedInvt{F'}$ (and the cobordisms have crosscap
  number $\geq 3$) then $F$ and $F'$ have generalized stabilization
  distance $>0$.
\end{corollary}

\begin{remark}
  There is also an obstruction to destabilizing orientable surfaces
  from Heegaard Floer homology~\cite[Proposition 5.5]{JZ-hf-clasp}.
\end{remark}

\subsection{Closed surfaces}\label{sec:closed}
The following shows that the mixed invariant is often zero for closed
surfaces (and is always zero for connected closed surfaces). By
contrast, in Section~\ref{sec:comps} we will see that for surfaces
with boundary the mixed invariant does contain interesting
information.

\begin{theorem}\label{thm:cc-3-vanish}
  Let $F$ be a closed surface with crosscap number $\geq 3$,
  normal Euler number $e(F)$, Euler characteristic $\chi(F)$,
  $s_o(F)$ stars on orientable components, and $s_n(F)$ stars on
  nonorientable components. If its mixed invariant $\MixedInvt{F}$ is
  non-zero then $e(F)=-2$, $s_n(F)=0$, and $\chi(F)=1+2s_o(F)$.
\end{theorem}

\begin{corollary}\label{cor:closed-vanish}
  If $F$ is a closed, connected surface with crosscap number $\geq 3$
  then its mixed invariant $\MixedInvt{F}$ vanishes.
\end{corollary}

\begin{proof}[Proof of Theorem~\ref{thm:cc-3-vanish}]
  By Theorem~\ref{thm:vanishing}, $s_n(F)=0$.
  
  Since the mixed invariant $\MixedInvt{F}$ is an $\Ring[U]$-module homomorphism
  $\Ring[U]=\BNh^-(\emptyset)\to
  \BNh^+(\emptyset)=\Ring[U^{-1},U]/\Ring[U]$, $\MixedInvt{F}$ may
  be viewed as an element of $\Ring[U^{-1},U]/\Ring[U]$ (the image of $1$). By
  Lemma~\ref{lem:mixed-grading}, $\MixedInvt{F}$ is in bigrading
  $(-1-e/2,\chi-3e/2-2s_o)$. Since $\BNh^+(\emptyset)$ is supported in homological grading $0$, this forces $e(F)=-2$.

  In the Bar-Natan theory, by Corollary~\ref{cor:H-vanish},
  $\MixedInvt{F}(1)\in\ker(H^2)\subset\Ring[H^{-1},H]/\Ring[H]$, which
  is $\Ring\langle H^{-1},H^{-2}\rangle$, supported in bigradings
  $(0,2)$ and $(0,4)$, which forces $(\chi-2s_o)\in\{-1,1\}$. In the Lee
  theory, again by Corollary~\ref{cor:H-vanish},
  $\MixedInvt{F}(1)\in\ker(4T)\subset \Ring[T^{-1},T]/\Ring[T]$, which
  is $\Ring\langle T^{-1}\rangle$ (recall that $2$ is invertible in
  $\Ring$), supported in bigrading $(0,4)$, forcing $\chi-2s_o=1$.

  Thus, the only case remaining to exclude is $\chi-2s_o=-1$. So, for the
  rest of the proof assume $e(F)=-2$, $\chi(F)-2s_o(F)=-1$, and $s_n(F)=0$.
  We need to show $\MixedInvt{F}=0$. To settle this case, we will
  need to study the mixed invariants over various Frobenius algebras,
  so we will use superscripts to denote the various Frobenius algebras
  that we are working over. We have already observed that the mixed
  invariant in the Lee theory, $\MixedInvt{F}^{\Ring[T,X]/(X^2=T)}$,
  vanishes for grading reasons over any ring $\Ring$ (with $2$
  invertible).

  Consider the Frobenius algebra $\Ring[\sqrt{T},X]/(X^2=T)$ obtained
  from the Lee Frobenius algebra by adjoining a formal square root of
  $T$, as in Proposition~\ref{prop:or-gens-part1}. Since
  $\Ring[\sqrt{T}]$ is free over $\Ring[T]$, all versions of the
  Khovanov chain complexes and homologies over the Frobenius algebra
  $\Ring[\sqrt{T},X]/(X^2=T)$ can be obtained from the corresponding
  versions over the Frobenius algebra $\Ring[T,X]/(X^2=T)$ by
  tensoring with $\Ring[\sqrt{T}]$ over $\Ring[T]$; similarly, the
  maps over the Frobenius algebra $\Ring[\sqrt{T},X]/(X^2=T)$ can be
  obtained from the maps over the Frobenius algebra
  $\Ring[T,X]/(X^2=T)$ by tensoring with $\Ring[\sqrt{T}]$ over
  $\Ring[T]$. Therefore, the mixed invariant over this new Frobenius
  algebra, $\MixedInvt{F}^{\Ring[\sqrt{T},X]/(X^2=T)}$, vanishes over any
  ring $\Ring$ (with $2$ invertible).  In particular, with
  $\Ring=\QQ$, we get $\MixedInvt{F}^{\QQ[\sqrt{T},X]/(X^2=T)}=0$.

  Now consider the Bar-Natan Frobenius algebra over the rationals,
  $\QQ[H,X]/(X^2=HX)$. This is \emph{twist-equivalent} to the above
  Frobenius algebra $\QQ[\sqrt{T},X]/(X^2=T)$, in the sense of
  Khovanov~\cite{Kho-kh-Frobenius}. Specifically, we have an
  isomorphism
  \[
    \begin{gathered}
      \phi\from \QQ[\sqrt{T},X]/(X^2=T)\to\QQ[H,X]/(X^2=HX)\\
      \phi(1)=1,\qquad\phi(X)=2X-H,\qquad\phi(\sqrt{T})=H
    \end{gathered}
  \]
  which preserves the algebra structure, and twists the counit $\eta$ and
  comultiplication $\Delta$ by the invertible element $2\in\QQ$:
  \[
    \eta(\phi(a))=2\phi(\eta(a))\qquad\Delta(\phi(a))=\tfrac{1}{2}\phi(\Delta(a))\qquad\forall
    a\in\QQ[\sqrt{T},X]/(X^2=T).
  \]
  Khovanov's proof of invariance under twist
  equivalence~\cite[Proposition~3]{Kho-kh-Frobenius} works also for
  $\BNh^-$ (respectively, $\BNh^+$), and shows that the $\BNh^-$
  (respectively, $\BNh^+$) Khovanov homologies over
  $\QQ[\sqrt{T},X]/(X^2=T)$ and $\QQ[H,X]/(X^2=HX)$ are
  isomorphic. Moreover, the proof can be modified to see that for both 
  versions, the maps induced by cobordisms agree over
  $\QQ[\sqrt{T},X]/(X^2=T)$ and $\QQ[H,X]/(X^2=HX)$ up to
  multiplication by (possibly negative) powers of $2$. Therefore,
  $\MixedInvt{F}^{\QQ[H,X]/(X^2=HX)}=2^k\MixedInvt{F}^{\QQ[\sqrt{T},X]/(X^2=T)}$
  for some integer $k$, and hence is zero.

  Finally, consider the Bar-Natan mixed invariant over the integers,
  $\MixedInvt{F}^{\ZZ[H,X]/(X^2=HX)}$. Recall its definition at the
  chain level. We choose an admissible cut and decompose the surface
  $F$ into two cobordisms $F_0\from\varnothing\to L$ and
  $F_1\from L\to\varnothing$, and choose movies $M_0$ and $M_1$
  describing $F_0$ and $F_1$. Consider the generator
  $1\in\BNcx^-(\varnothing)=\ZZ[H]$, and its image
  $\BNcx^-(F_0)(1)\in \BNcx^-(L)$ induced by the movie $M_0$. This is
  a boundary when viewed as a cycle in $\BNcx^\infty(L)$; choose a
  chain $c\in \BNcx^\infty(L)$ with $\bdy c=\BNcx^-(F_0)(1)$. Let
  $\bar{c}$ be the image of $c$ in $\BNcx^+(L)$ obtained by removing
  all the terms with non-negative powers of $H$, and consider its
  image
  $\BNcx^+(F_1)(\bar{c})\in \BNcx^+(\varnothing)=\ZZ[H^{-1},H]/\ZZ[H]$
  induced by the movie $M_1$. Since we assumed
  $e=-2$, $\chi-2s_o=-1$, and $s_n=0$, $\BNcx^+(F_1)(\bar{c})$ lies in bigrading
  $(0,2)$, and so gives an element of $\ZZ$ (which is the $\gr_q=2$
  part of $\ZZ[H^{-1},H]/\ZZ[H]$); this is the mixed invariant
  $\MixedInvt{F}^{\ZZ[H,X]/(X^2=HX)}(1)$.

  For any ring $\Ring$, if we tensor each step in the above chain-level description
  of the definition of $\MixedInvt{F}^{\ZZ[H,X]/(X^2=HX)}(1)$ with $\Ring$, we get
  a chain-level description of the definition of
  $\MixedInvt{F}^{\Ring[H,X]/(X^2=HX)}(1)$. So, the Bar-Natan mixed invariant
  $\MixedInvt{F}^{\Ring[H,X]/(X^2=HX)}(1)$, viewed as an element of
  $\Ring$ (which is the $\gr_q=2$ part of
  $\BNcx^+(\varnothing)=\Ring[H^{-1},H]/\Ring[H]$), can be obtained from
  the above element by tensoring with $\Ring$ over $\ZZ$. Since
  $\MixedInvt{F}^{\QQ[H,X]/(X^2=HX)}(1)=0\in\QQ$, we get
  $\MixedInvt{F}^{\ZZ[H,X]/(X^2=HX)}(1)=0\in\ZZ$, and therefore,
  $\MixedInvt{F}^{\Ring[H,X]/(X^2=HX)}(1)=0\in\Ring$ for all rings
  $\Ring$.
\end{proof}

\section{Computations, applications, and questions}\label{sec:comps}
Theorems~\ref{thm:vanishing} and~\ref{thm:cc-3-vanish} give many examples where the mixed
invariant vanishes. In this section, we use Lemma~\ref{lem:hat} to
give some examples where it does not vanish, and note some
corollaries.

\subsection{A first direct computation}\label{sec:direct-comp}
Let $M$ denote the obvious M\"obius band in $S^3$ with boundary the
(right-handed) trefoil $3_1$. View the boundary sum
$M\natural M\natural M$ as a cobordism from the empty link to
$3_1\#3_1\#3_1$; explicitly, $M\natural M\natural M$ is given by the
following movie:
\begin{center}
  \includegraphics[scale=.3333]{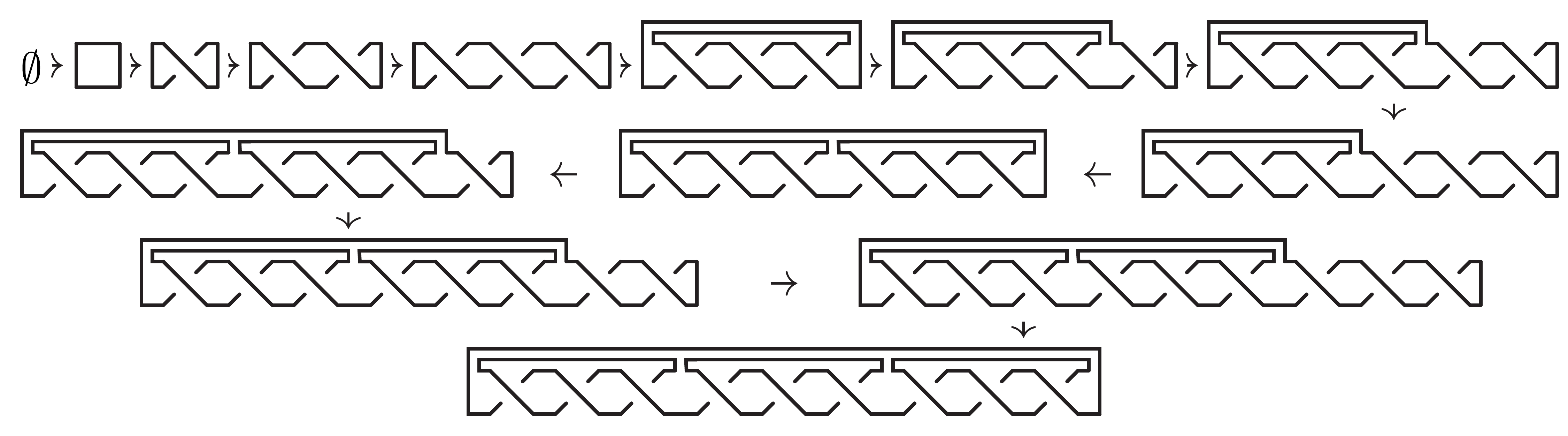}.
\end{center}

This movie corresponds to a cobordism with crosscap number $3$ and
normal Euler number $-18$ (see the proof of
Lemma~\ref{lem:bigrading-shift-saddle}, and perform one saddle at a
time). 
We compute directly that the mixed invariant, with respect to the Bar-Natan
deformation, is non-vanishing, and then observe that this also follows
from Lemma~\ref{lem:hat}. The frame $3_1$ in the movie above is an
admissible cut, decomposing the cobordism as $F_1\circ F_0$. The
normal Euler number of $F_1$ is $-6$, so the map
$\BNh^-(F_0)\co \ZZ[H]=\BNh^-(\emptyset)\to \BNh^-(3_1)$ shifts the
$(\gr_h,\gr_q)$-bigrading by $(3,9)$. The image of $\BNh^-(\emptyset)$
at each stage of the movie $F_0$ lies in the all-1 resolution, and a
generator labels each circle $1$:
\begin{equation}\label{eq:tref-eg-1}
  \mathcenter{\includegraphics[scale=.3]{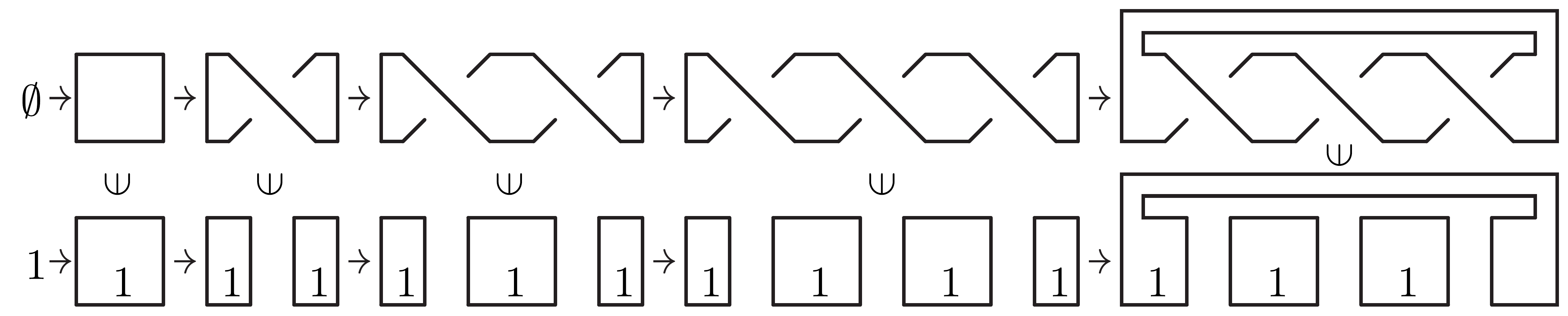}.}
\end{equation}

The element $\BNh^-(F_0)(1)\in\BNh^-(3_1)$ is non-zero, but its image
in $\BNh^\infty$ is zero: the element $\BNcx^-(F_0)(1)$ is the
boundary of the following element:
\begin{center}
  \includegraphics[scale=.3333]{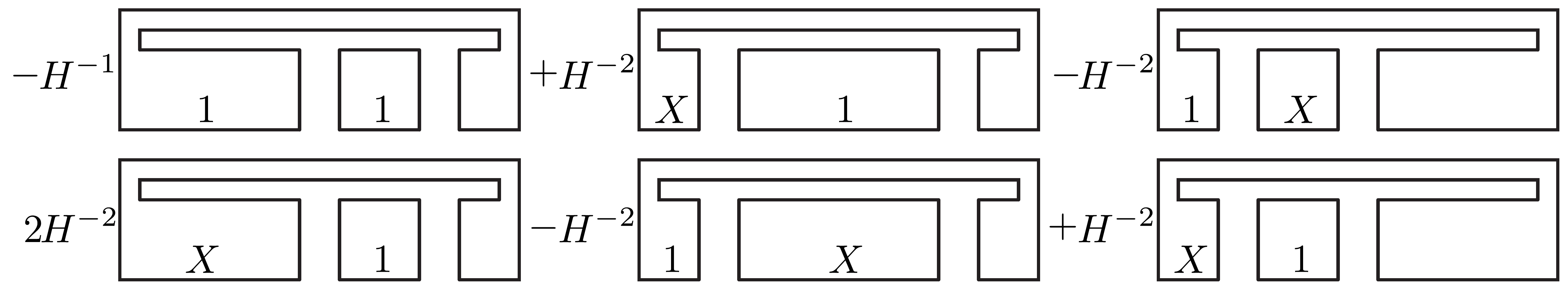}.
\end{center}
(For computing the signs, we have ordered the crossings from
left to right.)
In particular, $\BNh^-(F_0)(1)$ is an element of $\BNh^{\red}(3_1)$,
as expected. The element of $\BNcx^\infty(L)$ shown with boundary
$\BNcx^-(F_0)(1)$ lies in $\BNcx^+(L)$, so to compute the mixed
invariant, we apply $\BNh^+(F_1)$ to this element. The result is:
\begin{center}
  \includegraphics[scale=.3333]{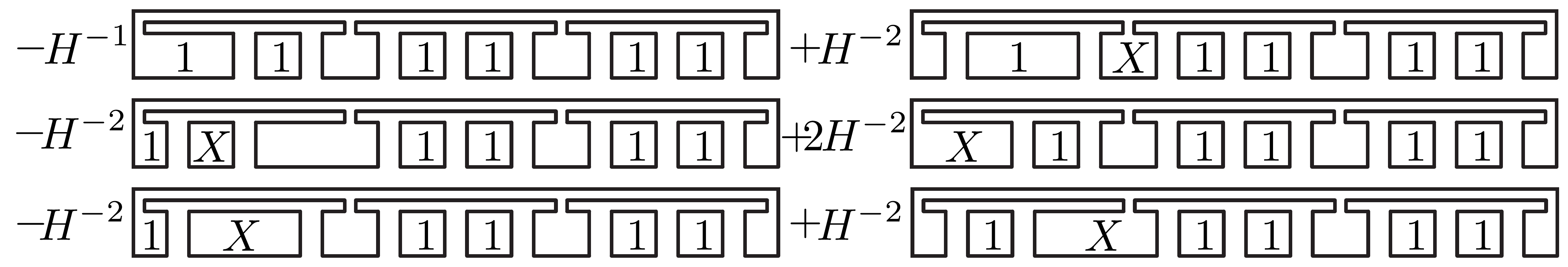}
\end{center}
We could compute directly that this is a nontrivial element of
$\BNh^+(L_1)$, but it is slightly easier to apply the connecting
homomorphism to $\wh{\BNh}(L_1)$. The image under the connecting
homomorphism has a term
\begin{equation}\label{eq:tref-eg-last}
  \mathcenter{\includegraphics[scale=.3333]{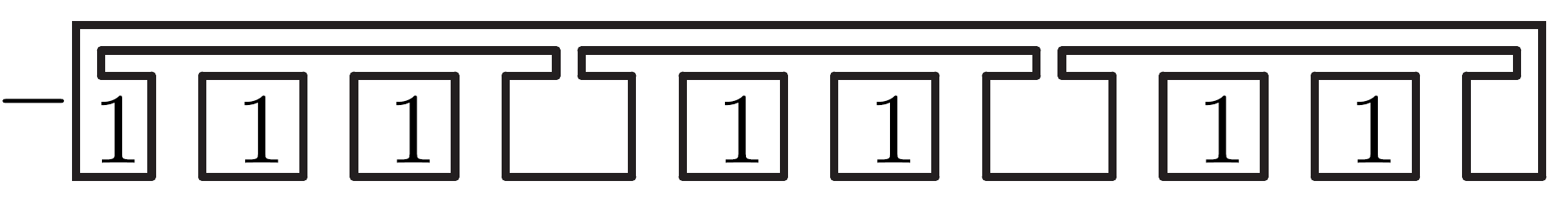}}
\end{equation}
which is does not appear in the boundary of any element of
$\wh{\BNcx}(L_1)$. So, the mixed invariant, in $\BNh^+(L_1)$, is
nontrivial.

We can obtain the same result slightly more easily using
Lemma~\ref{lem:hat}. By that lemma, it suffices to show that the image
of the class $1\in\wh{\BNh}(L_0)$, under $\wh{\BNh}(F)$, is
non-zero. A similar computation to Formula~(\ref{eq:tref-eg-1}) shows
that $\wh{\BNh}(F)(1)$ is the element in the all-1 resolution where
every circle is labeled $1$, i.e., the element shown in
Formula~(\ref{eq:tref-eg-last}). Since all maps into this resolution
are split maps, this is a nontrivial element of $\wh{\BNh}(L_1)$.

In particular, by Theorem~\ref{thm:vanishing}, this cobordism is not
obtained by taking the connected sum of a crosscap-number $2$
cobordism with $\RP^2$ or $\overline{\RP}^2$. The fact that the
cobordism does not split off a copy of $\overline{\RP}^2$ also follows
from the Gordon-Litherland formula~\cite{GL-top-signature}: if
$F=F'\#\overline{\RP}^2$ then $F'$ is a surface with $b_1(F')=2$,
$e(F')=-20$, and boundary $3_1\#3_1\#3_1$, so $\sigma(K)-e(F')/2=4$
but such a surface violates the inequality
$|\sigma(K)-e(F')/2|\leq b_1(F')$. This inequality seems not to
obstruct splitting off a copy of $\RP^2$. This computation also
provides a little more evidence that the 4-dimensional crosscap number
of $3_1\#3_1\#3_1$ is $3$, a conjecture which appears to be open.

By contrast, a similar direct computation to the above shows that the
mixed invariant associated to the mirror of this cobordism, from
$\emptyset$ to $m(3_1)\#m(3_1)\#m(3_1)$, vanishes. (Here, we mean a
different mirror from Section~\ref{sec:first-obs}: the map $(t,x,y,z,w)\mapsto (t,-x,y,z,w)$.)

\subsection{A more interesting example}\label{sec:SunSwann}
Sundberg-Swann showed that the map on Khovanov homology distinguishes
two slice disks for the knot $9_{46}$. As we will see, their proof actually gives
somewhat more: two distinct punctured $\RP^2\#\RP^2\#\RP^2$s
with boundary $3_1\# m(3_1)$ and normal Euler number $-6$.

\begin{figure}
  \centering
  \includegraphics[scale=.5]{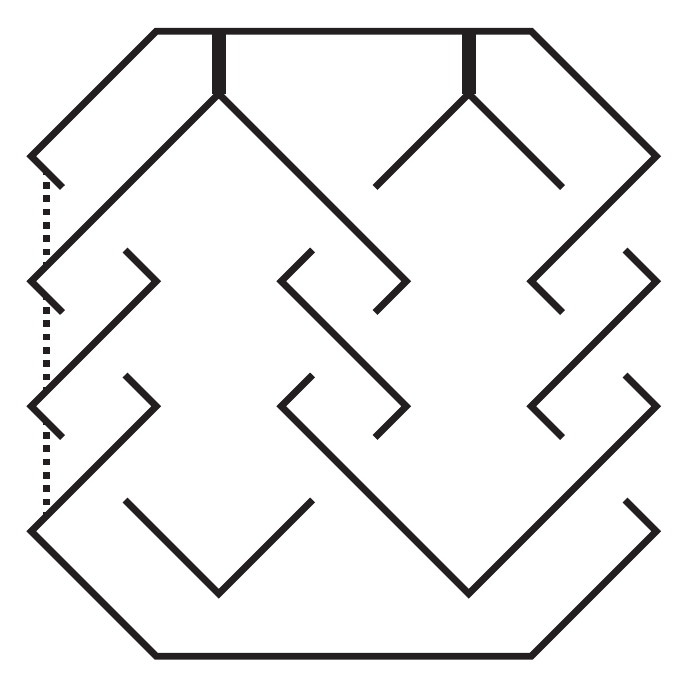}
  \caption{\textbf{The knot $9_{46}$.} The slice disks $\Sigma_L$ and
    $\Sigma_R$ are obtained by attaching a saddle at the
    left and right thick lines, respectively, to obtain a 2-component
    unlink. The cobordism to $3_1\#m(3_1)$ comes from attaching
    saddles at the three dotted lines.}
  \label{fig:946}
\end{figure}

Recall that the knot $9_{46}$ has two slice disks, corresponding to
attaching saddles at two of the handles shown in Figure~\ref{fig:946};
we will refer to these as the left and right slice disks $\Sigma_L$ and
$\Sigma_R$, respectively. We will view $\Sigma_L$ and $\Sigma_R$ as cobordisms from
$\emptyset$ to $9_{46}$. There is also a cobordism $C$ from $9_{46}$ to
$3_1\#m(3_1)$ with crosscap number $3$ and normal Euler number $-6$,
obtained by attaching three saddles to $9_{46}$; again, see
Figure~\ref{fig:946}. (Attaching just one of these saddles gives
$8_{20}$, and attaching two gives $6_1$.)

Sundberg-Swann call each of these three saddles a \emph{trim
  cobordism}. They show, by direct computation, that the map
$\widehat{\BNh}(C\circ \Sigma_L)=0$ while $\widehat{\BNh}(C\circ \Sigma_R)$
sends the generator $1\in\wh{\BNh}(\emptyset)$ to the class in
$\wh{\BNh}(3_1\#m(3_1))$ shown in Figure~\ref{fig:ss-image}
\cite[Proof of Theorem 6.3]{SS-kh-surf}. In particular, by
Proposition~\ref{prop:BNh-functorial}, the surfaces $C\circ \Sigma_L$ and
$C\circ \Sigma_R$ are not smoothly isotopic.

\begin{figure}
  \centering
  \includegraphics{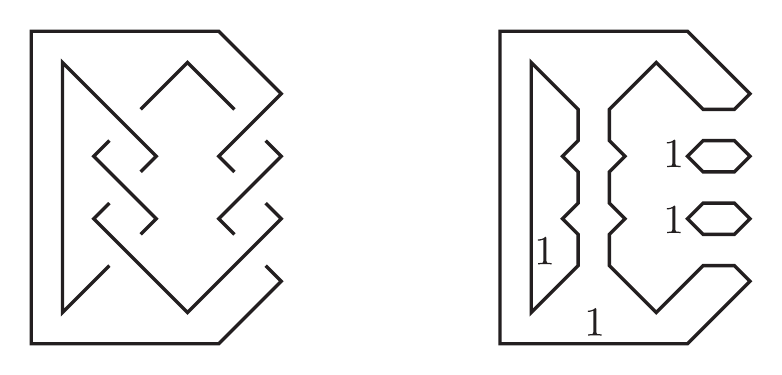}
  \caption{\textbf{The image of $\wh{\BNh}(C\circ \Sigma_R)$.} Left:
    the diagram for $3_1\#m(3_1)$ obtained by performing three
    saddle moves to $9_{46}$. Right: the element $\wh{\BNh}(C\circ
    \Sigma_R)$ lies in the all-$1$ resolution, and labels every circle
    by $1$.}
  \label{fig:ss-image}
\end{figure}

By Lemma~\ref{lem:hat}, the mixed invariant $\MixedInvt{C\circ \Sigma_R}$
is non-zero, as is $\bdy\circ\MixedInvt{C\circ \Sigma_R}$. By contrast,
$\bdy\circ\MixedInvt{C\circ \Sigma_L}=0$. The element $\MixedInvt{C\circ
  \Sigma_L}$ lies in bigrading $(2,7)$, and the generator in this
bigrading is not in the image of multiplication by $U$; see Figure~\ref{fig:compute}. So, from the
long exact sequence~(\ref{eq:les}), $\bdy$ is injective in this
bigrading, so $\MixedInvt{C\circ\Sigma_L}=0$, as well.

\begin{figure}
  \centering
\begin{tikzpicture}[xscale=.6, yscale=.6,every node/.style={inner sep=0,outer sep=0}]
  \draw[step=1, very thin] (0,0) grid (3,10);
  \draw[xstep=2, ystep=1, xshift=1cm, very thin] (2,0) grid (4,10);
  \draw[step=1, very thin] (5,0) grid (8,10);
  \draw (0.5,-.2) node[below] {$-3$};
  \draw (1.5,-.2) node[below] {$-2$};
  \draw (2.5,-.2) node[below] {$-1$};
  \draw (4,-.2) node[below] {$0$};
  \draw (5.5,-.2) node[below] {$1$};
  \draw (6.5,-.2) node[below] {$2$};
  \draw (7.5,-.2) node[below] {$3$};
  \draw (-.2,0.5) node[left] {$-7$};
  \draw (-.2,1.5) node[left] {$-5$};
  \draw (-.2,2.5) node[left] {$-3$};
  \draw (-.2,3.5) node[left] {$-1$};
  \draw (-.2,4.5) node[left] {$1$};
  \draw (-.2,5.5) node[left] {$3$};
  \draw (-.2,6.5) node[left] {$5$};
  \draw (-.2,7.5) node[left] {$7$};
  \draw (-.2,8.5) node[left] {$9$};
  \draw (-.2,9.5) node[left] {$11$};
  \node at (0.5, 0.5) (a) {$a$};
  \node at (1.5, 2.5) (b) {$b$};
  \node at (2.5, 2.5) (c) {$c$};
  \node at (3.5, 3.5) (d) {$d$};
  \node at (4.5,3.5) (e) {$\vphantom{d}e$};
  \node at (3.5, 4.5) (f) {$f\vphantom{g}$};
  \node at (4.5, 4.5) (g) {$\vphantom{f}g$};
  \node at (5.5, 5.5) (h) {$h$};
  \node at (6.5, 5.5) (i) {$i$};
  \node at (7.5, 7.5) (j) {$j$};
  \draw[->, thick] (a) to (b);
  \draw[->, thick] (c) to (f);
  \draw[->, thick] (e) to (h);
  \draw[->, thick] (i) to (j);
\end{tikzpicture}\qquad
\begin{tikzpicture}[xscale=1.1, yscale=.6, every node/.style={inner sep=0,outer sep=0}]
  \draw[step=1, very thin] (0,0) grid (3,10);
  \draw[xstep=2, ystep=1, xshift=1cm, very thin] (2,0) grid (4,10);
  \draw[step=1, very thin] (5,0) grid (8,10);
  \draw (0.5,-.2) node[below] {$-3$};
  \draw (1.5,-.2) node[below] {$-2$};
  \draw (2.5,-.2) node[below] {$-1$};
  \draw (4,-.2) node[below] {$0$};
  \draw (5.5,-.2) node[below] {$1$};
  \draw (6.5,-.2) node[below] {$2$};
  \draw (7.5,-.2) node[below] {$3$};
  \draw (-.2,0.5) node[left] {$-7$};
  \draw (-.2,1.5) node[left] {$-5$};
  \draw (-.2,2.5) node[left] {$-3$};
  \draw (-.2,3.5) node[left] {$-1$};
  \draw (-.2,4.5) node[left] {$1$};
  \draw (-.2,5.5) node[left] {$3$};
  \draw (-.2,6.5) node[left] {$5$};
  \draw (-.2,7.5) node[left] {$7$};
  \draw (-.2,8.5) node[left] {$9$};
  \draw (-.2,9.5) node[left] {$11$};
  \node at (0.5, 0.5) (a) {$a$};
  \node at (0.5, 2.5) (am1) {$T^{-1}\!a$};
  \node at (0.5, 4.5) (am2) {$T^{-2}\!a$};
  \node at (0.5, 6.5) (am3) {$T^{-3}\!a$};
  \node at (0.5, 8.5) (am4) {$T^{-4}\!a$};
  \node at (1.5, 2.5) (b) {$b$};
  \node at (1.5, 0.5) (b1) {$Tb$};
  \node at (1.5, 4.5) (bm1) {$T^{-1}\!b$};
  \node at (1.5, 6.5) (bm2) {$T^{-2}\!b$};
    \node at (1.5, 8.5) (bm3) {$T^{-3}\!b$};
    \node at (2.5, 2.5) (c) {$c$};
    \node at (2.5, 0.5) (c1) {$Tc$};
    \node at (2.5, 4.5) (cm1) {$T^{-1}\!c$};
    \node at (2.5, 6.5) (cm2) {$T^{-2}\!c$};
    \node at (2.5, 8.5) (cm3) {$T^{-3}\!c$};
    \node at (3.5, 3.5) (d) {$d$};
    \node at (3.5, 1.5) (d1) {$Td$};
    \node at (3.5, 5.5) (dm1) {$T^{-1}\!d$};
    \node at (3.5, 7.5) (dm2) {$T^{-2}\!d$};
    \node at (3.5, 9.5) (dm3) {$T^{-3}\!d$};
    \node at (4.5,3.5) (e) {$\vphantom{d}e$};
    \node at (4.5,1.5) (e1) {$Te$};
    \node at (4.5,5.5) (em1) {$T^{-1}\!e$};
    \node at (4.5,7.5) (em2) {$T^{-2}\!e$};
    \node at (4.5,9.5) (em3) {$T^{-3}\!e$};
    \node at (3.5, 4.5) (f) {$f\vphantom{g}$};
    \node at (3.5, 2.5) (f1) {$Tf$};
    \node at (3.5, 0.5) (f2) {$T^2\!f$};
    \node at (3.5, 6.5) (fm1) {$T^{-1}\!f$};
    \node at (3.5, 8.5) (fm2) {$T^{-2}\!f$};
    \node at (4.5, 4.5) (g) {$\vphantom{f}g$};
    \node at (4.5, 2.5) (g1) {$Tg$};
    \node at (4.5, 0.5) (g2) {$T^2\!g$};
    \node at (4.5, 6.5) (gm1) {$T^{-1}g$};
    \node at (4.5, 8.5) (gm2) {$T^{-2}g$};
    \node at (5.5, 5.5) (h) {$h$};
    \node at (5.5, 3.5) (h1) {$Th$};
    \node at (5.5, 1.5) (h2) {$T^2\!h$};
    \node at (5.5, 7.5) (hm1) {$T^{-1}\!h$};
    \node at (5.5, 9.5) (hm2) {$T^{-2}\!h$};
    \node at (6.5, 5.5) (i) {$i$};
    \node at (6.5, 3.5) (i1) {$Ti$};
    \node at (6.5, 1.5) (i2) {$T^2\!i$};
    \node at (6.5, 7.5) (im1) {$\mathbf{T^{-1}\!i}$};
    \node at (6.5, 9.5) (im2) {$T^{-2}\!i$};
    \node at (7.5, 7.5) (j) {$j$};
    \node at (7.5, 5.5) (j1) {$Tj$};
    \node at (7.5, 3.5) (j2) {$T^{2}\!j$};
    \node at (7.5, 1.5) (j3) {$T^3\!j$};
    \node at (7.5, 9.5) (jm1) {$T^{-1}\!j$};
    \draw[->] (a) to (b1);
    \draw[->] (am1) to (b);
    \draw[->] (am2) to (bm1);
    \draw[->] (am3) to (bm2);
    \draw[->] (am4) to (bm3);
    \draw[->] (c) to (f1);
    \draw[->] (c1) to (f2);
    \draw[->] (cm1) to (f);
    \draw[->] (cm2) to (fm1);
    \draw[->] (cm3) to (fm2);
    \draw[->] (e) to (h1);
    \draw[->] (e1) to (h2);
    \draw[->] (em1) to (h);
    \draw[->] (em2) to (hm1);
    \draw[->] (em3) to (hm2);
    \draw[->] (i) to (j1);
    \draw[->] (i1) to (j2);
    \draw[->] (i2) to (j3);
    \draw[->] (im1) to (j);
    \draw[->] (im2) to (jm1);
    \draw[very thick] (0,1)--(1,1)--(1,3)--(3,3)--(3,5)--(5,5)--(5,6)--(7,6)--(7,8)--(8,8);
\end{tikzpicture}
  \caption{\textbf{Computing $\BNh^+(3_1\# m(3_1))$.} Left:
    $\wh{\BNh}$ as computed by knotkit~\cite{KKI-kh-knotkit}, with
    $\QQ$ coefficients. Each letter is a basis element over $\QQ$. (This computation can also be deduced from
    $\BNh(3_1)$, with a little work.) The arrows are the differentials
    in the Lee spectral sequence, which one can deduce from knowing
    that the Lee homology is $\QQ\oplus\QQ$ in bidegrees $(0,\pm1)$,
    since $s(3_1\#m(3_1))=0$. Right: the differentials on
    $\BNh^\infty(3_1\#m(3_1))$, which one can read off from the
    top-left computation. The subcomplex $\BNcx^-$ lies below the thick
    steps, and the quotient complex $\BNcx^+$ lies above the thick steps. The generator of $\BNh^+$ in bigrading
    $(2,7)$ is in bold; $U^{-1}$ times this generator is not a cycle
    in $\BNcx^+$. The analogous computation for the Bar-Natan
    deformation is similar, but the differential on the left has
    bi-degree $(1,2)$ instead of $(1,4)$, and the variable $H$ has
    bi-degree $(0,-2)$.}
  \label{fig:compute}
\end{figure}

In conclusion, both the map $\wh{\BNh}$ and the the mixed invariant
distinguish this pair of surfaces. By Theorem~\ref{thm:vanishing},
this implies that $C\circ\Sigma_L$ and $C\circ\Sigma_R$ do not differ
by taking a connected sum with a smoothly embedded
$2$-sphere. Further, non-vanishing of $\wh{\BNh}(C\circ \Sigma_R)$
implies that $C\circ \Sigma_R$ is not obtained from another connected
surface by attaching a $1$-handle, and is not a crosscap
stabilization. Hence, we have proved Theorem~\ref{thm:intro}.

\begin{remark}
  Using one of the three dotted saddles in Figure~\ref{fig:946} gives
  a pair of M\"obius bands with boundary $8_{20}$ distinguished by
  Khovanov homology, and using two of them gives a pair of punctured
  Klein bottles with boundary $6_1$ distinguished by Khovanov
  homology.
\end{remark}

\subsection{An exotic pair of surfaces}\label{sec:exotic}
Recall that a pair of surfaces $F,F'\subset B^4$ with boundary
$K\subset S^3$ is \emph{exotic} if there is a homeomorphism
$\phi\co B^4\to B^4$ so that $\phi|_{S^3}=\Id$ and $\phi(F)=F'$, but
no diffeomorphism with these properties. (See also
Remark~\ref{rem:exotic-exotic}.)  Hayden-Sundberg give a family of
exotic pairs of surfaces~\cite{HS-kh-exotic}. The simplest of their
pairs is the pair of slice disks shown in Figure~\ref{fig:exotic} for
the knot $J$. The slice disk $D$ (respectively $D'$) is obtained by
attaching a saddle along the arc $b$ (respectively $b'$) shown there,
and then capping the resulting 2-component unlink with disks. The fact
that these surfaces are distinct is witnessed by the map on Khovanov
homology: for the element $\phi$ of $\wh{\BNh}(J)$ shown in
Figure~\ref{fig:exotic}, $\wh{\BNh}(D')(\phi)=0$ (obvious) but
$\wh{\BNh}(D)(\phi)=1\in\ZZ=\wh{\BNh}(\emptyset)$~\cite[Figure
4]{HS-kh-exotic}.

\begin{figure}
  \centering
  \includegraphics{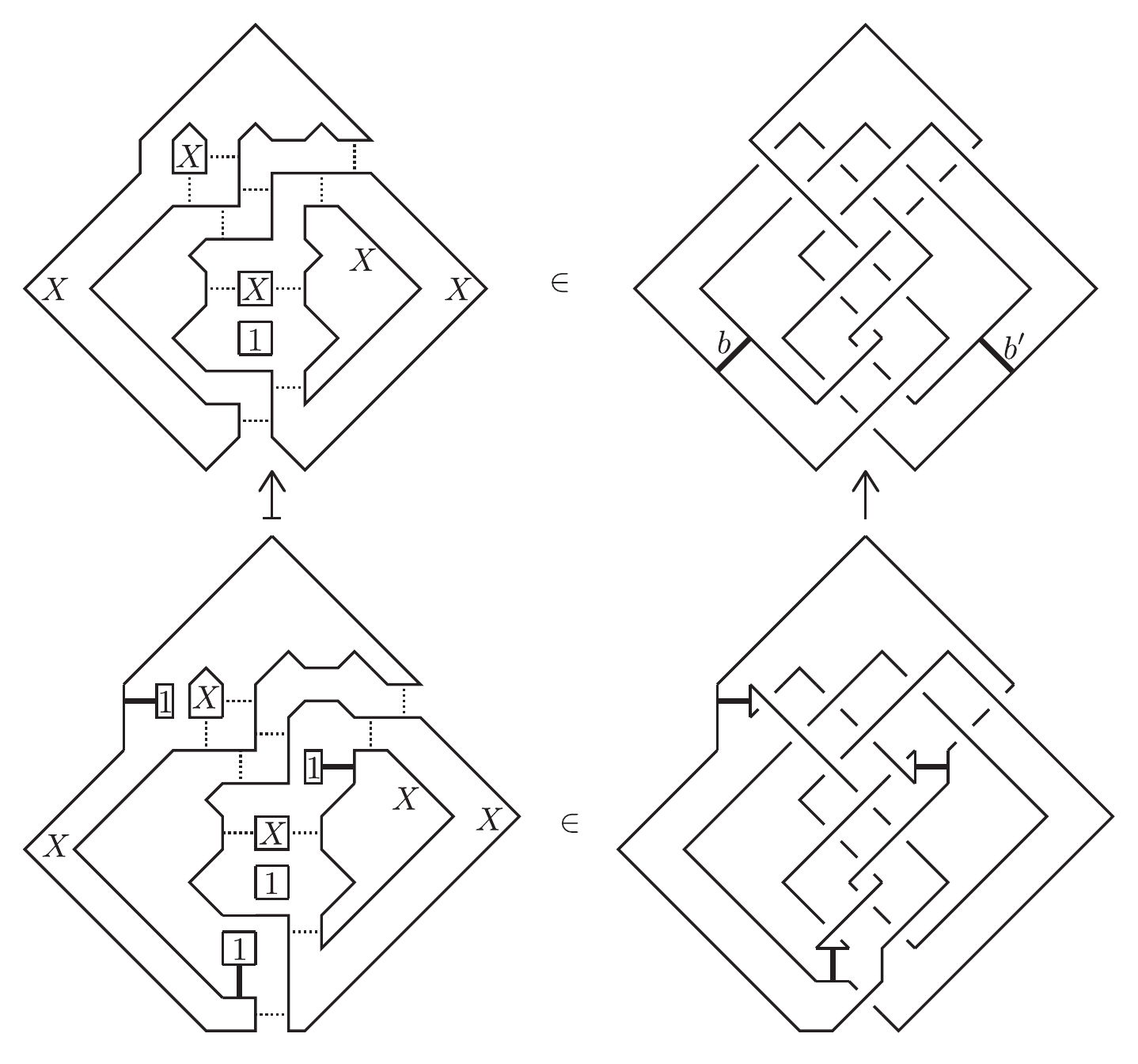}
  \caption{\textbf{An exotic pair.} Top right: Hayden-Sundberg's knot $J$ and
    the two saddles $b$ and $b'$ giving an exotic pair of distinct slice disks
    for $J$, drawn as thick line segments. Top-left: Hayden-Sundberg's
    cycle $\phi$ in $\wh{\BNcx}(J)$ witnessing the
    fact that these slice disks are distinct. Bottom-right: a diagram for
    $12^n_{309}$ and three saddles giving a nonorientable cobordism to
    $J$. Bottom-left: a cycle $\psi$ in $\wh{\BNcx}(12^n_{309})$ which maps to
    Hayden-Sundberg's cycle. As in Hayden-Sundberg's figure, dotted lines
    indicate which crossings had $0$-resolutions.}
  \label{fig:exotic}
\end{figure}

There is a cobordism $C$ with crosscap number $3$ and normal Euler
number $-6$ from the knot $12^n_{309}$ to $J$, so that
$\phi=\wh{\BNh}(\psi)$ for an appropriate class
$\psi\in\wh{\BNh}(12^n_{309})$: see Figure~\ref{fig:exotic}. So,
$\wh{\BNh}(D\circ C)(\psi)=1$ while $\wh{\BNh}(D'\circ
C)(\psi)=0$. Thus, $D\circ C$ is not diffeomorphic to $D'\circ C$ rel
boundary. On the other hand, since $D$ is homeomorphic to $D'$ rel
boundary, $D\circ C$ is homeomorphic to $D'\circ C$ rel
boundary. Thus, we have proved Theorem~\ref{thm:intro-2}.

By the second case of Lemma~\ref{lem:hat}, the mixed invariant also
distinguishes $D\circ C$ and $D'\circ C$ (compare
Remark~\ref{rem:mixed-strong}). The fact that the pair
$D\circ C$ and $D'\circ C$ are not diffeomorphic is, of course, slightly stronger than
the statement that $D$ and $D'$ are not diffeomorphic.

\begin{remark}
  Hayden-Sundberg's example also immediately gives an exotic pair of
  crosscap number $3$ surfaces with boundary $12^n_{404}$, as well as
  exotic pairs of crosscap number $3$ surfaces with boundary on
  several links, by an easy adaptation of Figure~\ref{fig:exotic}.
\end{remark}

\begin{remark}\label{rem:exotic-exotic}
  Hayden-Sundberg take a slightly different definition of
  \emph{exotic} than we have: they define a pair of surfaces
  $F,F'\subset B^4$ to be exotic if there is an ambient isotopy
  through homeomorphisms taking $F$ to $F'$ but no ambient isotopy
  through diffeomorphisms (which are, in both cases, the identity on
  $S^3$). Since their surfaces are distinguished by the map on
  Khovanov homology, however, by
  Proposition~\ref{prop:BNh-functorial}, their computation shows that
  there is no diffeomorphism from $B^4$ to itself which is the
  identity on the boundary and takes $F$ to $F'$ (even one which is
  not isotopic to the identity). That is, their pairs of surfaces
  really are exotic in the sense described above.
\end{remark}

\subsection{Some questions}
To put the results above in context, and in particular to acknowledge
the cases they do not cover, we conclude with some open questions.

In Corollary~\ref{cor:closed-vanish}, we showed that the mixed invariant
does not distinguish closed, connected surfaces. By
Proposition~\ref{prop:or-gens-part2} for the nonorientable case and
work of Gujral-Levine~\cite{LG-kh-split} for the orientable case, the
map on $\BNh^\bullet$ also does not distinguish disconnected surfaces.
\begin{question}
  Is there a pair $F,F'\subset S^4$ of closed, disconnected surfaces
  with the same topology and componentwise normal Euler numbers so
  that $\MixedInvt{F}\neq \MixedInvt{F'}$?
\end{question}
If we are only considering surfaces without stars, by
Theorem~\ref{thm:cc-3-vanish}, these surfaces would have to have
(total) normal Euler number $-2$ and Euler characteristic
$1$. For example, perhaps $F$ could be a knotted copy of
$\RP^2\amalg (\RP^2\# \overline{\RP}^2)$ with total normal Euler
number $-2$ and non-vanishing mixed invariant, and $F'$ the standard
$\RP^2\amalg (\RP^2\# \overline{\RP}^2)$, which has vanishing mixed
invariant.

The Seiberg-Witten invariant 
is not just defined when $b_2^+\geq 3$, but also when $b_2^+=2$.
As noted in Remark~\ref{rem:2-crosscaps}, we can define a
Khovanov mixed invariant when the crosscap number is $2$, but we do
not know if it is well-defined.
\begin{question}
  If $F,F'$ are isotopic surfaces (rel boundary) with crosscap number
  $2$ and
  admissible cuts $(S,V,\phi)$ and $(S',V',\phi')$, respectively, is
  the mixed invariant of $F$ with respect to $(S,V,\phi)$ equal to the
  mixed invariant of $F'$ with respect to $(S',V',\phi')$?
\end{question}

In the examples in Sections~\ref{sec:SunSwann} and~\ref{sec:exotic}
of surfaces distinguished by their mixed invariants, the surfaces were
also distinguished by the induced maps on ordinary Khovanov homology
$\wh{\BNh}$. By Remark~\ref{rem:mixed-strong}, for surfaces with
connected boundary, the mixed invariant is at least as strong as the
map on $\wh{\BNh}$.
\begin{question}
  Is there a pair of surfaces with crosscap number $\geq 3$ and the
  same topology and normal Euler number which are distinguished by the
  mixed invariant but not by the map on $\wh{\BNh}$? Is there such a
  pair not distinguished by the homotopy class of maps on $\BNcx^-$? Is there an exotic
  pair of surfaces with this property?
\end{question}
Lemma~\ref{lem:hat} and its proof give restrictions on what the
Khovanov homology of such a pair must look like, as does
Corollary~\ref{cor:H-vanish}.

On a related point, by Theorem~\ref{thm:vanishing}, the map on
$\BNh^\bullet$ vanishes for stabilizations of surfaces, so never
distinguishes them.  There is one case in which the mixed invariant
could potentially distinguish stabilized surfaces:
\begin{question}
  Is there an exotic pair of M\"obius bands $F,F'\subset B^4$ so that
  the mixed invariant distinguishes their stabilizations? That is, if
  $F\# T^2$ denotes a standard stabilization of $F$, is there an exotic pair
  of M\"obius bands $F,F'$ with boundary some knot $K$ so that
  $\MixedInvt{F\#T^2}\neq \MixedInvt{F'\#T^2}\in\BNh^+(K)$?
\end{question}

The examples of nonorientable surfaces in Sections~\ref{sec:SunSwann}
and~\ref{sec:exotic} came from pairs of slice disks. Indeed, the
nonorientable surfaces were apparent from the slice disks and class
in Khovanov homology. Perhaps this phenomenon is general:
\begin{question}
  Is it true that for every exotic pair of slice disks $D,D'$, for any
  knot $K$, there is a nonorientable cobordism $F$ from $K$ to
  another knot $K'$ so that $F\circ D$ and $F\circ D'$ are also an
  exotic pair? Can $F$ be chosen to have crosscap number $\geq 3$?
\end{question}
One can ask the same question, but for exotic pairs detected by
Khovanov homology:
\begin{question}
  Is it true that for every pair of slice disks $D,D'$, for a knot $K$
  such that $\wh{\BNh}(D)\neq \wh{\BNh}(D')$, there is a
  nonorientable cobordism $F$ from $K$ to another knot $K'$ so that
  $\wh{\BNh}(F\circ D)\neq \wh{\BNh}(F\circ D')$?
\end{question}
One could also replace $\wh{\BNh}$ by $\BNh^\bullet$ in the question, or
require that $F$ have crosscap number $\geq 3$ and ask if
$\MixedInvt{F\circ D}\neq \MixedInvt{F\circ D'}$.

As noted in the introduction, our mixed invariant is inspired by
Ozsv\'ath-Szab\'o's mixed invariant in Heegaard Floer homology.
\begin{question}
  Is there a precise relationship between the Khovanov mixed invariant
  of a surface $F$ and the Heegaard Floer mixed invariant of the
  branched double cover of $F$?
\end{question}
Note that having crosscap number $\geq 3$ does not give an inequality
for $b_2^+$
(consider the standard
$\overline{\RP}^2\#\overline{\RP}^2\#\overline{\RP}^2$), nor does
having $b_2^+(\Sigma(F))\geq 2$ imply crosscap number $\geq 3$ (consider
$\RP^2\#\RP^2$ or, for that matter, an orientable surface of genus
$g\geq 2$), so the two invariants are not
defined in exactly the same cases; perhaps this argues against a
direct relationship.

\vspace{-0.3cm}
\bibliographystyle{myalpha}
\bibliography{newbibfile}

\newcommand{\etalchar}[1]{$^{#1}$}
\providecommand{\bysame}{\leavevmode\hbox to3em{\hrulefill}\thinspace}
\providecommand{\MR}{\relax\ifhmode\unskip\space\fi MR }
\providecommand{\MRhref}[2]{%
  \href{http://www.ams.org/mathscinet-getitem?mr=#1}{#2}
}
\providecommand{\href}[2]{#2}
\begin{thebibliography}{CMW09}

\bibitem[Ali19]{Ali-kh-unknotting}
Akram Alishahi, \emph{Unknotting number and {K}hovanov homology}, Pacific J.
  Math. \textbf{301} (2019), no.~1, 15--29.

\bibitem[AN06]{AN-top-Klein-tqft}
A.~Alexeevski and S.~Natanzon, \emph{Noncommutative two-dimensional topological
  field theories and {H}urwitz numbers for real algebraic curves}, Selecta
  Math. (N.S.) \textbf{12} (2006), no.~3-4, 307--377.

\bibitem[Bal]{Bal-kh-E1}
William Ballinger, \emph{Concordance invariants from the {$E(-1)$} spectral
  sequence on {K}hovanov homology}, arXiv:2004.10807.

\bibitem[Bar05]{Bar-kh-tangle-cob}
Dror Bar-Natan, \emph{Khovanov's homology for tangles and cobordisms}, Geom.
  Topol. \textbf{9} (2005), 1443--1499.

\bibitem[BNM06]{BNM-kh-degeneration}
Dror Bar-Natan and Scott Morrison, \emph{The {K}aroubi envelope and {L}ee's
  degeneration of {K}hovanov homology}, Algebr. Geom. Topol. \textbf{6} (2006),
  1459--1469.

\bibitem[BS16]{BS-top-stabilizations}
R.~\.{I}nan\c{c} Baykur and Nathan Sunukjian, \emph{Knotted surfaces in
  4-manifolds and stabilizations}, J. Topol. \textbf{9} (2016), no.~1,
  215--231.

\bibitem[CMW09]{CMW-kh-functoriality}
David Clark, Scott Morrison, and Kevin Walker, \emph{Fixing the functoriality
  of {K}hovanov homology}, Geom. Topol. \textbf{13} (2009), no.~3, 1499--1582.

\bibitem[CS93]{CS-knot-movie}
J.~Scott Carter and Masahico Saito, \emph{Reidemeister moves for surface
  isotopies and their interpretation as moves to movies}, J. Knot Theory
  Ramifications \textbf{2} (1993), no.~3, 251--284.

\bibitem[FKV87]{FKV-top-bulletin}
S.~M. Finashin, M.~Kreck, and O.~Ya. Viro, \emph{Exotic knottings of surfaces
  in the {$4$}-sphere}, Bull. Amer. Math. Soc. (N.S.) \textbf{17} (1987),
  no.~2, 287--290.

\bibitem[FKV88]{FKV-top-knot-surf}
\bysame, \emph{Nondiffeomorphic but homeomorphic knottings of surfaces in the
  {$4$}-sphere}, Topology and geometry---{R}ohlin {S}eminar, Lecture Notes in
  Math., vol. 1346, Springer, Berlin, 1988, pp.~157--198.

\bibitem[GL]{LG-kh-split}
Onkar~Singh Gujral and Adam~Simon Levine, \emph{{K}hovanov homology and
  cobordisms between split links}, arXiv:2009.03406.

\bibitem[GL78]{GL-top-signature}
C.~McA. Gordon and R.~A. Litherland, \emph{On the signature of a link}, Invent.
  Math. \textbf{47} (1978), no.~1, 53--69.

\bibitem[Hay]{Hay-kh-disks}
Kyle Hayden, \emph{Exotically knotted disks and complex curves},
  arXiv:2003.13681.

\bibitem[HKK{\etalchar{+}}]{HKKMPS-kh-Brun}
Kyle Hayden, Alexandra Kjuchukova, Siddhi Krishna, Maggie Miller, Mark Powell,
  and Nathan Sunukjian, \emph{{B}runnian exotic surface links in the 4-ball},
  arXiv:2106.13776.

\bibitem[HN13]{HN-kh-detects}
Matthew Hedden and Yi~Ni, \emph{Khovanov module and the detection of unlinks},
  Geom. Topol. \textbf{17} (2013), no.~5, 3027--3076.

\bibitem[HS]{HS-kh-exotic}
Kyle Hayden and Isaac Sundberg, \emph{{K}hovanov homology and exotic surfaces
  in the 4-ball}, arXiv:2108.04810.

\bibitem[Jac04]{Jac-kh-cobordisms}
Magnus Jacobsson, \emph{An invariant of link cobordisms from {K}hovanov
  homology}, Algebr. Geom. Topol. \textbf{4} (2004), 1211--1251 (electronic).

\bibitem[JMZ21]{MJZ-hf-exotic}
Andr\'{a}s Juh\'{a}sz, Maggie Miller, and Ian Zemke, \emph{Transverse
  invariants and exotic surfaces in the 4-ball}, Geom. Topol. \textbf{25}
  (2021), no.~6, 2963--3012.

\bibitem[JZa]{JZ-hf-clasp}
Andr{\'a}s Juh{\'a}sz and Ian Zemke, \emph{New {H}eegaard {F}loer slice genus
  and clasp number bounds}, arXiv:2007.07106.

\bibitem[JZb]{JZ-hf-stab}
\bysame, \emph{Stabilization distance bounds from link {F}loer homology},
  arXiv:1810.09158.

\bibitem[Kho00]{Kho-kh-categorification}
Mikhail Khovanov, \emph{A categorification of the {J}ones polynomial}, Duke
  Math. J. \textbf{101} (2000), no.~3, 359--426.

\bibitem[Kho06a]{Kho-kh-cobordism}
\bysame, \emph{An invariant of tangle cobordisms}, Trans. Amer. Math. Soc.
  \textbf{358} (2006), no.~1, 315--327.

\bibitem[Kho06b]{Kho-kh-Frobenius}
\bysame, \emph{Link homology and {F}robenius extensions}, Fund. Math.
  \textbf{190} (2006), 179--190.

\bibitem[Kor02]{Kork-top-mcg}
Mustafa Korkmaz, \emph{Mapping class groups of nonorientable surfaces}, Geom.
  Dedicata \textbf{89} (2002), 109--133.

\bibitem[KR22]{KhRo-kh-Frob2}
Mikhail Khovanov and Louis-Hadrien Robert, \emph{Link homology and {F}robenius
  extensions {II}}, Fund. Math. \textbf{256} (2022), no.~1, 1--46.

\bibitem[Lee05]{Lee-kh-endomorphism}
Eun~Soo Lee, \emph{An endomorphism of the {K}hovanov invariant}, Adv. Math.
  \textbf{197} (2005), no.~2, 554--586.

\bibitem[LZ19]{LV-kh-ribbon}
Adam~Simon Levine and Ian Zemke, \emph{Khovanov homology and ribbon
  concordances}, Bull. Lond. Math. Soc. \textbf{51} (2019), no.~6, 1099--1103.

\bibitem[Mel77]{Melvin-top-thesis}
Paul~Michael Melvin, \emph{Blowing up and down in 4-manifolds}, ProQuest LLC,
  Ann Arbor, MI, 1977, Thesis (Ph.D.)--University of California, Berkeley.

\bibitem[Miy86]{Miy-86-stab}
Katura Miyazaki, \emph{On the relationship among unknotting number, knotting
  genus and {A}lexander invariant for {$2$}-knots}, Kobe J. Math. \textbf{3}
  (1986), no.~1, 77--85.

\bibitem[MP19]{MP-top-stab}
Allison~N. Miller and Mark Powell, \emph{Stabilization distance between
  surfaces}, Enseign. Math. \textbf{65} (2019), no.~3-4, 397--440.

\bibitem[MWW]{MWW-kh-blobs}
Scott Morrison, Kevin Walker, and Paul Wedrich, \emph{Invariants of 4-manifolds
  from {K}hovanov-{R}ozansky link homology}, arXiv:1907.12194.

\bibitem[OSS17]{OSS-hf-unoriented}
Peter~S. Ozsv\'{a}th, Andr\'{a}s~I. Stipsicz, and Zolt\'{a}n Szab\'{o},
  \emph{Unoriented knot {F}loer homology and the unoriented four-ball genus},
  Int. Math. Res. Not. IMRN (2017), no.~17, 5137--5181.

\bibitem[OSz04]{OSz-hf-3manifolds}
Peter Ozsv{\'a}th and Zolt{\'a}n Szab{\'o}, \emph{Holomorphic disks and
  topological invariants for closed three-manifolds}, Ann. of Math. (2)
  \textbf{159} (2004), no.~3, 1027--1158.

\bibitem[OSz05]{OSz-hf-branched}
\bysame, \emph{On the {H}eegaard {F}loer homology of branched double-covers},
  Adv. Math. \textbf{194} (2005), no.~1, 1--33.

\bibitem[OSz06]{OSz-hf-4manifolds}
\bysame, \emph{Holomorphic triangles and invariants for smooth four-manifolds},
  Adv. Math. \textbf{202} (2006), no.~2, 326--400.

\bibitem[Pie16]{Pie-top-nonor}
Alex Pieloch, \emph{Curve complexes of non-orientable surfaces},
  \url{https://lsa.umich.edu/content/dam/math-assets/math-document/reu-documents/Pieloch_REUPaper.pdf},
  2016.

\bibitem[Ras]{Rasmussen-kh-closed}
Jacob Rasmussen, \emph{Khovanov's invariant for closed surfaces},
  arXiv:math/0502527.

\bibitem[Ras10]{Ras-kh-slice}
\bysame, \emph{Khovanov homology and the slice genus}, Invent. Math.
  \textbf{182} (2010), no.~2, 419--447.

\bibitem[Sar20]{Sar-ribbon}
Sucharit Sarkar, \emph{Ribbon distance and {K}hovanov homology}, Algebr. Geom.
  Topol. \textbf{20} (2020), no.~2, 1041--1058.

\bibitem[See]{KKI-kh-knotkit}
Cotton Seed, \emph{Knotkit}, \url{https://github.com/cseed/knotkit}.

\bibitem[SS]{SS-kh-surf}
Isaac Sundberg and Jonah Swann, \emph{Relative {K}hovanov-{J}acobsson classes},
  arXiv:2103.01438.

\bibitem[Tan06]{Tanaka-kh-closed}
Kokoro Tanaka, \emph{Khovanov-{J}acobsson numbers and invariants of
  surface-knots derived from {B}ar-{N}atan's theory}, Proc. Amer. Math. Soc.
  \textbf{134} (2006), no.~12, 3685--3689.

\bibitem[TT06]{TT-top-un-tqft}
Vladimir Turaev and Paul Turner, \emph{Unoriented topological quantum field
  theory and link homology}, Algebr. Geom. Topol. \textbf{6} (2006),
  1069--1093.

\bibitem[Tur20]{Tur-kh-diag}
Paul Turner, \emph{Khovanov homology and diagonalizable {F}robenius algebras},
  J. Knot Theory Ramifications \textbf{29} (2020), no.~1, 1950095, 10.

\end{thebibliography}
\vspace{1cm}
\end{document}